\documentclass[11pt,a4paper,reqno]{amsart} 
\usepackage{amsmath}
\usepackage{amssymb}
\usepackage{euscript}
\usepackage{graphicx}
\usepackage[all]{xy}
\usepackage[dvipsnames]{xcolor}
\usepackage[colorlinks=true,linktocpage=true,pagebackref=false, citecolor=black,linkcolor=black]{hyperref}
\usepackage{pifont}
\usepackage{tikz,pgfplots}
\pgfplotsset{compat=1.10}
\usepgfplotslibrary{fillbetween}
\usetikzlibrary{cd,arrows,decorations.pathmorphing,backgrounds,automata,positioning,fit,matrix}
\usepackage{caption,subcaption}
\usepackage{comment}
\usepackage{mathabx}
\usepackage{upgreek}

\pgfplotsset{soldot/.style={color=black,only marks,mark=*}} \pgfplotsset{holdot/.style={color=black,fill=white,only marks,mark=*}}
\pgfplotsset{bluedot/.style={color=blue,fill=blue,only marks,mark=*}}

\newtheorem{thm}{Theorem}[section]

\newtheorem{lem}[thm]{Lemma}

\newtheorem{cor}[thm]{Corollary}

\newtheorem{fact}[thm]{Fact}

\theoremstyle{definition}
\newtheorem{defn}[thm]{Definition}

\theoremstyle{remark}
\newtheorem{remark}[thm]{Remark}
\newtheorem{remarks}[thm]{Remarks}
\newtheorem{example}[thm]{Example}
\newtheorem{examples}[thm]{Examples}

\numberwithin{equation}{section}
\numberwithin{figure}{section}

 \newcommand{\N}{{\mathbb N}}
 \newcommand{\R}{{\mathbb R}}
 \newcommand{\C}{{\mathbb C}}

\newcommand{\sph}{{\mathbb S}} 
 \newcommand{\PP}{{\mathbb P}}

\newcommand{\an}{{\rm an}}

 \newcommand{\Cont}{{\mathcal C}}

\newcommand{\Ff}{{\EuScript F}}

\newcommand{\Ss}{{\EuScript S}}
\newcommand{\Tt}{{\EuScript T}}

\newcommand{\Bb}{{\EuScript B}}
\newcommand{\Cc}{{\EuScript C}}

\newcommand{\Qq}{{\EuScript Q}}

\newcommand{\Aa}{{\EuScript A}}
\newcommand{\Mm}{{\EuScript M}}
\newcommand{\Rr}{{\EuScript R}}

\newcommand{\Ww}{{\EuScript W}}
\newcommand{\Oo}{{\EuScript O}}

\newcommand{\Reg}{\operatorname{Reg}}
\newcommand{\Sing}{\operatorname{Sing}}
\newcommand{\Int}{\operatorname{Int}}
\newcommand{\cl}{\operatorname{Cl}}
\newcommand{\dist}{\operatorname{dist}}
\newcommand{\id}{\operatorname{id}}

\newcommand{\zar}{\operatorname{zar}}

\newcommand{\im}{\operatorname{im}}
\newcommand{\Sth}{\operatorname{Sth}}

\newcommand{\x}{{\tt x}} \newcommand{\y}{{\tt y}} 
\newcommand{\z}{{\tt z}} 
 
\newcommand{\w}{{\tt w}} 

\newcommand{\veps}{\varepsilon}
\newcommand{\ol}{\overline}

\usepackage{scalerel}

\usepackage{colortbl}

\definecolor{silver}{rgb}{0.808,0.808,0.808}
\definecolor{grey}{rgb}{0.9,0.9,0.9}
\newcolumntype{a}{>{\columncolor{grey}}c}
\newcolumntype{b}{>{\columncolor{silver}}c}

\begin{document}

\title[Nash uniformization of chessboard sets]{Nash uniformization of chessboard sets\\ by Nash manifolds with corners}

\author{Antonio Carbone}
\address{Dipartimento di Scienze dell'Ambiente e della Prevenzione, Palazzo Turchi di Bagno, C.so Ercole I D'Este, 32, Università di Ferrara, 44121 Ferrara (ITALY)}
\email{antonio.carbone@unife.it}
\thanks{The first author is supported by GNSAGA of INDAM}

\author{Jos\'e F. Fernando}
\address{Departamento de \'Algebra, Geometr\'\i a y Topolog\'\i a, Facultad de Ciencias Matem\'aticas, Universidad Complutense de Madrid, Plaza de Ciencias 3, 28040 MADRID (SPAIN)}
\email{josefer@mat.ucm.es}
\thanks{The second author is supported by Spanish STRANO PID2021-122752NB-I00. This article has been developed during several one month research stays of the second author in the Dipartimento di Matematica of the Universit\`a di Trento. The second author would like to thank the department for the invitations and the very pleasant working conditions. This article contains a part of the results of the Ph.D. Thesis of the first author written under the supervision of the second author.
}

\date{22/07/2026}

\begin{abstract}
Bierstone and Parusi\'nski studied the desingularization of $d$-dimensional closed subanalytic sets and in particular of $d$-dimensional closed semialgebraic sets. Their main tools are Hironaka's desingularization of real algebraic sets (to `uniform' the Zariski closure of the closed semialgebraic set) and Hironaka's embedded desingularization of real algebraic subsets of non-singular real algebraic sets (to `uniform' afterwards the Zariski closure of the boundary of the uniformed closed semialgebraic set). Both procedures preserve the number of $d$-dimensional components connected by analytic paths of the involved closed semialgebraic sets, so they have a good behavior for pure dimensional closed semialgebraic sets. If the involved $d$-dimensional closed semialgebraic set is not pure dimensional, some components connected by analytic paths of smaller dimension then $d$ could be dropped during the desingularization process. For instance, classical Whitney's umbrella $W:=\{\y^2\z-\x^2=0\}\subset\R^3$ has two componentes connected by analytic paths (one of dimension $2$ and the other of dimension $1$), whereas its desingularization is (biregularly equivalent to) the plane, which has only one component connected by analytic paths, that has dimension $2$.

The obtained models in the desingularization process, that we call in the following {\em closed chessboard sets}, are the closures of (finite) unions of connected components of the complements of normal-crossings divisors of non-singular real algebraic sets. The local models for $d$-dimensional chessboard sets are unions of (standard) closed orthants of $\R^d$, that is, $\bigcup_{(\veps_1,\ldots,\veps_d)\in{\mathfrak F}}\{\veps_1\x_1\geq0,\ldots,\veps_d\x_d\geq0\}\subset\R^d$ for some set ${\mathfrak F}\subset\{-1,1\}^d$.

In this work we study the Nash uniformization of $d$-dimensional closed chessboard sets $\Ss$ using {\em Nash manifolds with corners} $\Qq$ with the same number of connected components as $\Ss$ (or equivalently the same number of irreducible components). Nash manifolds with corners are closed chessboard sets whose local models are either $\R^d$ or semialgebraic sets of the type $\{\x_1\geq0,\ldots,\x_k\geq0\}$ for some $1\leq k\leq d$. More generally, a {\em chessboard set} is a semialgebraic set in between a finite union of connected components of the complement of a normal-crossings divisor of non-singular real algebraic set and its closure. We also provide a Nash uniformization result for general chessboard sets $\Ss$ using {\em Nash quasi-manifolds with corners} $\Qq^\bullet$ with the same number of connected components as $\Ss$ (or equivalently the same number of irreducible components). The Nash quasi-manifold with corners $\Qq^\bullet$ is obtained from a Nash manifold with corners $\Qq$ after erasing some of the `faces' of its boundary. The difficult point in both results is to preserve the number of connected components in the uniformization process.

As an application of the previous results together with Bierstone and Parusi\'nski's desingularization of pure dimensional closed semialgebraic sets we obtain Nash uniformization results for general semialgebraic sets that preserve components connected by analytic paths. In addition the previous results allow to prove that each Nash manifold with corners admit a (semialgebraic) compactification that is as well a Nash manifold with corners.
\end{abstract}

\maketitle

\section{Introduction}\label{s1}

Hironaka's resolution of singularities \cite{hi} of an algebraic variety (over a field of characteristic $0$) is a widespread celebrated discipline that has many applications in many areas of Mathematics \cite{ko,li}. It has been developed by many other authors in the analytic and subanalytic cases (we refer to Abhyankar \cite{a1,a2,a3}, Bierstone-Milman \cite{bm1,bm2,bm3}, Villamayor \cite{vi}, Zariski \cite{z1,z2,z3} as an example). For simplicity we restrict to the case when the ground field is $\R$ and we recall that a {\em biregular diffeomorphism between two constructible sets $S\subset\R^m$ and $T\subset\R^n$} is a bijective regular map $f:S\to T$ whose inverse $f^{-1}:T\to S$ is also a regular map. The general approach consists of the following: {\em Given a real algebraic set $X\subset\R^n$, one finds a non-singular real algebraic set $X'\subset\R^m$ together with a proper polynomial map $f:X'\to X$ that is a biregular diffeomorphism outside the set of singular points of $X$}. We recall the precise statement:

\begin{thm}[Hironaka's desingularization]\label{hi1}
Let $X\subset\R^n$ be a real algebraic set. There exist a non-singular real algebraic set $X'\subset\R^m$ and a proper polynomial map $f:X'\to X$ such that the restriction $f|_{X'\setminus f^{-1}(\Sing(X))}:X'\setminus f^{-1}(\Sing(X))\to X\setminus\Sing(X)$ is a biregular diffeomorphism.
\end{thm}

\begin{remark}
If $X$ is pure dimensional, $X\setminus\Sing(X)$ is dense in $X$ and $f$ is surjective.
\hfill$\sqbullet$
\end{remark}

A related important construction (which will also be very useful for our purposes) is Hironaka's embedded resolution of singularities \cite{hi}, which involves the concept of normal-crossings divisor. Let $X\subset Y\subset\R^n$ be real algebraic sets such that $Y$ is non-singular and has dimension $d$. Recall that $X$ is a \em normal-crossings divisor of $Y$ \em if the irreducible components of $X$ are non-singular, have codimension $1$ in $Y$ and constitute a transversal family. More precisely, for each point $x\in X$ there exists a regular system of parameters $\x_1,\ldots,\x_d$ for $Y$ at $x$ such that $X$ is given on an open Zariski neighborhood of $x$ in $Y$ by the equation $\x_1\cdots\x_k=0$ for some $k\leq d$. 

\begin{thm}[Hironaka's embedded desingularization]\label{hi2}
Let $X\subsetneq Y\subset\R^n$ be real algebraic sets such that $Y$ is non-singular. Then there exist a non-singular real algebraic set $Y'\subset\R^m$ and a proper surjective polynomial map $g:Y'\to Y$ such that $g^{-1}(X)$ is a normal-crossings divisor of $Y'$ and the restriction $g|_{Y'\setminus g^{-1}(X)}:Y'\setminus g^{-1}(X)\to Y\setminus X$ is a biregular diffeomorphism.
\end{thm}

\subsection{Semialgebraic setting}
A set $\Ss\subset\R^n$ is \em semialgebraic \em when it admits a description in terms of a finite boolean combination of polynomial equalities and inequalities, which we will call a \em semialgebraic description\em. The category of semialgebraic sets is closed under basic boolean operations but also under usual topological operations: taking closures (denoted by $\cl(\cdot)$), interiors (denoted by $\Int(\cdot)$), connected components, etc. If $\Ss\subset\R^n$ is a semialgebraic set, the set $\Ss_{(k)}$ of points of $\Ss$ of (local) dimension $k$ is also a semialgebraic set for each $k\geq0$. In addition, if $\Ss$ has dimension $d$, the set $\Ss_{(d)}$ is a closed subset of $\Ss$. 

We denote $\ol{\Ss}^{\zar}$ the Zariski closure of a semialgebraic set $\Ss\subset\R^n$. A $d$-dimensional semialgebraic set $\Ss\subset\R^n$ is a {\em chessboard set} if its Zariski closure $\ol{\Ss}^{\zar}$ is a pure dimensional non-singular real algebraic set of dimension $d$ and there exist a normal-crossings divisor $Z\subset\ol{\Ss}^{\zar}$ and connected components $\Cc_1,\ldots,\Cc_s$ of $\ol{\Ss}^{\zar}\setminus Z$ such that $\Ss$ is a semialgebraic set between $\bigcup_{i=1}^s\Cc_i$ and its closure. Observe that for each $x\in\Ss$ there exist an open neighborhood $U^x\subset\ol{\Ss}^{\zar}$ endowed with a Nash diffeomorphism $u:U^x\to\R^d$ and a subset ${\mathfrak F}_x\subset\{-1,1\}^d$ such that 
$$
\bigcup_{(\veps_1,\ldots,\veps_d)\in{\mathfrak F}_x}\{\veps_1\x_1>0,\ldots,\veps_d\x_d>0\}\subset u(\Ss\cap U_x)\subset\bigcup_{(\veps_1,\ldots,\veps_d)\in{\mathfrak F}_x}\{\veps_1\x_1\geq0,\ldots,\veps_d\x_d\geq0\}.
$$
In fact, by \cite[Thm.1.6]{bfr} we can cover $\ol{\Ss}^{\zar}$ by finitely many Nash charts of the previous type. We recall in \S\ref{orthants} that the set $\bigcup_{(\veps_1,\ldots,\veps_d)\in{\mathfrak F}}\{\veps_1\x_1\geq0,\ldots,\veps_d\x_d\geq0\}$ where ${\mathfrak F}\subset\{-1,1\}^d$ is a $\Cont^1$ manifold with corners if and only if it is equal to either $\R^d$ or $\{\delta_{i_1}\x_{i_1}\geq0,\ldots,\delta_{i_r}\x_{i_r}\geq0\}$ for some $1\leq i_1<\cdots<i_r\leq d$ and $(\delta_{i_1},\ldots,\delta_{i_r})\in\{-1,1\}^r$. As orthants are in addition semialgebraic sets, the same type of result holds in the Nash category.

If $\Ss\subset\R^m$ and $\Tt\subset\R^n$ are semialgebraic sets, a map $f:\Ss\to\Tt$ is \em semialgebraic \em if its graph is a semialgebraic set. Two relevant types of semialgebraic maps $f:\Ss\to\Tt$ are restrictions to $\Ss$ of {\em polynomial maps} $f:=(f_1,\ldots,f_n):\R^m\to\R^n$ (where each $f_k\in\R[\x]:=\R[\x_1,\ldots,\x_n]$) whose images are contained in $\Tt$ and restrictions to $\Ss$ of {\em rational maps} $f:=(\frac{g_1}{h_1},\ldots,\frac{g_n}{h_n}):\R^m\dasharrow\R^n$ (where each $g_k,h_k\in\R[\x]:=\R[\x_1,\ldots,\x_n]$ and each $h_k\neq0$) whose images are contained in $\Tt$. In case $\Ss\cap\{h_k=0\}=\varnothing$ for each $k=1,\ldots,n$ we say $f|_\Ss$ is a {\em regular map on $\Ss$}.

A {\em Nash map} on an open semialgebraic set $U\subset\R^n$ is a semialgebraic smooth map $f:U\to\R^m$. Along this article {\em smooth} means $\Cont^\infty$. Given a semialgebraic set $\Ss\subset\R^n$, a \em Nash map on $\Ss$ \em is the restriction to $\Ss$ of a Nash map $F:U\to\R^m$ on an open semialgebraic neighborhood $U\subset\R^n$ of $\Ss$. We denote with ${\mathcal N}(\Ss)$ the ring of Nash functions on $\Ss$ and following \cite{fg1} we say that the semialgebraic set $\Ss$ is {\em irreducible} if ${\mathcal N}(\Ss)$ is an integral domain. In \cite[\S4]{fg1} we prove that each semialgebraic set $\Ss$ can be decomposed uniquely as a finite union of irreducible semialgebraic sets $\Ss_1,\ldots,\Ss_r$ such that each $\Ss_i$ is a maximal irreducible semialgebraic subset of $\Ss$ with respect to the inclusion. The semialgebraic sets $\Ss_1,\ldots,\Ss_r$ are called the {\em irreducible components of $\Ss$}. The irreducible components $\Ss_1,\ldots,\Ss_r$ are closed semialgebraic subsets of $\Ss$. If $\Ss$ is irreducible (as a semialgebraic set), its Zariski closure $\ol{\Ss}^{\zar}$ is irreducible (as a algebraic set), whereas the converse is not true in general. If $\Ss$ is a chessboard set, the irreducible components of $\Ss$ are by \cite[Prop.3.3]{fg1} its connected components.

A semialgebraic set $\Ss\subset\R^n$ is {\em connected by analytic paths} if for each pair of points $x,y\in\Ss$ there exists an analytic path $\alpha:[0,1]\to\Ss$ such that $\alpha(0)=x$ and $\alpha(0)=y$. By \cite[Thm.9.2]{fe3}, each semialgebraic set $\Ss$ can be decomposed uniquely as a finite union of semialgebraic sets $\Tt_1,\ldots,\Tt_\ell$ connected by analytic paths such that each $\Tt_i$ is a maximal semialgebraic subset of $\Ss$ connected by analytic paths with respect to the inclusion. The semialgebraic sets $\Tt_1,\ldots,\Tt_\ell$ are called the {\em components of $\Ss$ connected by analytic paths} and they are closed subsets of $\Ss$. If $\Ss$ is pure dimensional of dimension $d$, all the components $\Tt_j$ of $\Ss$ connected by analytic paths have by \cite[Prop.9.1.8]{bcr} and \cite[Thm.9.2]{fe3} pure dimension $d$. 

A component $\Tt_j$ connected by analytic paths of $\Ss$ is irreducible (as a semialgebraic set), so it is contained in at least one of the irreducible component $\Ss_i$ of $\Ss$ (maybe of larger dimension). Each irreducible component $\Ss_i$ of $\Ss$ is a (finite) union of components connected by analytic paths of $\Ss$, that is, the components connected by analytic paths of $\Ss_i$ are also components connected by analytic paths of $\Ss$. If $\Ss$ is a chessboard set, the components of $\Ss$ connected by analytic paths coincide by Lemma \ref{conccccap} with its connected components and by \cite[Prop.3.3]{fg1} with its irreducible components.

A {\em Nash subset} $X\subset M$ of a {\em Nash manifold} $M\subset\R^n$ (that is, a semialgebraic set that is a smooth submanifold of $\R^n$) is the zero set of a Nash function $f:M\to\R$, whereas the {\em Nash closure in $M$} of a semialgebraic set $\Ss\subset M$ is the smallest Nash subset $X$ of $M$ that contains $\Ss$. In this setting, a {\em Nash normal-crossings divisor} of a Nash manifold $M$ is a Nash set $X\subset M$ whose Nash irreducible components are Nash submanifolds of codimension $1$ in $M$ and constitute a transversal family. We denote the set of interior points of a semialgebraic subset $\Ss\subset M$ with $\Int_M(\Ss)$, which is again a semialgebraic set.

\subsection{Desingularization of closed semialgebraic sets using closed chessboard sets.}

In \cite{bp} Bierstone and Parusi\'nski developed the desingularization of closed semialgebraic sets using closed chessboard sets. 

\begin{thm}[Desingularization of closed semialgebraic sets, {\cite[Thm.1.1., Rmks.2.3 \& 2.6]{bp}}]\label{bpd}
Let $\Ss\subset\R^m$ be a $d$-dimensional closed semialgebraic set and let $\ol{\Ss}^{\zar}$ be its Zariski closure. Then there exist: 
\begin{itemize}
\item[{\rm(i)}] a pure dimensional non-singular real algebraic set $X\subset\R^n$ of dimension $d$, 
\item[\rm{(ii)}] a polynomial map $f:\R^n\to\R^m$ such that the restriction $f|_X:X\to\ol{\Ss}^{\zar}$ is proper, 
\item[\rm{(iii)}] an algebraic set $Z\subset\ol{\Ss}^{\zar}$ of dimension strictly smaller than $d$ such that $Y:=f^{-1}(Z)$ is a normal-crossings divisor of $X$,
\item[\rm{(v)}] a union $V$ of connected components of the difference $X\setminus Y$,
\end{itemize}
and they satisfy that the restriction $f|_{X\setminus f^{-1}(Z)}:X\setminus f^{-1}(Z)\to\ol{\Ss}^{\zar}\setminus Z$ is a biregular diffeomorphism and $f(\cl(V))$ is the set $\Ss_{(d)}$ of points of $\Ss$ of dimension $d$. 
\end{thm}
\begin{remarks}\label{example}
(i) Observe that $\cl(V)$ above is a closed chessboard set and the number of connected components of $\cl(V)$ coincides with the number of components connected by analytic paths of $\Ss_{(d)}$ (use \cite[Lem.7.16]{fe3} and Lemma \ref{conccccap}).

(ii) Consider the closed chessboard set $\Ss:=\R^2\setminus\{\y_1+\y_2>0,\y_1-\y_2>0\}$ (which is the complement of Nash manifold with one corner, see Figure \ref{scoppiamnento}) and the blow-up with center the origin $f:\R^2\to\R^2,\ (x_1,x_2)\mapsto(x_1,x_1x_2)$. We have $f^{-1}(\Ss)=\R^2\setminus f^{-1}(\{\y_1+\y_2>0,\y_1-\y_2>0\})=\R^2\setminus\{\x_1(1+\x_2)>0,\x_1(1-\x_2)>0\}$ (which is the complement of Nash manifold with two corners). Thus, usual desingularization techniques are no more useful to achieve our goal of uniformizing a checkerboard set by means of a Nash manifold with corners.\hfill$\sqbullet$
\end{remarks}

\begin{figure}[!ht]\begin{center}
\begin{tikzpicture}[scale=0.75]

\draw[fill=gray!50,opacity=0.4,draw=none] (0,0)--(0,5) -- (2.5,2.5) -- (5,5) -- (5,0) -- (0,0);

\draw[line width=1pt] (0,5) -- (2.5,2.5)--(5,5);
\draw[line width=1pt, dashed] (0,0) -- (2.5,2.5)--(5,0);

\draw (0.5,1.5) node {\small$\Ss$};

\draw[fill=gray!50,opacity=0.4,draw=none] (8,2)--(8,5) -- (10.5,5) -- (10.5,3) -- (13,3) -- (13,0)--(10.5,0)--(10.5,2)--(8,2);

\draw[line width=1pt] (8,2) -- (10.5,2);
\draw[line width=1pt, dashed] (10.5,2) -- (13,2);
\draw[line width=1pt] (10.5,5) -- (10.5,3);
\draw[line width=1pt] (10.5,2) -- (10.5,0);
\draw[line width=1pt, dashed] (10.5,3) -- (10.5,2);

\draw[line width=1pt,dashed] (8,3) -- (10.5,3);
\draw[line width=1pt] (10.5,3) -- (13,3);

\draw (8.5,1.5) node{\small $f^{-1}(\Ss)$};

\end{tikzpicture}
\end{center}
\caption{\small{The closed chessboard sets $\Ss$ (left) and $f^{-1}(\Ss)$ (right).\label{scoppiamnento}
}}
\end{figure}

\subsection{Nash uniformization of closed chessboard sets by Nash manifolds with corners.}
A {\em Nash manifold with corners} is a semialgebraic set that is a smooth submanifold with corners of $\R^n$. A Nash manifold with corners $\Qq\subset\R^n$ is contained, as a closed (semialgebraic) subset, in a Nash manifold $M\subset\R^n$ of its same dimension \cite[Prop.1.2]{fgr}. In fact, we restrict our scope to Nash manifolds $\Qq\subset\R^n$ with corners such that the Nash closure in $M$ of the boundary $\partial\Qq$ is a (Nash) normal-crossings divisor of $M$ (maybe after shrinking $M$). This property implies by \cite[Lem.C.2]{fe3} that each Nash manifold with corners admits up to Nash diffeomorphism a structure of closed chessboard set.

\begin{thm}[Nash uniformization of closed chessboard sets]\label{red2}
Let $\Ss\subset\R^m$ be a $d$-dimensional closed chessboard set and let $\ol{\Ss}^{\zar}$ be its Zariski closure. Then there exist: 
\begin{itemize}
\item[{\rm(i)}] A pure dimensional non-singular real algebraic set $X\subset\R^n$ of dimension $d$ and a normal-crossings divisor $Y\subset X$. 
\item[{\rm(ii)}] A $d$-dimensional Nash manifold with corners $\Qq\subset X$ (which is a closed subset of $X$) whose boundary $\partial\Qq$ has $Y$ as its Zariski closure. In addition, both $\Qq$ and $\Ss$ have the same number of connected components (or equivalently the same number of irreducible components).
\item[\rm{(iii)}] A polynomial map $f:\R^n\to\R^m$ such that $f(X)=\ol{\Ss}^{\zar}$, $f(\Qq)=\Ss$ and the restrictions $f|_X:X\to\ol{\Ss}^{\zar}$ and $f|_{\Qq}:\Qq\to\Ss$ are proper. 
\item[\rm{(iv)}] A closed semialgebraic set $\Rr:=\Ss\setminus\Int_{\ol{\Ss}^{\zar}}(\Ss)$ of dimension strictly smaller than $d$ such that $\Ss\setminus\Rr$ and $\Qq\setminus f^{-1}(\Rr)$ are Nash manifolds of dimension $d$ and the polynomial map $f|_{\Qq\setminus f^{-1}(\Rr)}:\Qq\setminus f^{-1}(\Rr)\to\Ss\setminus\Rr$ is a Nash diffeomorphism. 
\end{itemize}
In particular, if $\ol{\Ss}^{\zar}$ is compact, also $X$ is compact.
\end{thm}
\begin{remarks}
(i) In the previous statement we have imposed that the number of connected components of $\Ss$ and $\Qq$ coincide because of its relation with \cite[Main Thm.1.4]{fe3} and \cite[Thm. 1.14 \& 1.15]{cf1} and to avoid trivial approaches as the one we present next. 

(ii) Let $X:=\R^n$ and let $Z$ be a normal-crossings divisor of $X$ whose irreducible components are hyperplanes. The closure of each connected component of $X\setminus Z$ are Nash manifolds with corners. Let $\Cc_1,\ldots,\Cc_r$ be finitely many of such connected components of $X\setminus Z$ such that $\Ss:=\bigcup_{k=1}^r\cl(\Cc_k)$ is connected. Define 
$$
Y:=\bigsqcup_{k=1}^r(X\times\{k\})\subset\bigsqcup_{k=1}^r(\R^n\times\{k\})\subset\R^n\times\R=\R^{n+1},
$$
which is a non-singular real algebraic set, and $\Qq:=\bigsqcup_{k=1}^r(\cl(\Cc_k)\times\{k\})$, which is a Nash manifold with corners with $r$ connected components (instead of one connected components as $\Ss$). The projection $\pi:Y\to X$ is a proper polynomial map and its restriction $\pi|_\Qq:\Qq\to\Ss$ is also proper. Define $\Rr:=\bigcup_{k=1}^r(\cl(\Cc_k)\setminus\Cc_k)$, which is a semialgebraic set of dimension $d-1$. The restriction $\pi|_{\Qq\setminus\pi^{-1}(\Rr)}:\Qq\setminus\pi^{-1}(\Rr)\to\Ss\setminus\Rr$ is a Nash diffeomorphism.
\hfill$\sqbullet$
\end{remarks}

\subsection{Nash uniformization of chessboard sets by Nash quasi-manifolds with corners.}
As Nash manifolds with corners are locally compact, their images under proper maps (with values in locally compact Hausdorff topological spaces as $f|_X:X\to\ol{\Ss}^{\zar}$) are also locally compact, so Theorem \ref{red2} cannot be extended directly to general semialgebraic sets. In order to solve this, we will use {\em Nash quasi-manifolds with corners}, which are essentially Nash manifolds with corners with some `faces' erased. More precisely, we construct a natural semialgebraic partition ${\mathfrak S}(\Qq)$ of any Nash manifold with corners $\Qq\subset\R^n$ that takes into account the structure of its boundary (Definition \ref{sap}). A semialgebraic set $\Tt\subset\R^n$ is a {\em Nash quasi-manifold with corners} if $\Qq:=\cl(\Tt)$ is a Nash manifold with corners and $\Qq\setminus\Tt$ is a union of elements of ${\mathfrak S}(\Qq)$. In this case, it is necessary to substitute (proper) polynomial maps by (proper) Nash maps. We explain in Example \ref{truffa} why this additional change is mandatory.

\begin{thm}[Nash uniformization of chessboard sets]\label{red4}
Let $\Ss\subset\R^m$ be a $d$-dimensional chessboard set. Then there exist: 
\begin{itemize}
\item[{\rm(i)}] A pure dimensional compact non-singular real algebraic set $X\subset\R^n$ of dimension $d$ and a normal-crossings divisor $Y\subset X$. 
\item[{\rm(ii)}] A $d$-dimensional Nash quasi-manifold with corners $\Ss^\bullet\subset X$ whose closure in $X$ is a compact Nash manifold with corners $\Qq^\bullet\subset X$ and whose boundary $\partial\Qq^\bullet$ has $Y$ as its Zariski closure. In addition, the number of connected components of $\Ss$, $\Ss^\bullet$ and $\Qq^\bullet$ coincide and the number of irreducible components of $\Ss$, $\Ss^\bullet$ and $\Qq^\bullet$ also coincide.
\item[\rm{(iii)}] A Nash map $f:\R^n\to\R^m$ such that $f(\Ss^\bullet)=\Ss$ and the restriction $f|_{\Ss^\bullet}:\Ss^\bullet\to\Ss$ is proper. 
\item[\rm{(iv)}] A closed semialgebraic subset $\Rr:=\Ss\setminus\Int_{\ol{\Ss}^{\zar}}(\Ss)$ of dimension strictly smaller than $d$ such that $\Ss\setminus\Rr$ and $\Ss^\bullet\setminus f^{-1}(\Rr)$ are Nash manifolds and the Nash map $f|_{\Ss^\bullet\setminus f^{-1}(\Rr)}:\Ss^\bullet\setminus f^{-1}(\Rr)\to\Ss\setminus\Rr$ is a Nash diffeomorphism. 
\end{itemize}
\end{thm}

\subsubsection{Drilling blow-up}

We have seen in Remark \ref{example}(ii) that usual blow-ups may not be useful to prove Theorems \ref{red2} and \ref{red4} and we need further tools: the {\em drilling blow-up} of a Nash manifold $M$ with center a closed Nash submanifold $N$, see \cite{fe3}. We refer the reader to \cite{ds, hi2} for the oriented blow-up of a real analytic space with center a closed subspace, which is the counterpart of the construction in \cite{fe3} for the real analytic setting. In \cite[\S5]{hpv} appears a presentation of the oriented blow-up in the analytic case closer to the drilling blow-up described in \cite{fe3}. There the authors consider the case of the oriented blow-up of a real analytic manifold $M$ with center a closed real analytic submanifold $N$ whose vanishing ideal inside $M$ is finitely generated (this happens for instance if $N$ is compact). In \cite[\S3]{fe2} a similar construction is presented in the semialgebraic setting, which is used to `appropriately embed' semialgebraic sets in affine spaces. In \cite[\S3]{ffqu} oriented blow-ups appear to compare regoluous images of $\R^2$ with regular images of $\R^2$. In Section \ref{s3} we recall the main properties of drilling blow-ups developed in \cite{fe3} and we improve some of them in order to prove Theorems \ref{red2} and \ref{red4}.

\subsection{Applications}
We present next two applications of Theorems \ref{red2} and \ref{red4}.

\subsubsection{Nash compactifications of Nash manifolds with corners}
A useful tool in Semialgebraic geometry is the use of semialgebraic compactifications of semialgebraic sets \cite{fg2,fg3}. When dealing with Nash manifolds with corners $\Qq\subset\R^n$ one can find, as an application of the techniques used to prove Theorems \ref{red2} and \ref{red4}, compactifications of $\Qq$ that are again Nash manifolds with corners. This results has further applications for approximation (and relative approximations) results of $\Cont^r$ semialgebraic maps with target space a Nash manifold with corners by Nash maps with the same target space \cite{cf2,cf3}.

\begin{thm}\label{compactclosure}
Let $\Qq\subset\R^n$ be a Nash manifold with corners. Then there exists a Nash embedding ${\tt j}:\Qq\hookrightarrow\R^{n+p}$ for some $p\geq0$ such that $\cl(\Qq)\subset\R^{n+p}$ is a compact Nash manifold with corners.
\end{thm}

\subsubsection{Nash uniformization of semialgebraic sets}
One can combine Bierstone-Parusi\'nski's desingularization of closed semialgebraic sets with Theorems \ref{red2} and \ref{red4} to obtain the following Nash uniformization result for general semialgebraic sets.

\begin{cor}[Nash uniformization of semialgebraic sets]\label{amal}
Let $\Ss\subset\R^m$ be a semialgebraic set. Then there exist: 
\begin{itemize}
\item[{\rm(i)}] A pairwise disjoint finite union $X$ of pure dimensional non-singular real algebraic sets $X_i\subset\R^n$ of dimensions $d_i$ and normal-crossings divisor $Y_i\subset X_i$ for $i=1,\ldots,r$. 
\item[{\rm(ii)}] Nash quasi-manifolds with corners $\Ss_i^\bullet\subset X_i$ of dimension $d_i$ whose closure in $X_i$ is a Nash manifold with corners $\Qq^\bullet_i\subset X$ and whose boundary $\partial\Qq^\bullet_i$ has $Y_i$ as its Zariski closure for $i=1,\ldots,r$. If $\Ss$ is closed, $\Ss_i^\bullet$ is a Nash manifold with corners for each $i=1,\ldots,r$. In addition, the number of connected components $\Ss^\bullet:=\bigsqcup_{i=1}^r\Ss^\bullet_i$ and $\Qq^\bullet:=\bigsqcup_{i=1}^r\Qq^\bullet_i$ coincide, the number or irreducible components of $\Ss^\bullet$ and $\Qq^\bullet$ coincide and both numbers coincide with the number of components connected by analytic paths of $\Ss$.
\item[\rm{(iii)}] A Nash map $f:\R^n\to\R^m$ such that $f(\Ss^\bullet)=\Ss$ and the restriction $f|_{\Ss^\bullet}:\Ss^\bullet\to\Ss$ is proper. In addition, if $\Ss$ is closed, we may choose $f$ polynomial.
\item[\rm{(iv)}] Closed semialgebraic subsets $\Rr_i\subset\Ss_i:=f(\Ss_i^\bullet)$ of dimension strictly smaller than the dimension of $\Ss_i$ such that $\Ss_i\setminus\Rr_i$ and $\Ss_i^\bullet\setminus f^{-1}(\Rr_i)$ are Nash manifolds (of the same dimension) and the Nash map $f|_{\Ss_i^\bullet\setminus f^{-1}(\Rr_i)}:\Ss_i^\bullet\setminus f^{-1}(\Rr_i)\to\Ss_i\setminus\Rr_i$ is a Nash diffeomorphism for $i=1,\ldots,r$. 
\end{itemize}
In addition, we may assume in all cases that both $X$ and $\Ss^\bullet$ are compact, but having only that $f$ is regular map (instead of a polynomial map), when $\Ss$ is besides a closed semialgebraic set.
\end{cor}
\begin{remark}
To prove the latter part of the statement of Corollary \ref{amal} concerning the regularity of $f$ (when $\Ss$ is besides a closed semialgebraic set), one has to use the proof of Theorem \ref{compactclosure} taking into account that we employ stereographic projections of high dimensional spheres (which are regular maps), \cite[Lem.C.2]{fe3} and Artin-Mazur's description \cite[\S8.4]{bcr} of Nash manifolds and maps and Mostowski's trick \cite[Lem.6]{mo} (see also \cite[\S2.1]{cf2}), which involve affine projections (which are polynomial maps). We leave the concrete details to the reader. 
\end{remark}

\subsection{Structure of the article}
The article is organized as follows. In Section \ref{s2} we recall the concepts and main properties of regular and smooth points of a semialgebraic set. We also analyze the main properties of Nash manifolds with corners and we refer the reader to \cite{fgr} for further details. We also prove in this section Theorems \ref{red2} and \ref{red4} for the $1$-dimensional case. In Section \ref{s3} we recall the main properties of the drilling blow-up of a Nash manifold along a closed Nash submanifold proposed in \cite[\S3]{fe3}. We also present some additional new properties of the drilling blow-up (original of this article) that we need to prove Theorems \ref{red2} and \ref{red4}. In Section \ref{s4} we introduce the checkerboard sets (as a particular type of chessboard sets) and we reduce the proof of Theorem \ref{red2} to prove it for closed checkerboard sets, which is done in \S\ref{44}. We also prove in Section \ref{s4} Theorem \ref{compactclosure} (taking profit of the techniques already introduced in this section). In Section \ref{s5} we prove Theorem \ref{red4} after reducing it to the case of general checkerboard sets. Combining this result with Bierstone-Parusi\'nski's Theorem we show Corollary \ref{amal}.


\subsection{Acknowledgements}
The authors are very grateful to E. Bierstone for useful comments and to S. Schramm for a careful reading of the final version and for the suggestions to refine its redaction.

\section{Preliminary results}\label{s2}

In this section we collect some preliminary concepts and results that will be used freely along this article. We include them for the sake of completeness and to ease the reading of the article. We refer the reader to \cite{fe3,fgr} for further details. We also present some new results in \S\ref{1dimc}.

\subsection{Regular points versus smooth points.}\label{regsmooth}

The set $\Reg(\Ss)$ of \em regular points of a semialgebraic set $\Ss\subset\R^n$ \em is defined as follows. Let $X$ be the Zariski closure of $\Ss$ in $\R^n$ and $\widetilde{X}$ the complexification of $X$, that is, the smallest complex algebraic subset of $\C^n$ that contains $X$. Define $\Reg(X):=X\setminus\Sing(\widetilde{X})$ and let $\Reg(\Ss)$ be the interior of $\Ss\setminus\Sing(\widetilde{X})$ in $\Reg(X)$. Observe that $\Reg(\Ss)$ is a finite union of disjoint {\em Nash manifolds} maybe of different dimensions. We refer the reader to \cite[\S2.A]{fe3} for further details concerning the set of regular points of a semialgebraic~set.

A point $x\in\Ss$ is \em smooth \em if there exists an open neighborhood $U\subset\R^n$ of $x$ such that $U\cap\Ss$ is a Nash manifold. It holds that each regular point is a smooth point, but the converse is not always true even if $\Ss=X$ is a real algebraic set \cite[Ex.2.1]{fe3}. The set $\Sth(\Ss)$ of smooth points of a semialgebraic set $\Ss\subset\R^n$ is by \cite{st} a semialgebraic subset of $\R^n$ (and consequently a union of Nash submanifolds of $\R^n$ possibly of different dimension), which contains $\Reg(\Ss)$ (maybe as a proper subset as it happens in \cite[Ex.2.1]{fe3}), and it is open in $\Ss$. The set of points of dimension $k$ of $\Sth(\Ss)$ is either the empty-set or a Nash manifold of dimension $k$ for each $k=0,1,\ldots,d$. In particular, if $\Ss$ is pure dimensional, $\Sth(\Ss)$ is a Nash submanifold of $\R^n$. If $X$ is a real algebraic set, $\Sing(X)$ is always an algebraic subset of $X$ whereas the set $X\setminus\Sth(X)$ of non-smooth points is in general only a semialgebraic subset of $X$ (see \cite[Ex.2.1]{fe3}). 

If $X$ is a non-singular real algebraic set of dimension $d$, then $\Reg(X)=\Sth(X)$. Thus, if $\Ss\subset X$ is a pure dimensional semialgebraic set of dimension $d$, we have $\Reg(\Ss)=\Sth(\Ss)$, because the Zariski closure $\ol{\Ss}^{\zar}$ of $\Ss$ is a union of irreducible components of $X$, so it is a non-singular real algebraic set of dimension $d$.

\subsection{Unions of orthants versus $\Cont^1$ manifolds with corners}\label{orthants}
We next recall that a union of orthants of $\R^d$ with vertex at the origin is a $\Cont^1$ manifold with corners if and only if it equals, possibly after reordering the variables and reversing inequalities, either $\R^d$ or $\{\x_1\geq0,\ldots,\x_r\geq0\}$ for some $1\leq r\leq d$. This result is presumably folklore, but since we have found no precise proof of it in the literature, we include one below. For each $\veps:=(\veps_1,\ldots,\veps_d)\in\{-1,1\}^d$ we denote $\Oo_\veps:=\{\veps_1\x_1\geq0,\ldots,\veps_d\x_d\geq0\}$, and for each ${\mathfrak F}\subset\{-1,1\}^d$ we write $C({\mathfrak F}):=\bigcup_{\veps\in{\mathfrak F}}\Oo_\veps$. There are exactly $2^d$ distinct orthants (one for each $\veps\in\{-1,1\}^d$), $C(\{-1,1\}^d)=\R^d$, and $\Int(\Oo_\veps)\cap\Oo_\delta=\varnothing$ whenever $\veps\neq\delta$.

We refer the reader to \cite[\S3.7]{o} for general background on the metric structure of cube graphs and their main properties. Let $\Gamma:=({\mathfrak V}_\Gamma,{\mathfrak E}_\Gamma)$ be the {\em cube graph} whose vertices are the points of the set ${\mathfrak V}_\Gamma:=\{-1,1\}^d$, with an edge $e\in{\mathfrak E}_\Gamma$ between two vertices if and only if they differ in exactly one coordinate. We consider on $\Gamma$ the {\em Hamming distance} $d(\cdot,\cdot):\{-1,1\}^d\times\{-1,1\}^d\to\N\cup\{0\}$, which assigns length $1$ to each edge, so that the distance between two vertices is the minimum number of edges connecting them. The {\em metric interval} $I(\veps,\delta)$ between two vertices $\veps,\delta\in\{-1,1\}^d$ of $\Gamma$ is the set of $\rho\in\{-1,1\}^d$ such that $d(\veps,\rho)+d(\rho,\delta)=d(\veps,\delta)$. If we denote $A(\veps,\delta):=\{i\in\{1,\ldots,d\}:\ \veps_i=\delta_i\}$, one has $I(\veps,\delta)=\{\rho\in\{-1,1\}^d:\ \rho_i=\veps_i\ \forall i\in A(\veps,\delta)\}$.

We assign to each subset ${\mathfrak F}\subset\{-1,1\}^d$ the subgraph $\Lambda({\mathfrak F})$ whose vertices are the points of ${\mathfrak F}$ and whose edges are those edges of ${\mathfrak E}_\Gamma$ between two vertices of ${\mathfrak F}$. For simplicity we identify $\Lambda({\mathfrak F})$ with its set of vertices ${\mathfrak F}$. We say that ${\mathfrak F}\subset\{-1,1\}^d$ is {\em (geodesically) convex} if $I(\veps,\delta)\subset{\mathfrak F}$ for each pair of vertices $\veps,\delta\in{\mathfrak F}$. By \cite[Thm.3.26]{o}, ${\mathfrak F}\subset\{-1,1\}^d$ is (geodesically) convex if and only if there exist $I\subset\{1,\ldots,d\}$ and $\sigma:=(\sigma_i)_{i\in I}\in\{-1,1\}^I$ such that ${\mathfrak F}:=\{\rho\in\{-1,1\}^d:\ \rho_i=\sigma_i\ \forall i\in I\}$ (that is, ${\mathfrak F}$ is itself a cube inside $\{-1,1\}^d$). Observe that $C({\mathfrak F})=\{\sigma_i\x_i\geq0:\ i\in I\}$ (and conversely, if this is the case, then ${\mathfrak F}$ is (geodesically) convex).

\begin{thm}[Convexity of unions of orthants]\label{conort}
A union of orthants $C({\mathfrak F}):=\bigcup_{\veps\in{\mathfrak F}}\Oo_\veps$ is convex if and only if $I(\veps,\delta)\subset{\mathfrak F}$ for each pair of vertices $\veps,\delta\in{\mathfrak F}$, or equivalently, if $C({\mathfrak F})=\{\sigma_i\x_i\geq0:\ i\in I\}$ for some $I\subset\{1,\ldots,d\}$ and $\sigma:=(\sigma_i)_{i\in I}\in\{-1,1\}^I$.
\end{thm}
\begin{proof}
We first claim: {\em If $\veps,\delta\in\{-1,1\}^d$, then $\Oo_\veps+\Oo_\delta=\bigcup_{\rho\in I(\veps,\delta)}\Oo_\rho$}.
We first prove the left-to-right inclusion ($\subset$). Let $x:=(x_1,\ldots,x_d)\in\Oo_\veps$, $y:=(y_1,\ldots,y_d)\in\Oo_\delta$, and denote $z:=x+y=(z_1,\ldots,z_d)$. If $i\in A(\veps,\delta)$, both $\veps_ix_i\geq0$ and $\veps_iy_i=\delta_iy_i\geq0$, so $\veps_iz_i=\veps_i(x_i+y_i)\geq0$. Define $\eta\in\{-1,1\}^d$ by
$$
\eta_i=\begin{cases}
\veps_i&\text{if $i\in A(\veps,\delta)$,}\\
1&\text{if $z_i\geq0$ and $i\not\in A(\veps,\delta)$,}\\
-1&\text{if $z_i<0$ and $i\not\in A(\veps,\delta)$.}
\end{cases}
$$
Observe that $\eta\in I(\veps,\delta)$ and $z\in\Oo_\eta\subset\bigcup_{\rho\in I(\veps,\delta)}\Oo_\rho$.
To prove the converse inclusion $(\supset)$, pick $z:=(z_1,\ldots,z_d)\in\Oo_\rho$ for some $\rho\in I(\veps,\delta)$. We have $\rho_iz_i\geq0$ for each $i=1,\ldots,d$. Define $x,y\in\R^d$ coordinatewise as follows. If $i\in A(\veps,\delta)$ (that is, $\rho_i=\veps_i=\delta_i$), set $x_i=z_i$ and $y_i=0$, so that $\veps_ix_i\geq0$ and $\delta_i y_i=0$. If $i\not\in A(\veps,\delta)$ (that is, $\veps_i=-\delta_i$), set $x_i=z_i$ and $y_i=0$ if $\veps_i z_i\geq0$, and $x_i=0$ and $y_i=z_i$ if $\delta_iz_i=-\veps_i z_i\geq0$. Observe that $x:=(x_1,\ldots,x_d)\in\Oo_\veps$, $y:=(y_1,\ldots,y_d)\in\Oo_\delta$, and $z=x+y\in\Oo_\veps+\Oo_\delta$, as claimed.

As $C({\mathfrak F})$ is a finite union of cones, it is itself a cone (that is, $\lambda C({\mathfrak F})=C({\mathfrak F})$ for every $\lambda\geq0$); hence it is convex if and only if $x+y\in C({\mathfrak F})$ for all $x,y\in C({\mathfrak F})$. By the claim, this means $I(\veps,\delta)\subset{\mathfrak F}$ for each pair of vertices $\veps,\delta\in{\mathfrak F}$. As noted above, this is equivalent to $C({\mathfrak F})=\{\sigma_i\x_i\geq0:\ i\in I\}$ for some $I\subset\{1,\ldots,d\}$ and $\sigma:=(\sigma_i)_{i\in I}\in\{-1,1\}^I$, as required.
\end{proof}

For the following proof we make use of Bouligand's tangent cone \cite{b}. Given a set $S\subset\R^d$, the {\em Bouligand tangent cone $T_B(S,x)$ to $S$ at $x\in S$} is the collection of all vectors $v\in\R^d$ satisfying the following condition: {\em there exist a sequence of points $\{x_k\}_{k\geq1}\subset S$ converging to $x$ and a decreasing sequence of positive real numbers $\{t_k\}_{k\geq1}$ converging to $0$ such that $\lim_{k\to\infty}\left(\frac{x_k-x}{t_k}\right)=v$}. Observe that if $v\in T_B(S,x)$, then $\lambda v\in T_B(S,x)$ for each non-negative real number $\lambda\geq0$. A straightforward computation shows that if ${\mathfrak F}\subset\{-1,1\}^d$ and $0$ denotes the origin of $\R^d$, then $T_B(C({\mathfrak F})\cap U,0)=C({\mathfrak F})$ for each open neighborhood $U$ of the origin in $\R^d$.

\begin{thm}
A union $C$ of orthants of $\R^d$ is a ${\Cont}^1$ manifold with corners if and only if it is convex, or, equivalently, if there exist $I\subset\{1,\ldots,d\}$ and $\sigma:=(\sigma_i)_{i\in I}\in\{-1,1\}^I$ such that $C=\{\sigma_i\x_i\geq0:\ i\in I\}$.
\end{thm}
\begin{proof}
The ``if'' implication follows from Theorem \ref{conort}, so let us prove the converse. As $C$ is a ${\Cont}^1$ manifold with corners, there exist an open neighborhood $U$ of the origin and a $\Cont^1$ diffeomorphism $\varphi:U\to\R^d$ such that $\varphi(0)=0$ and $\varphi(U\cap C)$ equals either $\R^d$ or $\{x_1\geq0,\ldots,x_e\geq0\}$ for some $1\leq e\leq d$. As $\varphi$ is a $\Cont^1$ diffeomorphism, Taylor's theorem for $\Cont^1$ maps gives $d_0\varphi(C)=d_0\varphi(T_B(C\cap U,0))=T_B(\varphi(U\cap C),0)$, which equals either $\R^d$ or $\{x_1\geq0,\ldots,x_e\geq0\}$ for some $1\leq e\leq d$. As in both cases $T_B(\varphi(U\cap C),0)$ is convex, and $d_0\varphi$ is a linear isomorphism (hence preserves convexity), it follows that $C$ is also convex. By Theorem \ref{conort} we deduce $C=\{\sigma_i\x_i\geq0:\ i\in I\}$ for some $I\subset\{1,\ldots,d\}$ and $\sigma:=(\sigma_i)_{i\in I}\in\{-1,1\}^I$, as required.
\end{proof}

\subsection{Nash manifolds with corners.}\label{angoliPre}
Let $\Qq\subset\R^n$ be a Nash manifold with corners. The set of {\em internal points} of $\Qq$ is $\Int(\Qq):=\Sth(\Qq)$. The {\em boundary} $\partial\Qq$ of $\Qq$ is $\partial\Qq:=\Qq\setminus\Int(\Qq)=\Qq\setminus\Sth(\Qq)$. If $\ol{\Qq}^{\zar}$ is a non-singular real algebraic set, then $\Reg(\Qq)=\Sth(\Qq)=\Int(\Qq)$ and $\partial\Qq=\Qq\setminus\Reg(\Qq)$. Otherwise, $\partial\Qq\subset\Qq\setminus\Reg(\Qq)$ and the inclusion could be strict.

A Nash manifold with corners $\Qq\subset\R^n$ is locally closed. Consequently, {\em $\Qq$ is a closed Nash submanifold with corners of the Nash manifold $\R^n\setminus(\cl(\Qq)\setminus\Qq)$.} In \cite[Thm.1.11]{fgr} it is shown that $\Qq$ is a closed subset of an affine Nash manifold of its same dimension. Recall that a Nash subset $Y$ of a Nash manifold $M\subset\R^n$ {\em has only Nash normal-crossings in $M$} if for each point $y\in Y$ there exists an open semialgebraic neighborhood $U\subset M$ such that $Y\cap U$ is a Nash normal-crossings divisor of $U$.

\begin{thm}[{\cite[Thm.1.11]{fgr}}]\label{corners0}
Let $\Qq\subset\R^n$ be a $d$-dimensional Nash manifold with corners. There exists a $d$-dimensional Nash manifold $M\subset\R^n$ that contains $\Qq$ as a closed subset and satisfies:
\begin{itemize}
\item[\rm{(i)}] The Nash closure $Y$ of $\partial\Qq$ in $M$ has only Nash normal-crossings in $M$ and $\Qq\cap Y=\partial\Qq$.
\item[\rm{(ii)}] For every $x\in\partial\Qq$ the smallest analytic germ that contains the germ $\partial\Qq_x$ is $Y_x$.
\item[\rm{(iii)}] $M$ can be covered by finitely many open semialgebraic subsets $U_i$ (for $i=1,\ldots,r$) equipped with Nash diffeomorphisms $u_i:=(u_{i1},\dots,u_{id}):U_i\to\R^d$ such that:
$$\hspace{-3mm}
\begin{cases}
\text{$U_i\subset\Int(\Qq)$ or $U_i\cap\Qq=\varnothing$},&\text{\!if $U_i$ does not meet $\partial\Qq$,}\\ 
U_i\cap\Qq=\{u_{i1}\ge0,\dots,u_{ik_i}\ge0\},&\text{\!if $U_i$ meets $\partial\Qq$ (for a suitable $k_i\ge1$).} 
\end{cases}
$$
\end{itemize}
\end{thm} 

The Nash manifold $M$ is called a {\em Nash envelope of $\Qq$}. In general, it is not guaranteed that the Nash closure $Y$ of $\partial\Qq$ in $M$ is a Nash normal-crossings divisor of $M$ as we show next.

\begin{figure}[!ht]
\begin{center}
\begin{tikzpicture}[scale=1.8]
\draw[thick,->] (-0.5, 0) -- (1.3, 0) ;
\draw[thick,->] (0, -0.8) -- (0, 0.8) ;
\draw [thick,scale=1,domain=0:1, draw=black, fill=gray!100, fill opacity=0.3,smooth,samples=1000]
(0, 0)
-- plot ({\x}, {sqrt(\x*\x-\x*\x*\x*\x)})
-- (1, 0);
\draw [thick,scale=1,domain=0:1, draw=black, fill=gray!100, fill opacity=0.3,samples=1000]
(0, 0)
-- plot ({\x},{ -sqrt(\x*\x-\x*\x*\x*\x)})
-- (1, 0);
\end{tikzpicture}
\end{center}
\caption{\small{The teardrop.}}
\label{teardrop}
\end{figure}

\begin{example}
The {\em teardrop} $\Qq:=\{\x\geq 0,\y^2\leq\x^2-\x^4\}\subset\R^2$ is a Nash manifold with corners (Figure \ref{teardrop}). Given any open semialgebraic neighborhood $M$ of $\Qq$ in $\R^2$ the Nash closure of $\partial\Qq$ in $M$ is not a Nash normal-crossings divisor. 
\hfill$\sqbullet$
\end{example} 

We define now Nash manifolds with divisorial corners. 

\begin{defn}
A Nash manifold with corners $\Qq\subset\R^n$ is a Nash manifold with {\em divisorial corners} if there exists a Nash envelope $M\subset\R^n$ such that the Nash closure of $\partial\Qq$ in $M$ is a Nash normal-crossings divisor.\hfill$\sqbullet$
\end{defn}

A {\em facet} of a Nash manifold with corners $\Qq\subset\R^n$ is the (topological) closure in $\Qq$ of a connected component of $\Sth(\partial\Qq)$. As $\partial\Qq=\Qq\setminus\Sth(\Qq)$ is semialgebraic, the facets are semialgebraic and finitely many. The non-empty intersections of facets of $\Qq$ are the {\em faces} of $\Qq$. In \cite{fgr} the following characterization for Nash manifolds with divisorial corners is shown:

\begin{thm}[{\cite[Thm.1.12, Cor.6.5]{fgr}}]\label{divisorialPre}
Let $\Qq\subset\R^n$ be a $d$-dimensional Nash manifold with corners. The following assertions are equivalent:
\begin{itemize}
\item[\rm{(i)}] There exists a Nash envelope $M\subset\R^n$ where the Nash closure of $\partial\Qq$ is a Nash normal-crossings divisor.
\item[\rm{(ii)}] Every facet $\Ff$ of $\Qq$ is contained in a Nash manifold $X\subset\R^n$ of dimension $d-1$.
\item[\rm{(iii)}] The number of facets of $\Qq$ that contain every given point $x\in\partial\Qq$ coincides with the number of connected components of the germ $\Sth(\partial\Qq)_x$.
\item[\rm{(iv)}] All the facets of $\Qq$ are Nash manifold with divisorial corners. 
\end{itemize}
If that is the case, the Nash manifold $M$ in {\rm(i)} can be chosen such that the Nash closure in $M$ of every facet $\Ff$ of $\Qq$ meets $\Qq$ exactly along $\Ff$.
\end{thm}

Note that properties (ii), (iii) and (iv) are intrinsic properties of $\Qq$ and do not depend on the Nash envelope $M$. The faces of a Nash manifold with divisorial corners are again Nash manifolds with divisorial corners. If a Nash envelope $M\subset\R^n$ of $\Qq$ satisfies (one of the equivalent) conditions of Theorem \ref{divisorialPre}, then every open semialgebraic neighborhood $M'\subset M$ of $\Qq$ satisfies such conditions. For the rest of this article, we make the following: 

\noindent{\bf Assumption.} {\em A Nash manifold with corners means a Nash manifold with divisorial corners}.

\subsection{The $1$-dimensional case}\label{1dimc}

Let us prove here Theorems \ref{red2} and \ref{red4} when $\Ss\subset\R^n$ is a $1$-dimensional chessboard sets.

\begin{proof}[Proof of Theorems {\em \ref{red2}} and {\em \ref{red4}} for $1$-dimensional chessboard sets]
As $\ol{\Ss}^{\zar}$ is non-singular, the connected components of $\ol{\Ss}^{\zar}$ are by \cite[Prop.1.6]{fe3} and its proof Nash diffeomorphic to either $\R$ or the sphere $\sph^1$. Thus, the connected components $\Ss_1,\ldots,\Ss_r$ of $\Ss$ are Nash diffeomorphic to either $(0,1)$, $[0,1)$, $[0,1]$ or $\sph^1$, that is, they are Nash manifolds with corners (some of them with empty boundary). If $\Ss$ is closed in $\R^n$, we define $X:=\ol{\Ss}^{\zar}$, $\Qq:=\Ss$, $\Rr:=\varnothing$ and $f:=\id_X:X\to X$. 

Suppose next $\Ss$ is not closed in $\R^n$. Observe that $\Tt:=\cl(\Ss)\setminus\Ss$ is a (non-empty) finite set (as it is a semialgebraic set of dimension $<1$). The $1$-dimensional non-singular real algebraic set $Y:=\{(x,y)\in\ol{\Ss}^{\zar}\times\R:\ yh(x)=1\}$, where $h\in\R[\x]$ is any polynomial whose zero set is $\Tt$, and the projection $\pi:\ol{\Ss}^{\zar}\times\R\to\ol{\Ss}^{\zar}$ (onto the first factor) satisfy that $\pi|_{Y}:Y\to\ol{\Ss}^{\zar}\setminus\Tt$ is a biregular diffeomorphism. Let $(X,g)$ be the desingularization of the projective closure $\ol{Y}$ of $Y$ in the projective space $\R\PP^{n+1}$, which is a compact non-singular real algebraic set that contains a semialgebraic set $\Ss^\bullet$ biregularly diffeomorphic to $\Ss$, because the singular points of $\ol{Y}$ does not belong to $Y$ and $\Ss\subset Y$. Thus, to finish it is enough to take $\Qq^\bullet:=\cl(\Ss^\bullet)$ and $f:=\pi\circ g:g^{-1}(Y)\to\ol{\Ss}^{\zar}$.
\end{proof}

In Sections \ref{s4} and \ref{s5} we prove Theorems \ref{red2} and \ref{red4} for semialgebraic sets of dimension $d\geq2$.

\section{Drilling blow-up}\label{s3}

In this section we recall the main properties of the drilling blow-up of a Nash manifold with center a closed Nash submanifold \cite[\S5]{fe3}. We prove some additional results (Fact \ref{bigstepa6}, Fact \ref{locald-1}, \S\ref{addbu} and \S\ref{exst0}) that we need in Sections \ref{s4} and \ref{s5}.

\subsection{Local structure of the drilling blow-up.}\label{bigstep}

Let $M\subset\R^m$ be a $d$-dimensional Nash manifold and $N\subset M$ a closed $e$-dimensional Nash submanifold. As we are interested in the local structure, assume that there exists a Nash diffeomorphism $u:=(u_1,\ldots,u_d):M\to\R^d$ such that $N=\{u_{e+1}=0,\ldots,u_d=0\}$. Denote $\psi:=u^{-1}:\R^d\equiv\R^e\times\R^{d-e}\to M$. Let $k\in\N$ and let $\zeta_{e+1},\ldots,\zeta_d:\R^d\to\R^k$ be Nash maps such that the vectors $\zeta_{e+1}(y,0),\ldots,\zeta_d(y,0)$ are linearly independent for each $y\in\R^e$. Write $z\in\R^{d-e}$ as $z:=(z_{e+1},\ldots,z_d)$. Consider the Nash maps 
\begin{align*}
&\varphi:\R^d\equiv\R^e\times\R^{d-e}\to\R^k,\ (y,z)\mapsto\zeta_{e+1}(y,z)z_{e+1}+\cdots+\zeta_d(y,z)z_d,\\
&\phi:\R^e\times\R\times\sph^{d-e-1}\to\R^k,\ (y,\rho,w)\mapsto\zeta_{e+1}(y,\rho w)w_{e+1}+\cdots+\zeta_d(y,\rho w)w_d
\end{align*} 
and assume that $\varphi(y,z)=0$ if and only if $z=0$. 

\begin{fact}\cite[\S5.A.1]{fe3}
Consider the (well-defined) Nash map:
$$
\Phi:\R^e\times\R\times\sph^{d-e-1}\to M\times\sph^{k-1},\ (y,\rho,w)\mapsto\Big(\psi(y,\rho w),\frac{\phi(y,\rho,w)}{\|\phi(y,\rho,w)\|}\Big).
$$
\end{fact}

\begin{fact}\label{bigstepa2}\cite[\S5.A.2]{fe3}
Fix $\epsilon=\pm$ and denote 
$$
I_\epsilon:=\begin{cases}
[0,+\infty)&\text{ if $\epsilon=+$,}\\
(-\infty,0]&\text{ if $\epsilon=-$.}
\end{cases}
$$ 
The closure $\widetilde{M}_\epsilon$ in $M\times\sph^{k-1}$ of the set
$$
\Gamma_\epsilon:=\Big\{\Big(\psi(y,z),\epsilon\frac{\varphi(y,z)}{\|\varphi(y,z)\|}\Big)\in M\times\sph^{k-1}:\,z\neq 0\Big\}
$$ 
is a Nash manifold with boundary such that:
\begin{itemize}
\item[{\rm(i)}] $\widetilde{M}_\epsilon\subset\im(\Phi)$.
\item[{\rm(ii)}] The restriction of $\Phi$ to $\R^e\times I_\epsilon\times\sph^{d-e-1}$ induces a Nash diffeomorphism between $\R^e\times I_\epsilon\times\sph^{d-e-1}$ and $\widetilde{M}_\epsilon$. Consequently, $\partial\widetilde{M_\epsilon}=\Phi(\R^e\times\{0\}\times\sph^{d-e-1})$ and $\Gamma_\epsilon=\Int(\widetilde{M_\epsilon})=\Phi(\R^e\times(I_\epsilon\setminus\{0\})\times\sph^{d-e-1})$.
\end{itemize}
\end{fact}

\begin{fact}\label{bigstepa3}\cite[\S5.A.3]{fe3}
Denote $\Rr:=\partial\widetilde{M}_+=\partial\widetilde{M}_-$ and $\widehat{M}:=\widetilde{M}_+\cup\widetilde{M}_-=\Gamma_+\sqcup\Rr\sqcup\Gamma_-$. Then $\Phi$ induces a Nash diffeomorphism between $\R^e\times\R\times\sph^{d-e-1}$ and $\widehat{M}$, which is the Nash closure of $\widetilde{M}_+$ and $\widetilde{M}_-$ in $M\times\sph^{k-1}$. In addition, the Nash map $\sigma: M\times\sph^{k-1}\to M\times\sph^{k-1}, (a,b)\to(a,-b)$ induces a Nash involution on $\widehat{M}$ without fixed points such that $\sigma(\widetilde{M}_+)=\widetilde{M}_-$ and $\Phi(y,-\rho,-w)=(\sigma\circ\Phi)(y,\rho,w)$ for each $(y,\rho,w)\in\R^e\times\R\times\sph^{d-e-1}$.
\end{fact} 

\begin{fact}\cite[\S5.A.4]{fe3}
Consider the projection $\pi:M\times\sph^{k-1}\to M$ onto the first factor and denote $\pi_\epsilon:=\pi|_{\widetilde{M_\epsilon}}$ for $\epsilon=\pm$. Then
\begin{itemize}
\item[{\rm(i)}] $\pi_\epsilon$ is proper, $\pi_\epsilon(\widetilde{M_\epsilon})=M$ and $\Rr=\pi^{-1}_\epsilon(N)$.
\item[{\rm(ii)}] The restriction $\pi_\epsilon|_{\Gamma_\epsilon}:\Gamma_\epsilon\to M\setminus N$ is a Nash diffeomorphism.
\item[{\rm(iii)}] For each $q\in N$ it holds $\pi_\epsilon^{-1}(q)=\{q\}\times\sph_q^{d-e-1}$ where $\sph_q^{d-e-1}$ is the sphere of dimension $d-e-1$ obtained when intersecting the sphere $\sph^{k-1}$ with the linear subspace $L_q$ generated by $(\zeta_{e+1}\circ u)(q),\ldots,(\zeta_d\circ u)(q)$.
\end{itemize}
\end{fact}

\begin{figure}[!ht]
\begin{center}
\begin{tikzpicture}[scale=0.75]

\draw (0.75,4) arc (90:270:0.75cm and 1.5cm);
\draw[dashed] (0.75,1) arc (270:450:0.75cm and 1.5cm);

\draw (1.1,2) -- (6.1,2);
\draw (1.2,2.5) -- (6.15,2.5);
\draw (1.1,3) -- (6.1,3);

\draw(0.75,3.75) arc (90:270:0.55cm and 1.25cm);
\draw[dashed] (0.75,1.25) arc (270:450:0.55cm and 1.25cm);

\draw[line width=2pt] (0.75,3.5) arc (90:270:0.35cm and 1cm);
\draw[line width=2pt, dashed] (0.75,1.5) arc (270:450:0.35cm and 1cm);

\draw (0.75,1) -- (5.75,1);
\draw (0.75,4) -- (5.75,4);

\draw(0.75,1.25) -- (5.75,1.25);
\draw(0.75,3.75) -- (5.75,3.75);

\draw[line width=2pt] (0.75,1.5) -- (5.75,1.5);
\draw[line width=2pt] (0.75,3.5) -- (5.75,3.5);

\draw[fill=gray!100,opacity=0.5,draw=none] (0.75,3.5) arc (90:270:0.35cm and 1cm) -- (5.75,1.5) arc (270:90:0.35cm and 1cm) -- (5.75,3.5);
\draw[fill=gray!70,opacity=0.5,draw=none] (0.75,3.75) arc (90:270:0.55cm and 1.25cm) -- (5.75,1.25) arc (270:90:0.55cm and 1.25cm) -- (5.75,3.75);

\draw (3,2.5) node{\small$\Phi^{-1}(\widehat{\pi}^{-1}(N))$};

\draw[fill=gray!10,opacity=0.5,draw=none] (0.75,4) arc (90:270:0.75cm and 1.5cm) -- (5.75,1) arc (270:90:0.75cm and 1.5cm) -- (5.75,4);

\draw[fill=gray!10,opacity=0.5,draw=none] (5.75,2.5) ellipse (0.75cm and 1.5cm);
\draw[fill=gray!70,opacity=0.5,draw=none] (5.75,2.5) ellipse (0.55cm and 1.25cm);
\draw[fill=gray!100,opacity=0.5,draw=none] (5.75,2.5) ellipse (0.35cm and 1cm);

\draw (5.45,2) -- (6.1,2);
\draw (5.40,2.5) -- (6.15,2.5);
\draw (5.45,3) -- (6.1,3);

\draw (5.75,2.5) ellipse (0.75cm and 1.5cm);
\draw (5.75,2.5) ellipse (0.55cm and 1.25cm);
\draw[line width=2pt] (5.75,2.5) ellipse (0.35cm and 1cm);

\draw (9.75,4) arc (90:270:0.75cm and 1.5cm);
\draw[dashed] (9.75,1) arc (270:450:0.75cm and 1.5cm);

\draw (9.75,3.75) arc (90:270:0.55cm and 1.25cm);
\draw[dashed] (9.75,1.25) arc (270:450:0.55cm and 1.25cm);

\draw (9.75,1) -- (14.75,1);
\draw (9.75,4) -- (14.75,4);

\draw (9.75,1.25) -- (14.75,1.25);
\draw (9.75,3.75) -- (14.75,3.75);

\draw (9.75,2.5) node{\small$\bullet$};
\draw[line width=2pt] (9.75,2.5) -- (14.75,2.5);

\draw[fill=gray!70,opacity=0.5,draw=none] (9.75,3.75) arc (90:270:0.55cm and 1.25cm) -- (14.75,1.25) arc (270:90:0.55cm and 1.25cm) -- (14.75,3.75);

\draw (12.25,2.85) node{\small$u(N)$};
\draw[fill=gray!10,opacity=0.5,draw=none] (9.75,4) arc (90:270:0.75cm and 1.5cm) -- (14.75,1) arc (270:90:0.75cm and 1.5cm) -- (14.75,4);

\draw[fill=gray!10,opacity=0.5,draw=none] (14.75,2.5) ellipse (0.75cm and 1.5cm);
\draw[fill=gray!70,opacity=0.5,draw=none] (14.75,2.5) ellipse (0.55cm and 1.25cm);

\draw (14.75,2.5) ellipse (0.75cm and 1.5cm);
\draw (14.75,2.5) ellipse (0.55cm and 1.25cm);

\draw[->, line width=1pt] (6.75,2.5) -- (8.75,2.5);
\draw[->, line width=1pt] (5.75,0.5) -- (9.75,0.5);
\draw[line width=1pt] (5.75,0.4) -- (5.75,0.6);
\draw (4.75,0.5) node{\small$(y,\rho,w)$};
\draw (10.75,0.5) node{\small$(y,\rho w)$};
\draw (14.5,4.35) node{\small$\R^d$};
\draw (3,4.35) node{\small$\R^e\times[0,+\infty)\times\sph^{d-e-1}$};
\draw (7.75,2.85) node{\small$u\circ\pi_+\circ\Phi$};
\draw (14.75,2.5) node{\small$\bullet$};

\end{tikzpicture}
\end{center}
\caption{\small{Local structure of the drilling blow-up $\widetilde{M}_+$ of $M$ of center $N$ (figure borrowed from \cite[Fig.3]{fe3}).}}
\end{figure}

Denote $\widehat{\pi}:=\pi|_{\widehat{M}}$ and consider the commutative diagram.
\begin{equation}\label{diagpol}
\begin{gathered}
\xymatrix{
\R^e\times\R\times\sph^{d-e-1}\ar[d]_{g:=u\circ\widehat{\pi}\circ\Phi}\ar[rr]^(0.6){\Phi}_(0.6){\cong}&&\widehat{M}\ar[d]^{\widehat{\pi}}&(y,\rho,w)\ar@{|->}[d]\ar@{|->}[r]&\Phi(y,\rho,w)\ar@{|->}[d]\\
\R^d&&\ar[ll]_{u}^{\cong}M&(y,\rho w)&\psi(y,\rho w)\ar@{|->}[l]
}
\end{gathered}
\end{equation}

\begin{fact}\cite[\S5.A.5]{fe3}
As a consequence, we have: \em The Nash maps $\pi_\epsilon$ and $\widehat{\pi}$ have local representations 
\begin{equation*}
(x_1,\ldots,x_d)\mapsto (x_1,\ldots,x_e,x_{e+1},x_{e+1}x_{e+2},\ldots,x_{e+1}x_d)
\end{equation*}
in an open neighborhood of each point $p\in\Rr$. In addition, $d\pi_p(T_p\widehat{M})\not\subset T_{\pi(p)}N$ \em if $e<d-1$.
\end{fact} 

\begin{fact}\label{bigstepa6}
{\em Denote $g:=u\circ\widehat{\pi}\circ\Phi$ and $g_+:=g|_{\R^e\times [0,+\infty)\times\sph^{d-e-1}}=u\circ\pi_+\circ\Phi|_{\R^e\times [0,+\infty)\times\sph^{d-e-1}}$. Consider the Nash normal-crossings divisor $Z:=\{\y_{e+1}\cdots\y_d=0\}\subset\R^d$. Consider coordinates $(w_{e+1},\ldots,w_d)$ in $\R^{d-e}$ and the sphere $\sph^{d-e-1}:=\{\w_{e+1}^2+\cdots+\w_d^2=1\}$.

(i) Write $Z_k:=\{\y_k=0\}$ for $k=e+1,\ldots,d$ and observe that
\begin{align*}
g^{-1}(Z_k)&=(\R^e\times\{0\}\times\sph^{d-e-1})\cup(\R^e\times\R\times(\sph^{d-e-1}\cap\{\w_{k}=0\})),\\
g_+^{-1}(Z_k)&=(\R^e\times\{0\}\times\sph^{d-e-1})\cup(\R^e\times[0,+\infty)\times(\sph^{d-e-1}\cap\{\w_{k}=0\}))
\end{align*}
for $k=e+1,\ldots,d$. Thus,
$$
g^{-1}(Z)=(\R^e\times\{0\}\times\sph^{d-e-1})\cup\bigcup_{k=e+1}^d(\R^e\times\R\times(\sph^{d-e-1}\cap\{\w_{k}=0\}))
$$
is a Nash normal-crossings divisor and
$$
g_+^{-1}(Z)=(\R^e\times\{0\}\times\sph^{d-e-1})\cup\bigcup_{k=e+1}^d(\R^e\times[0,+\infty)\times(\sph^{d-e-1}\cap\{\w_{k}=0\})).
$$
The Nash closure of $g_+^{-1}(Z)$ is $g^{-1}(Z)$.

(ii) Let $\epsilon:=(\epsilon_{e+1},\ldots,\epsilon_d)$ where $\epsilon_k=\pm1$ and denote $\Qq_\epsilon:=\{\epsilon_{e+1}\y_{e+1}\geq0,\ldots,\epsilon_d\y_d\geq0\}$. Write $-\epsilon:=(-\epsilon_{e+1},\ldots,-\epsilon_d)$. We have:
\begin{equation*}
\cl(g_+^{-1}(\Qq_\epsilon\setminus Z))=\R^e\times[0,+\infty)\times(\sph^{d-e-1}\cap\{\epsilon_{e+1}\w_{e+1}\geq0,\ldots,\epsilon_d\w_{d}\geq0\}).
\end{equation*}
Consequently, 
$$
\cl(g_+^{-1}(\Qq_\epsilon\setminus Z))\cap\cl(g_+^{-1}(\Qq_{-\epsilon}\setminus Z))=\R^e\times[0,+\infty)\times(\sph^{d-e-1}\cap\{\w_{e+1}=0,\ldots,\w_{d}=0\})=\varnothing, 
$$
because $\sph^{d-e-1}=\{\w_{e+1}^2+\cdots+\w_d^2=1\}$. In addition, if 
$$
\epsilon:=(\epsilon_{e+1},\ldots,\epsilon_{m},\epsilon_{m+1},\ldots,\epsilon_d)\ \text{and}\ \epsilon':=(\epsilon_{e+1},\ldots,\epsilon_{m},-\epsilon_{m+1},\ldots,-\epsilon_d)
$$ 
where $e<m<d$, then 
$$
\Qq_\epsilon\cap\Qq_{\epsilon'}=\{\epsilon_{e+1}\y_{e+1}\geq0,\ldots,\epsilon_m\y_m\geq0,\y_{m+1}=0,\ldots,\y_d=0\},
$$
which has dimension $e+(d-e)-(d-m)=m\geq e+1$. In addition,
{\small\begin{multline*}
\cl(g_+^{-1}(\Qq_\epsilon\setminus Z))\cap\cl(g_+^{-1}(\Qq_{\epsilon'}\setminus Z))\\
=\R^e\times[0,+\infty)\times(\sph^{d-e-1}\cap\{\epsilon_{e+1}\w_{e+1}\geq0,\ldots,\epsilon_m\w_{m}\geq0,\w_{m+1}=0,\ldots,\w_{d}=0\}),
\end{multline*}}which has dimension $e+1+(d-e-1-(d-m))=m\geq e+1$.

(iii) Let $Y_1,Y_2$ be intersections of dimension $e+1$ of irreducible components of $Z$ that contain $N$. We may assume $Y_1=\{\y_{e+1}=0,\ldots,\y_{d-1}=0\}$ and $Y_2=\{\y_{e+1}=0,\ldots,\y_{d-2}=0,\y_d=0\}$, so $Y_1\cap Y_2=\{\y_{e+1}=0,\ldots,\y_d=0\}=N$. Thus, 
\begin{align*}
g^{-1}(Y_1)&=(\R^e\times\{0\}\times\sph^{d-e-1})
\cup(\R^e\times\R\times(\sph^{d-e-1}\cap\{\w_{e+1}=0,\ldots,\w_{d-1}=0\})),\\
g_+^{-1}(Y_1)&=(\R^e\times\{0\}\times\sph^{d-e-1})\cup(\R^e\times[0,+\infty)\times(\sph^{d-e-1}\cap\{\w_{e+1}=0,\ldots,\w_{d-1}=0\})),\\
g^{-1}(Y_2)&=(\R^e\times\{0\}\times\sph^{d-e-1})\cup(\R^e\times\R\times(\sph^{d-e-1}\cap\{\w_{e+1}=0,\ldots,\w_{d-2}=0,\w_d=0\})),\\
g_+^{-1}(Y_2)&=(\R^e\times\{0\}\times\sph^{d-e-1})\cup(\R^e\times[0,+\infty)\times(\sph^{d-e-1}\cap\{\w_{e+1}=0,\ldots,\w_{d-2}=0,\w_d=0\})).
\end{align*}
As the intersection
\begin{equation*}
\sph^{d-e-1}\cap\{\w_{e+1}=0,\ldots,\w_{d-1}=0\}\cap\{\w_{e+1}=0,\ldots,\w_{d-2}=0,\w_d=0\}
\end{equation*}
is empty, we conclude that the intersection 
\begin{multline*}
\cl(g^{-1}(Y_1\setminus N))\cap\cl(g^{-1}(Y_2\setminus N))\\
=(g^{-1}(Y_1)\cap\cl(g^{-1}(Y_1\setminus N)))\cap(g^{-1}(Y_2)\cap\cl(g^{-1}(Y_2\setminus N)))
\end{multline*} 
is also empty. Analogously, the intersection 
\begin{multline*}
\cl(g_+^{-1}(Y_1\setminus N))\cap\cl(g_+^{-1}(Y_2\setminus N))\\
=(g_+^{-1}(Y_1)\cap\cl(g_+^{-1}(Y_1\setminus N)))\cap(g_+^{-1}(Y_2)\cap\cl(g_+^{-1}(Y_2\setminus N)))
\end{multline*}
is empty}.
\hfill$\sqbullet$
\end{fact}

Next, we analyze the properties of the local structure of drilling blow-up when $N$ has dimension $d-1$.

\begin{fact}\em\label{locald-1}
Assume $N$ has dimension $e=d-1$. The Nash diffeomorphism $u:=(u_1,\ldots,u_d):M\to\R^d$ satisfies $N=\{u_d=0\}$. Recall that $\psi:=u^{-1}:\R^d\equiv\R^{d-1}\times\R\to M$, $\zeta_d:\R^d\to\R^k$ is a Nash map that does not take the value $0\in\R^k$ and define the Nash map
$$
\Phi:\R^{d-1}\times\R\times\{\pm1\}\to M\times\sph^{k-1},\ (y,\rho,\pm1)\mapsto\Big(\psi(y,\pm\rho),\pm\frac{\zeta_d(y,\pm\rho)}{\|\zeta_d(y,\pm\rho)\|}\Big).
$$ 
Fix $\epsilon=\pm$ and denote 
$$
I_\epsilon:=\begin{cases}
[0,+\infty)&\text{ if $\epsilon=+$,}\\
(-\infty,0]&\text{ if $\epsilon=-$.}
\end{cases}
$$ 
The closure $\widetilde{M}_\epsilon$ in $M\times\sph^{k-1}$ of the set
$$
\Gamma_\epsilon:=\Big\{\Big(\psi(y,z),\epsilon\frac{z}{|z|}\frac{\zeta_d(y,z)}{\|\zeta_d(y,z)\|}\Big)\in M\times\sph^{k-1}:\,z\neq 0\Big\}
$$ 
is a Nash manifold with boundary such that:
\begin{itemize}
\item[{\rm(i)}] $\widetilde{M}_\epsilon\subset\im(\Phi)$.
\item[{\rm(ii)}] The restriction of $\Phi$ to $\R^{d-1}\times I_\epsilon\times\{\pm1\}$ induces a Nash diffeomorphism between $\R^{d-1}\times I_\epsilon\times\{\pm1\}$ and $\widetilde{M}_\epsilon$. Consequently, $\partial\widetilde{M_\epsilon}=\Phi(\R^{d-1}\times\{0\}\times\{\pm1\})$ and $\Gamma_\epsilon=\Int(\widetilde{M_\epsilon})=\Phi(\R^{d-1}\times(I_\epsilon\setminus\{0\})\times\{\pm1\})$.
\end{itemize}

Denote $\Rr:=\partial\widetilde{M}_+=\partial\widetilde{M}_-$ and $\widehat{M}:=\widetilde{M}_+\cup\widetilde{M}_-=\Gamma_+\sqcup\Rr\sqcup\Gamma_-$. Then $\Phi$ induces a Nash diffeomorphism between $\R^{d-1}\times\R\times\{\pm1\}$ and $\widehat{M}$, which is the Nash closure of $\widetilde{M}_+$ and $\widetilde{M}_-$ in $M\times\sph^{k-1}$. In addition, the Nash map $\sigma: M\times\sph^{k-1}\to M\times\sph^{k-1}, (a,b)\to(a,-b)$ induces a Nash involution on $\widehat{M}$ without fixed points such that $\sigma(\widetilde{M}_+)=\widetilde{M}_-$ and $\Phi(y,-\rho,\pm1)=(\sigma\circ\Phi)(y,\rho,\mp1)$ for each $(y,\rho,\pm1)\in\R^{d-1}\times\R\times\{\pm1\}$.

Consider the projection $\pi:M\times\sph^{k-1}\to M$ onto the first factor and denote $\pi_\epsilon:=\pi|_{\widetilde{M_\epsilon}}$. Then
\begin{itemize}
\item[{\rm(i)}] $\pi_\epsilon$ is proper, $\pi_\epsilon(\widetilde{M_\epsilon})=M$ and $\Rr=\pi^{-1}_\epsilon(N)$.
\item[{\rm(ii)}] The restriction $\pi_\epsilon|_{\Gamma_\epsilon}:\Gamma_\epsilon\to M\setminus N$ is a Nash diffeomorphism.
\item[{\rm(iii)}] For each $q\in N$ it holds $\pi_\epsilon^{-1}(q)=\{q\}\times\left\{\pm\frac{(\zeta_d\circ u)(q)}{\|\{(\zeta_d\circ u)(q)\}\|}\right\}$, that is, each point $q\in N$ has exactly two preimages under $\pi_\epsilon$.
\end{itemize}

Denote $\widehat{\pi}:=\pi|_{\widehat{M}}$ and consider the commutative diagram.
\begin{equation}\label{diagpol2}
\begin{gathered}
\xymatrix{
\R^{d-1}\times\R\times\{\pm1\}\ar[d]_{g:=u\circ\widehat{\pi}\circ\Phi}\ar[rr]^(0.6){\Phi}_(0.6){\cong}&&\widehat{M}\ar[d]^{\widehat{\pi}}&(y,\rho,\pm1)\ar@{|->}[d]\ar@{|->}[r]&\Phi(y,\rho,\pm1)\ar@{|->}[d]\\
\R^d&&\ar[ll]_{u}^{\cong}M&(y,\pm\rho)&\psi(y,\pm\rho)\ar@{|->}[l]
}
\end{gathered}
\end{equation}
Consequently: \em The Nash maps $\pi_\epsilon$ and $\widehat{\pi}$ have local representations $(x_1,\ldots,x_d)\mapsto (x_1,\ldots,x_d)$ in an open neighborhood of each point $p\in\Rr$. In addition, $d\pi_p(T_p\widehat{M})= T_{\pi(p)}M$\em. 
\end{fact}

\subsection{Global definition.}\label{gdndbu}
Let $M\subset\R^m$ be a $d$-dimensional Nash manifold and $N\subset M$ a closed $e$-dimensional Nash submanifold. Let $f_1,\ldots,f_k\in{\mathcal N}(M)$ be a finite system of generators of the ideal $\mathcal{I}(N)$ of Nash functions on $M$ vanishing identically on $N$. Consider the Nash map
\begin{equation*}
F:M\setminus N\to\sph^{k-1}, x\mapsto\frac{(f_1(x),\ldots,f_k(x))}{\|(f_1(x),\ldots,f_k(x))\|}.
\end{equation*}
We have:

\begin{fact}\label{bigstepa2g}\cite[\S5.B.1]{fe3} 
Fix $\epsilon=\pm$. The closure $\widetilde{M}_\epsilon$ in $M\times\sph^{k-1}$ of the graph 
$$
\Gamma_\epsilon:=\{(x,\epsilon F(x))\in M\times\sph^{k-1}:\ x\in M\setminus N\}
$$ 
is a Nash manifold with boundary. Denote $\Rr:=\partial\widetilde{M}_+=\partial\widetilde{M}_-$ and $\widehat{M}:=\widetilde{M}_+\cup\widetilde{M}_-=\Gamma_+\sqcup\Rr\sqcup\Gamma_-$, which is the Nash closure of $\widetilde{M}_+$ and $\widetilde{M}_-$ in $M\times\sph^{k-1}$ if $M$ is connected. In addition, $\widehat{M}$ is a Nash manifold and the Nash map $\sigma:M\times\sph^{k-1}\to M\times\sph^{k-1},\, (a,b)\to(a,-b)$ induces a Nash involution on $\widehat{M}$ without fixed points that maps $\widetilde{M}_+$ onto $\widetilde{M}_-$.
\end{fact}

\begin{fact}\label{bigstepb2}\cite[\S5.B.2]{fe3} 
Consider the projection $\pi:M\times\sph^{k-1}\to M$ onto the first factor. Denote $\pi_\epsilon:=\pi|_{\widetilde{M_\epsilon}}$ for $\epsilon=\pm$ and $\widehat{\pi}:=\pi|_{\widehat{M}}$. We have:
\begin{itemize}
\item[(i)] $\pi_\epsilon$ is proper, $\pi_\epsilon(\widetilde{M_\epsilon})=M$ and $\Rr=\pi^{-1}_\epsilon(N)$.
\item[(ii)] The restriction $\pi_\epsilon|_{\Gamma_\epsilon}:\Gamma_\epsilon\to M\setminus N$ is a Nash diffeomorphism.
\item[(iii)] Consider the Nash map $f:=(f_1,\ldots,f_k):M\to\R^k$ (whose components generate $\mathcal{I}(N)$). Fix $q\in N$ and let $E_q$ be any complementary linear subspace of $T_qN$ in $T_qM$. Then $\pi_\epsilon^{-1}(q)=\{q\}\times\sph_q^{d-e-1}$, where $\sph_q^{d-e-1}$ denotes the sphere of dimension $d-e-1$ obtained when intersecting $\sph^{k-1}$ with the $(d-e)$-dimensional linear subspace $d_qf(E_q)$. In case $e=d-1$, each $q\in N$ has exactly two preimages under $\pi_\epsilon$.
 \item[(iv)] The Nash maps $\pi_\epsilon$ and $\widehat{\pi}$ have local representations of the type 
$$
(x_1,\ldots,x_d)\mapsto(x_1,\ldots,x_e,x_{e+1},x_{e+1}x_{e+2},\ldots,x_{e+1}x_d)
$$
in an open neighborhood of each point $p\in\Rr$. In case $e=d-1$ the previous local representations are $(x_1,\ldots,x_d)\mapsto(x_1,\ldots,x_d)$. In addition, $
d\pi_p(T_p\widehat{M})\not\subset T_{\pi(p)}N$ and if $e=d-1$, we have $d\pi_p(T_p\widehat{M})=T_{\pi(p)}M$.
\end{itemize}
\end{fact}

\begin{fact}\label{bigstepb3}\cite[\S5.B.3]{fe3}
Up to Nash diffeomorphisms compatible with the respective projections, the pairs $(\widetilde{M_\epsilon},\pi_\epsilon)$ and $(\widehat{M},\widehat{\pi})$ do not depend on the generators $f_1,\ldots,f_k$ of $\mathcal{I}(N)$. Moreover, such Nash diffeomorphisms are unique.
\end{fact} 
 
\begin{defn}
The pair $(\widetilde{M}_+,\pi_+)$ is the \em drilling blow-up \em of the Nash manifold $M$ with center the closed Nash submanifold $N\subset M$ and $(\widehat{M},\widehat{\pi})$ is the \em twisted Nash double \em of $(\widetilde{M}_+,\pi_+)$.\hfill$\sqbullet$ 
\end{defn}

\subsection{Relationship between drilling blow-up and classical blow-up.}\label{dbme}
Let $M\subset\R^m$ be a $d$-dimensional Nash manifold and $N\subset M$ a closed $e$-dimensional Nash submanifold. Let $f_1,\ldots,f_k\in{\mathcal N}(M)$ be a system of generators of the ideal $\mathcal{I}(N)$. Define
$$
\Gamma':=\{(x,[f_1(x):\ldots:f_k(x)])\in M\times\R\PP^{k-1}:\ x\in M\setminus N\}.
$$
The closure $B(M,N)$ of $\Gamma'$ in $M\times\R\PP^{k-1}$ together with the restriction $\pi'$ to $B(M,N)$ of the projection $M\times\R\PP^{k-1}\to M$ is the classical \em blow-up \em of $M$ with center $N$ (in the Nash setting). 

\begin{cor}\label{bu}\cite[Cor.3.12]{fe3} 
Let $(\widehat{M},\widehat{\pi})$ be the twisted Nash double of the drilling blow-up $(\widetilde{M}_+,\pi_+)$. Let $\sigma:\widehat{M}\to\widehat{M},\ (a,b)\mapsto(a,-b)$ be the involution of $\widehat{M}$ without fixed points. Consider the Nash map $\Theta:M\times\sph^{k-1}\to M\times\R\PP^{k-1}, (p,q)\to (p,[q])$ and its restriction $\theta:=\Theta|_{\widehat{M}}:\widehat{M}\to B(M,N)$. We have:
\begin{itemize}
\item[(i)] $\theta(\widehat{M})=B(M,N)$, $\theta\circ\sigma=\theta$, $\pi'\circ\theta=\widehat{\pi}$ and $\theta^{-1}(a,[b])=\{(a,b),(a,-b)\}$ for each $(a,[b])\in B(M,N)$. 
\item[(ii)] $\theta$ is an unramified $2$ to $1$ Nash covering of $B(M,N)$.
\end{itemize}
\end{cor}

\subsection{Algebraic description of drilling blow-up.}\label{addbu}
We analyze a general enough situation for which we can guarantee that the twisted (Nash) double of the drilling blow-up is a non-singular real algebraic set. Let $X\subset\R^n$ be a $d$-dimensional non-singular (pure dimensional) real algebraic set and $Y\subset X$ a $e$-dimensional non-singular (pure dimensional) real algebraic subset. Let $f_1,\ldots,f_r\in\R[\x]:=\R[\x_1,\ldots,\x_n]$ be a system of generators of the ideal $\mathcal{I}(Y)$ of polynomials vanishing identically on $Y$ and denote $u:=(u_1,\ldots,u_r)\in\sph^{r-1}$. Fix $\epsilon=\pm$ and let $\widetilde{X}_\epsilon$ be the (topological) Euclidean closure of
$$
\Gamma_\epsilon:=\Big\{(x,u)\in(X\setminus Y)\times\sph^{r-1}:\ 
{\rm rk}\begin{pmatrix}
u_1&\cdots&u_r\\
f_1(x)&\cdots&f_r(x)
\end{pmatrix}=1,\, \epsilon(u_1f_1(x)+\cdots+u_rf_r(x))>0\Big\}
$$
in $X\times\sph^{r-1}$. Let $\pi:X\times\sph^{r-1}\to X$ be the projection (onto the first factor) and $\widehat{X}$ the (topological) Euclidean closure of
$$
\Gamma:=\Big\{(x,u)\in (X\setminus Y)\times\sph^{r-1}:\ 
{\rm rk}\begin{pmatrix}
u_1&\cdots&u_r\\
f_1(x)&\cdots&f_r(x)
\end{pmatrix}=1
\Big\}.
$$
Observe that $\widehat{X}=\widetilde{X}_+\cup\widetilde{X}_-$ is the union of the irreducible components of the real algebraic set
$$
\Big\{(x,u)\in X\times\sph^{r-1}:\ 
{\rm rk}\begin{pmatrix}
u_1&\cdots&u_r\\
f_1(x)&\cdots&f_r(x)
\end{pmatrix}=1
\Big\}
$$
different from $Y\times\sph^{r-1}$. If we denote $\widehat{\pi}:=\pi|_{\widehat{X}}$, it holds that $(\widehat{X},\widehat{\pi})$ is the twisted Nash double of the drilling blow-up $(\widetilde{X}_+,\pi_+)$ of $X$ with center $Y$ (where $\pi_+:=\pi|_{\widetilde{X}_+}$). Proceeding similarly to the proof of \cite[Prop.3.5.11(i)]{bcr}, one shows that $\widehat{X}$ is a non-singular real algebraic set. In addition, proceeding similarly to the proof of \cite[Prop.3.5.11(ii)]{bcr}, one shows that there exists a finite open cover of $Y$ by Zariski open subsets $U_i$ such that $\pi^{-1}(U_i)$ is Nash diffeomorphic to $U_i\times\sph^{d-e-1}$ via $\pi$ (use also Fact \ref{bigstepb2}(iii)). It holds $\widetilde{X}_\epsilon=\widehat{X}\cap\{\epsilon(u_1f_1+\cdots+u_rf_r)\geq0\}$ and the Zariski closure of $\widetilde{X}_\epsilon$ is contained in $\widehat{X}$.

If $\Ss\subset X$ is a pure dimensional semialgebraic set of dimension $d$, then $\Reg(\Ss)=\Sth(\Ss)$ and $\Tt:=\cl(\pi_+^{-1}(\Ss\setminus Y))\cap\pi_+^{-1}(\Ss)$ is a pure dimensional semialgebraic set of dimension $d$ such that $\Reg(\Tt)=\Sth(\Tt)$. This is because both $X$ and $\widehat{X}$ are non-singular real algebraic sets of dimension $d$ and $\pi_+^{-1}|_{\widetilde{X}_+}:\widetilde{X}_+\setminus\pi_+^{-1}(Y)\to X\setminus Y$ is a Nash diffeomorphism.

\begin{remark}
If $X$ is irreducible, then $\widehat{X}$ does not need to be irreducible (even if $X$ is connected). 

Consider for instance the plane $X:=\R^2$ and the line $Y:=\{\x=0\}$. Then $\mathcal{I}(Y)=(\x)$ and $\widehat{X}=\{\z=-1\}\cup\{\z=1\}\subset\R^3$, which is disconnected. However, $\widehat{X}$ is the Zariski closure of $\widetilde{X}_+=\{\x\geq0,\z=1\}\cup\{\x\leq 0,\z=-1\}$ (resp. $\widetilde{X}_-=\{\x\leq0,\z=1\}\cup\{\x\geq0,\z=-1\}$).
\hfill$\sqbullet$
\end{remark}

However, if $X$ is irreducible, at least we have the following result that we prove below:

\begin{lem}\label{irredx}
If $X$ is irreducible, then $\widehat{X}$ is the Zariski closure of $\widetilde{X}_\epsilon$ for $\epsilon=\pm1$.
\end{lem}

\subsubsection{Relationship with algebraic blow-up}\label{missing}
Denote $[z]:=[z_1:\cdots:z_r]\in\R\PP^{r-1}$ and define
$$
\Gamma':=\{(x,[z])\in (X\setminus Y)\times\R\PP^{r-1}:\ {\rm rk}\begin{pmatrix}
z_1&\cdots&z_r\\
f_1(x)&\cdots&f_r(x)
\end{pmatrix}=1\}.
$$
The Zariski closure $B(X,Y)$ of $\Gamma'$ in $X\times\R\PP^{r-1}$ together with the restriction $\pi'$ to $B(X,Y)$ of the projection $X\times\R\PP^{r-1}\to X$ (onto the first factor) is the (algebraic) \em blow-up \em of $X$ with center~$Y$. 

\begin{remarks}\label{closzar}
(i) If $X\subset\R^n$ is a $d$-dimensional non-singular (pure dimensional) real algebraic set and $Y\subset X$ is a $e$-dimensional non-singular (pure dimensional) algebraic subset, then $B(X,Y)$ is in fact the (topological) Euclidean closure of $\Gamma'$ in $X\times\R\PP^{r-1}$. 

Indeed, by \cite[Prop.3.5.11]{bcr} there exists a finite cover of $Y$ by Zariski open subsets $U_i$, such that $\pi'^{-1}(U_i)$ is biregularly diffeomorphic to $U_i\times\R\PP^{d-e-1}$. In fact, it is shown in the proof of \cite[Prop.3.5.11 (p.80)]{bcr} that for each $y\in U_i$ the inverse image $\pi'^{-1}(y)=\{y\}\times\R\PP^{d-e-1}$ is contained in the (topological) Euclidean closure of $\Gamma'$. As $B(X,Y)=\Gamma'\cup\pi'^{-1}(Y)=\Gamma'\cup\bigcup_i\pi'^{-1}(U_i)\subset\cl(\Gamma')$, we conclude that $B(X,Y)=\cl(\Gamma')$, as required.

(ii) Consider the polynomial map $\Theta:X\times\sph^{r-1}\to X\times\R\PP^{r-1}, (x,z)\to (x,[z])$ and its restriction $\theta:=\Theta|_{\widehat{X}}:\widehat{X}\to B(X,Y)$. We have:
\begin{itemize}
\item[(1)] $\theta(\widehat{X})=B(X,Y)$, $\pi'\circ\theta=\widehat{\pi}$ and $\theta^{-1}(x,[z])=\{(x,z),(x,-z)\}$ for each $(x,[z])\in B(X,Y)$. 
\item[(2)] $\theta$ is an unramified $2$ to $1$ polynomial covering map of $B(X,Y)$.\hfill$\sqbullet$
\end{itemize}
\end{remarks}

\begin{proof}[Proof of Lemma \em \ref{irredx}]
As $X$ is irreducible, $T:=B(X,Y)$ is irreducible (because $X\setminus Y$ and $T\setminus\pi'^{-1}(Y)$ are biregularly diffeomorphic via $\pi'|_{T\setminus\pi'^{-1}(Y)}$ \cite[Prop.3.5.8]{bcr}). The polynomial map $\theta:\widehat{X}\to T$ is an unramified $2$ to $1$ polynomial covering map. If $\widehat{X}$ is irreducible, $\widehat{X}$ is the Zariski closure of $\widetilde{X}_\epsilon$ for $\epsilon=\pm1$, because $\dim(\widehat{X})=\dim(\widetilde{X}_\epsilon)$. Thus, let us assume $\widehat{X}$ is reducible. Let $Z_1,\ldots,Z_\ell$ be irreducible components of $\widehat{X}=\widetilde{X}_+\cup\widetilde{X}_-$ and fix $1\leq i\leq\ell$. 

As $\widehat{X}$ is pure dimensional and non-singular, we deduce $Z_i$ is pure dimensional and non-singular, so it is a union of connected components of $\widehat{X}$. Thus, $Z_i$ is an open and closed subset of $\widehat{X}$. As $\theta$ is an open and closed map and each $Z_i$ is an open and closed subset of $\widehat{X}$, we deduce $\theta(Z_i)$ is a union of connected components of $T$ and the fibers of $\theta|_{Z_i}:Z_i\to\theta(Z_i)$ have cardinality either $1$ or $2$. We claim: {\em All the fibers of the restriction map $\theta|_{Z_i}:Z_i\to\theta(Z_i)$ have the same cardinality, which is either $1$ or $2$.}

Let $\widetilde{\widehat{X}}$ be the complexification of $\widehat{X}$, $\widetilde{Z}_i$ the complexification of $Z_i$ and $\widetilde{T}$ the complexification of $T$. Observe that $\widetilde{\widehat{X}}=\bigcup_{i=1}^\ell\widetilde{Z}_i$. Consider the rational extension $\widetilde{\theta}:\widetilde{\widehat{X}}\to\widetilde{T}$ of $\theta$ to $\widetilde{\widehat{X}}$. Its image is contained in $\widetilde{T}$, because $\theta(\widehat{X})=T$. Observe that $\widetilde{\theta}|_{\widetilde{Z}_i}$ is a dominant rational map, because $\theta$ is a local diffeomorphism and $Z_i$ is a union of connected components of $\widehat{X}$. In particular $\widetilde{\theta}(\widetilde{Z}_i)$ is a Zariski open subset of the irreducible algebraic set $\widetilde{T}$.

Denote with ${\mathcal R}(\widetilde{Z}_i)$ the field of rational functions on $\widetilde{Z}_i$ and with ${\mathcal R}(\widetilde{T})$ the field of rational functions on $\widetilde{T}$. The map $\widetilde{\theta}|_{\widetilde{Z}_i}^*:{\mathcal R}(\widetilde{T})\to{\mathcal R}(\widetilde{Z}_i),\ f\mapsto f\circ\widetilde{\theta}|_{\widetilde{Z}_i}$ is a homomorphism of fields of the same transcendence degree $d=\dim(X)$ over $\C$. Consequently, ${\mathcal R}(\widetilde{Z}_i)$ is an algebraic extension of ${\mathcal R}(\widetilde{T})$ of finite degree $m_i$. By \cite[Prop.7.16]{ha} the number of points in each generic fiber of $\widetilde{\theta}|_{\widetilde{Z}_i}$ is equal to $m_i$. As the difference $(\widetilde{\theta}|_{\widetilde{Z}_i})^{-1}(q)\setminus(\theta|_{Z_i})^{-1}(q)$ has even cardinality for each $q\in\theta(Z_i)$ and $1$ and $2$ have different parity, generic fibers of $\theta|_{Z_i}:Z_i\to\theta(Z_i)\subset T$ have the same cardinality, which is either $1$ or $2$. As $\theta$ is an unramified $2$ to $1$ polynomial covering map and for each connected component of $\theta(Z_i)$ there are generic fibers, we conclude that the claim holds. 

Thus, each $\theta|_{Z_i}:Z_i\to\theta(Z_i)\subset T$ is either a polynomial map that is a Nash diffeomorphism or an unramified $2$ to $1$ polynomial covering map for $i=1,\ldots,\ell$ and we distinguish the situations that appear:

\noindent{\sc Case 1.} If $\theta|_{Z_i}:Z_i\to\theta(Z_i)\subset T$ is a polynomial diffeomorphism for some $i=1,\ldots,\ell$, we claim: {\em $\theta(Z_i)=T$, $\ell=2$ and $\theta|_{Z_j}:Z_j\to T$ is a polynomial map that is a Nash diffeomorphism for $j=1,2$}.

Suppose that $\theta(Z_i)\neq T$ for some $i=1,\ldots,\ell$. Then there exists a connected component $M$ of $T$ such that $\theta(Z_i)\cap M=\varnothing$ (recall that $\theta(Z_i)$ is a union of connected components of $T$). As $T$ is pure dimensional of dimension $d$, we deduce that $M$ has dimension $d$. As $\theta|_{Z_i}:Z_i\to\theta(Z_i)$ is a polynomial diffeomorphism, we deduce that $m_i$ is odd. Pick a generic point $q\in M$ such that $(\widetilde{\theta}|_{\widetilde{Z}_i})^{-1}(q)$ has exactly $m_i$ points. As the cardinality of the difference $(\widetilde{\theta}|_{\widetilde{Z}_i})^{-1}(q)\setminus(\theta|_{Z_i})^{-1}(q)$ is even and $m_i$ is odd, the cardinality of $(\theta|_{Z_i})^{-1}(q)=\varnothing$ is also odd, which is a contradiction. Consequently, $\theta(Z_i)=T$, as claimed.

As $\theta:\widehat{X}\to T$ is an unramified $2$ to $1$ polynomial covering map, we deduce that $\theta$ is trivial, that is, $\ell=2$ and $\theta|_{Z_j}:Z_j\to T$ is a polynomial map diffeomorphism for $j=1,2$. Pick a point $y\in(\pi')^{-1}(Y)\subset T$. Observe that $\theta^{-1}(y)\subset\widetilde{X}_+\cap \widetilde{X}_-$, so $\widetilde{X}_\epsilon\cap Z_i\neq\varnothing$ for $\epsilon=\pm$ and $i=1,2$. As $\widetilde{X}_\epsilon$ is pure dimensional of dimension $d$ and each $Z_i$ is pure dimensional of dimension $d$ and open and closed in $\widetilde{X}_+\cup\widetilde{X}_-=\widehat{X}=Z_1\sqcup Z_2$, we conclude that each $Z_j$ is contained in the Zariski closure of $\widetilde{X}_\epsilon$, so $\widehat{X}$ is the Zariski closure of $\widetilde{X}_\epsilon$ for $\epsilon=\pm$.

\noindent{\sc Case 2.} $\theta|_{Z_i}:Z_i\to\theta(Z_i)\subset T$ is an unramified $2$ to $1$ polynomial covering map for each $i=1,\ldots,\ell$. Thus, $Z_i\cap\widetilde{X}_\epsilon\neq\varnothing$ and $Z_i$ is contained in the intersection of the Zariski closures of $\widetilde{X}_\epsilon$ for $\epsilon=\pm1$. Consequently, the Zariski closure of each $\widetilde{X}_\epsilon$ is $\bigcup_{i=1}^\ell Z_i=\widehat{X}$ for $\veps=\pm1$, as required.
\end{proof}
\begin{remark}
Although $\theta$ is a surjective polynomial map, we cannot guarantee, a priori in {\sc Case 2.} of the previous proof, that $\theta(Z_i)=T$ for some $i=1,\ldots,\ell$ (and consequently that $\ell=1$) if $T$ is not connected. This is so, because some of the connected components of $T$ can be covered by couples of complex conjugated points of $\widetilde{Z}_i$. As we check in the following example, the only fact we can assure is that the complex image $\widetilde{\theta}(\widetilde{Z}_i)$ is a Zariski dense subset of the irreducible algebraic set $\widetilde{T}$.\hfill$\sqbullet$ 
\end{remark}
\begin{example}
Consider for instance the real non-singular algebraic sets $T:=\{\y^2-\x^2-1=0\}\subset\R^2$ (which is irreducible) and $Z:=\{\y^2-\x^4-1=0,\z^2=1\}\subset\R^3$ (which is reducible) and the polynomial map $\theta:Z\to T,\ (x,y,z)\mapsto(zx^2,y)$, which is a $2$ to $1$ polynomial covering map. Observe that $Z$ has two irreducible components $Z_{\epsilon}:=\{\y^2-\x^4-1=0,\z=\epsilon\}$ for $\epsilon=\pm1$ and the restriction $\theta|_{Z_\epsilon}:Z_\epsilon\to T$ is a $2$ to $1$ polynomial covering that only covers the connected component $T_\epsilon=T\cap\{\epsilon\x>0\}$ of $T$ for $\epsilon=\pm1$. However, the complexification of $\widetilde{\theta}:\widetilde{Z}=\widetilde{Z}_{-1}\cup\widetilde{Z}_{+1}\to\widetilde{T}$ of $\theta$ is a $4$ to $1$ polynomial covering map and the restriction $\widetilde{\theta}|_{\widetilde{Z}_\epsilon}:\widetilde{Z}_\epsilon\to\widetilde{T}$ is a $2$ to $1$ polynomial covering map for $\epsilon=\pm1$.
\hfill$\sqbullet$ 
\end{example}

\subsection{Strict transforms.}\label{exst0}
Let $g:\Ss\to\Tt$ be a Nash map between semialgebraic sets $\Ss\subset\R^m$ and $\Tt\subset\R^n$. Let $\Rr\subset\Tt$ be a closed semialgebraic subset of $\Tt$ of striclty smaller dimension than $\dim(\Tt)$ and suppose that $\Tt\setminus\Rr$ and $\Ss\setminus g^{-1}(\Rr)$ are Nash manifolds and $g|_{\Ss\setminus g^{-1}(\Rr)}:\Ss\setminus g^{-1}(\Rr)\to\Tt\setminus\Rr$ is a Nash diffeomorphism. Let $\Aa\subset\Tt$ be a semialgebraic subset of $\Tt$ such that $\Aa\cap\Rr$ has strictly smaller dimension than $\Aa$. We define the {\em strict transform of $\Aa$ under $g$ with respect to $\Rr$} as $\widehat{\Aa}:=g^{-1}(\Aa)\cap\cl(g^{-1}(\Aa\setminus\Rr))$. Of course, if $\Aa$ is closed in $\Tt$, we have $\widehat{\Aa}=\cl(g^{-1}(\Aa\setminus\Rr))$. We recall next some properties of the strict transform in some well-known situations. We will use such properties freely along the sequel.

\begin{remarks}\label{exst}
(i) Let $X\subset\R^n$ be a non-singular (pure dimensional) real algebraic set and $Y\subset X$ a non-singular (pure dimensional) algebraic subset of smaller dimension. Let $(B(X,Y),\pi')$ be the blow-up of $X$ of center $Y$. Let $Z\subset X$ be a non-singular (pure dimensional) real algebraic subset of $X$ that contains $Y$ and has strictly larger dimension than $Y$. Let $(B(Z,Y),\pi'')$ be the blow-up of $Z$ of center $Y$. By Remark \ref{closzar}(i) $B(Z,Y)$ coincides with the strict transform $\cl(\pi'^{-1}(Z\setminus Y))$ of $Z$ under $\pi'$ (with respect to $Y$). This means in particular that $\cl(\pi'^{-1}(Z\setminus Y))$ is a non-singular (pure dimensional) real algebraic set.

(ii) Let $M_0\subset\R^n$ be a Nash manifold of dimension $d$ and let $N\subset M_0$ be a closed submanifold of dimension $e<d$. Let $M_1\subset M_0$ be a Nash submanifold that contains $N$ and has strictly larger dimension. Let $(\widetilde{M}_{k,+},\pi_{k,+})$ be the drilling blow-up of $M_k$ of center $N$ and let $(\widehat{M}_k,\widehat{\pi}_k)$ be the twisted Nash double of $(\widetilde{M}_{k,+},\pi_{k,+})$ for $k=0,1$. Following the corresponding definitions one realizes that $\widetilde{M}_{1,+}$ coincides with the strict transform of $M_1$ under $\pi_{0,+}$ (with respect to $N$) and $\widehat{M}_1$ coincides with the strict transform of $M_1$ under $\widehat{\pi}_0$.

Indeed, one only has to observe that a finite system of generators $f_1,\ldots,f_k\in{\mathcal N}(M_1)$ of the ideal $\mathcal{I}(N)$ can be obtained considering a finite system of generators $g_1,\ldots,g_k\in{\mathcal N}(M_0)$ of the ideal $\mathcal{I}(N)$ and defining $f_j:=g_j|_{M_1}$ for $j=1,\ldots,k$.
\hfill$\sqbullet$
\end{remarks}

\section{Nash uniformization of closed chessboard sets}\label{s4}

In this section we prove Theorem \ref{red2}, but this requires some preparation. We begin proving that connected chessboard sets are connected by analytic paths.

\begin{lem}\label{conccccap}
Let $\Ss\subset\R^n$ be a chessboard set. Then the connected components of $\Ss$ coincide with the components of $\Ss$ connected by analytic paths.
\end{lem}
\begin{proof}
It is enough to check that if $\Ss$ is a connected $d$-dimensional chessboard set, then $\Ss$ is connected by analytic paths. As $\Ss$ is a chessboard set, $X:=\ol{\Ss}^{\zar}$ is a non-singular real algebraic set and there exists a normal-crossings divisor $Z$ of $X$ and connected components $\Cc_1,\ldots,\Cc_r$ of $X\setminus Z$ such that $\bigsqcup_{i=1}^r\Cc_i\subset\Ss\subset\bigcup_{i=1}^r\cl(\Cc_i)$. As $\Ss$ is connected, we may reorder the indices $i=1,\ldots,r$ such that there exists a point $p_i\in\cl(\Cc_i)\cap\Ss\cap\bigcup_{j=1}^{i-1}\cl(\Cc_j)$ for $i=1,\ldots,r$. As $\Cc_i\cap\cl(\Cc_j)=\varnothing$ if $i\neq j$, we deduce 
$$
p_i\in(\cl(\Cc_i)\setminus\Cc_i)\cap\Ss\cap\bigcup_{j=1}^{i-1}(\cl(\Cc_j)\setminus\Cc_j)\subset Z
$$
Thus, there exists an open neighborhood $U_i\subset X$ of $p_i$ and a Nash diffeomorphism $\psi_i:U_i\to\R^d$ such that $\psi_i(p)=0$ and $\psi_i(Z\cap U_i)=\{\x_1\cdots\x_m=0\}$ for some $1\leq m\leq d$. We may assume that $\{\x_1>0,\ldots,\x_m>0\}\subset\psi_i(\Cc_i\cap U_i)$ and there exists $j=1,\ldots,i-1$ and $s=0,\ldots,m$ such that $\{\x_1>0,\ldots,\x_s>0,-\x_{s+1}>0,\ldots,-\x_m>0\}\subset\psi_j(\Cc_j\cap U_j)$. Consider the Nash arc 
\begin{align*}
\alpha:[-1,1]&\to\{\x_1>0,\ldots,\x_m>0\}\cup\{\x_1>0,\ldots,\x_s>0,-\x_{s+1}>0,\ldots,-\x_m>0\}\cup\{0\},\\ 
t&\mapsto(t^2,\overset{(s)}{\ldots},t^2,t,\overset{(n-s)}{\ldots},t).
\end{align*}
Consequently, there exists a Nash arc between $\Cc_i$ and $\Cc_j$ for some $j=1,\ldots,i-1$. By \cite[Main Thm.1.4 \& Cor.7.6]{fe3} we conclude that $\Ss$ is connected by analytic paths.
\end{proof}

We recall next the concept of checkerboard set \cite[Thm.8.4]{fe3}, which is a special type of chessboard sets, and some relevant properties.

\subsection{Checkerboard sets}\label{checkerb}

If $\Ss\subset\R^m$ is a semialgebraic set, we denote $\partial\Ss:=\cl(\Ss)\setminus\Reg(\Ss)$. If $\Ss$ is closed, we have $\partial\Ss=\Sing(\Ss):=\Ss\setminus\Reg(\Ss)$. If $\Qq$ is a Nash manifold with corners such that $\ol{\Qq}^{\zar}$ is a non-singular real algebraic set, then $\Int(\Qq)=\Reg(\Qq)$ and the differences $\Qq\setminus\Int(\Qq)$ and $\Qq\setminus\Reg(\Qq)$ define the same semialgebraic set $\partial\Qq$. In other situations, one should be careful with the possible ambiguity in the definition of the set $\partial\Qq$. In this section such ambiguity does not appear because all the Zariski closures of all involved Nash manifolds with corners are non-singular.

\begin{defn}
A pure dimensional semialgebraic set $\Tt\subset\R^n$ is a {\em checkerboard set} (Figure \ref{check}) if it satisfies the following properties:
\begin{itemize}
\item $\ol{\Tt}^{\zar}$ is a non-singular real algebraic set.
\item $ \ol{\partial\Tt}^{\zar}$ is a normal-crossings divisor of $\ol{\Tt}^{\zar}$.
\item $\Reg(\Tt)=\Sth(\Tt)$ is connected.\hfill$\sqbullet$ 
\end{itemize}
\end{defn}

\begin{figure}[!ht]
\begin{center}
\begin{tikzpicture}[scale=0.75]
\draw[fill=gray!50,opacity=0.4,draw=none] (0,0) -- (0,5) -- (5,5) -- (5,0) -- (0,0);
\draw[fill=white,draw=none] (1,2.5) -- (1,4) -- (2.5,4) -- (2.5,2.5) -- (1,2.5);
\draw[fill=white,draw=none] (2.5,1) -- (4,1) -- (4,2.5) -- (2.5,2.5) -- (2.5,1);

\draw[line width=1pt, dotted] (0,0) -- (0,5);
\draw[line width=1pt] (0,0) -- (2.5,0);
\draw[line width=1pt,dotted] (2.5,0) -- (5,0);

\draw[line width=1pt] (0,5) -- (1,5);
\draw[line width=1pt,dotted] (1,5) -- (3,5);
\draw[line width=1pt] (3,5) -- (5,5);

\draw[line width=1pt] (5,0) -- (5,5);

\draw[line width=1pt] (1,2.5) -- (1,4);
\draw[line width=1pt] (1,4)-- (2.5,4);
\draw[line width=1pt,dotted] (2.5,4) -- (2.5,2.5);
\draw[line width=1pt] (2.5,2.5) -- (1,2.5);

\draw[line width=1pt,dotted] (4,1) -- (2.5,1);
\draw[line width=1pt] (2.5,1) -- (2.5,2.5);
\draw[line width=1pt] (2.5,2.5) -- (4,2.5);
\draw[line width=1pt] (4,1) -- (4,2.5);

\draw[fill=black,draw] (0,0) circle (0.75mm);
\draw[fill=white,draw] (5,0) circle (0.75mm);
\draw[fill=white,draw] (0,5) circle (0.75mm);
\draw[fill=black,draw] (5,5) circle (0.75mm);

\draw[fill=white,draw] (1,5) circle (0.75mm);
\draw[fill=black,draw] (3,5) circle (0.75mm);
\draw[fill=white,draw] (2.5,0) circle (0.75mm);

\draw[fill=white,draw] (1,4) circle (0.75mm);
\draw[fill=black,draw] (1,2.5) circle (0.75mm);
\draw[fill=black,draw] (2.5,2.5) circle (0.75mm);
\draw[fill=white,draw] (2.5,1) circle (0.75mm);
\draw[fill=white,draw] (2.5,4) circle (0.75mm);
\draw[fill=white,draw] (4,1) circle (0.75mm);
\draw[fill=black,draw] (4,2.5) circle (0.75mm);

\draw (1.5,1.5) node {\small$\Ss$};

\draw[fill=gray!50,opacity=0.4,draw=none] (8,0) -- (8,5) -- (13,5) -- (13,0) -- (8,0);
\draw[fill=white,draw=none] (9,2.5) -- (9,4) -- (10.5,4) -- (10.5,2.5) -- (9,2.5);
\draw[fill=white,draw=none] (10.5,1) -- (12,1) -- (12,2.5) -- (10.5,2.5) -- (10.5,1);

\draw[line width=1pt] (8,0) -- (8,5) -- (13,5) -- (13,0) -- (8,0);
\draw[line width=1pt] (10.5,2.5) -- (9,2.5) -- (9,4) -- (10.5,4) -- (10.5,2.5);
\draw[line width=1pt] (10.5,2.5) -- (10.5,1) -- (12,1) -- (12,2.5) -- (10.5,2.5);

\draw (11.5,3.5) node{$\Tt$};

\end{tikzpicture}
\end{center}
\caption{\small{A general checkerboard set $\Ss$ (left) and a closed checkerboard set $\Tt$ (right).\label{check}}}
\end{figure}

Any checkerboard set is connected by analytic paths \cite[Main Thm.1.4, Lem.8.2]{fe3}. In \cite[Thm.8.4]{fe3} the following crucial result is proved (making use at its initial stage of Theorem \ref{bpd}).

\begin{thm}[{\cite[Thm.8.4]{fe3}}]\label{ridwell}
Let $\Ss\subset\R^m$ be a semialgebraic set of dimension $d\geq 2$ connected by analytic paths and denote $X:=\ol{\Ss}^{\zar}$. Then there exists a checkerboard set $\Tt\subset\R^n$ of dimension $d$ and a proper regular map $f:\ol{\Tt}^{\zar}\to X$ such that $f(\Tt)=\Ss$.
\end{thm}

As the map $f$ is proper, if the semialgebraic set $\Ss$ is compact, we may assume that also the checkerboard set $\Tt$ is compact (see the proof of \cite[Thm.8.4]{fe3}). Even if it is not explicitly quoted in the statement of \cite[Thm.8.4]{fe3}, there is actually proved much more. Looking at its proof and keeping in mind that the components connected by analytic paths of a pure dimensional semialgebraic set of dimension $d$ are again pure dimensional semialgebraic sets of dimension $d$, we can reformulate the statement of Theorem \ref{ridwell} in the following more general form:

\begin{thm}\label{ridwell2}
Let $\Ss\subset\R^m$ be a pure dimensional semialgebraic set of dimension $d\geq 2$ and let $r$ be the number of components of $\Ss$ connected by analytic paths. Then there exist: 
\begin{itemize}
\item[(i)] A pairwise disjoint finite union $\Tt$ of $r$ checkerboard sets $\Tt_i\subset\R^n$ of the same dimension $d$ such that $\ol{\Tt}^{\zar}$ is a non-singular real algebraic set. 
\item[(ii)] A proper polynomial map $f:\ol{\Tt}^{\zar}\to\ol{\Ss}^{\zar}$ such that $f(\Tt)=\Ss$ and the restriction $f|_{\Tt}:\Tt\to\Ss$ is also proper.
\item[(iii)] A semialgebraic set $\Rr\subset\Ss$ of dimension strictly smaller than $d$ such that $f^{-1}(\Rr)\subset\ol{\partial\Tt}^{\zar}$, $\Ss\setminus\Rr$ and $\Tt\setminus f^{-1}(\Rr)$ are Nash manifolds of dimension $d$ and $f|_{\Tt\setminus f^{-1}(\Rr)}:\Tt\setminus f^{-1}(\Rr)\to\Ss\setminus\Rr$ is a Nash diffeomorphism. 
\end{itemize}
In particular, if $\ol{\Ss}^{\zar}$ is compact, also $\ol{\Tt}^{\zar}$ is compact.
\end{thm}
\begin{remark}
Theorem \ref{ridwell2} reduces the proof of Theorems \ref{red2} and \ref{red4} to the cases of closed checkerboard sets and general checkerboard sets.\hfill$\sqbullet$
\end{remark}

\subsection{Definition and properties of the function $e$ of a closed checkerboard set}
Given a non-singular real algebraic set $X\subset\R^n$ of dimension $d$ and a normal-crossings divisor $Z\subset X$, we denote\begin{align*}
&\Sing_0(Z):=Z,\\
&\Sing_\ell(Z):=\Sing(\Sing_{\ell-1}(Z))\quad\text{for $1\leq\ell\leq d$}.
\end{align*}
Observe that if $\Sing_\ell(Z)\neq\varnothing$, then $\dim(\Sing_\ell(Z))=d-\ell-1$. In addition, if $\Sing_\ell(Z)=\varnothing$, then $\Sing_k(Z)=\varnothing$ for $k\geq\ell$. In particular, $\Sing_d(Z)=\varnothing$. Recall that if $A_x\subset\R^n_x$ is a set germ at $x$ of a subset $A\subset\R^n$, its {\em analytic closure} $\ol{A}^\an_x$ is the smallest analytic set germ of $\R^n_x$ that contains $A_x$.

Let $\Tt\subset\R^n$ be a closed checkerboard set and denote $X:=\ol{\Tt}^{\zar}$ and $\partial\Tt:=\Tt\setminus\Reg(\Tt)$. For each point $x\in\Tt$ there exists a coordinate system $(u_1,\ldots,u_d)$ of the Nash manifold $X$ at $x$ and an integer $0\leq r_x\leq d$ such that either $\ol{\partial\Tt}^\an_x=\{u_1\cdots u_{r_x}=0\}_x$ if $r_x\geq 1$ or $x\in\Reg(\Tt)=\Tt\setminus\partial\Tt$ if $r_x=0$. We denote with $e_x:=e_x(\Tt)\leq r_x$ the number of indices $1\leq i\leq r_x$ such that the germ $\Tt_x\setminus\{u_i=0\}_x$ is disconnected. If $r_x\leq1$, then $e_x=0$. We have the following:

\begin{lem}\label{angoli}
The value $e_x=0$ if and only if $\Tt_x$ is the germ at $x$ of a Nash manifold with corners.
\end{lem} 
\begin{proof}
The if implication is clear because after changing the sign of some of the variables if necessary, we may assume either $x\in\Reg(\Tt)$ or $\Tt_x=\{u_1\geq0,\ldots,u_{r_x}\geq0\}_x$ for some $1\leq r_x\leq d$, so $e_x=0$. 

Conversely, suppose $e_x=0$. Then, either $x\in\Reg(\Tt)$ (so $\Tt_x=\Reg(\Tt)_x$ is the germ at $x$ of a Nash manifold) or $\ol{\partial\Tt}^\an_x=\{u_1\cdots u_{r_x}=0\}_x$ for some $1\leq r_x\leq d$ and, after changing the sign of some of the variables, we may assume $\Tt_x\subset\{u_1\geq0,\ldots,u_{r_x}\geq0\}_x$, because $e_x=0$. As $\Tt_x\setminus\ol{\partial\Tt}^\an_x=\Reg(\Tt)_x\setminus\ol{\partial\Tt}^\an_x$ is an open and closed germ, $\Tt_x\setminus\ol{\partial\Tt}^\an_x$ is a union of connected components of $\{u_1\cdots u_{r_x}\neq0\}_x$ contained in $\{u_1>0,\ldots,u_{r_x}>0\}_x$, so $\Tt_x\setminus\ol{\partial\Tt}^\an_x=\{u_1>0,\ldots,u_{r_x}>0\}_x$. As $\Tt_x$ is closed and pure dimensional and $\dim(\ol{\partial\Tt}^\an_x)<\dim(\Tt_x)$, we conclude $\Tt_x=\{u_1\geq0,\ldots,u_{r_x}\geq0\}_x$ is the germ at $x$ of a Nash manifold with corners, as required.
\end{proof}

It follows from the previous statement: \em A closed checkerboard set $\Tt\subset\R^n$ is a Nash manifold with corners if and only if $e_x(\Tt)=0$ for each $x\in\Tt$.\em

\begin{lem}\label{dimensione giusta}
Let $\Tt\subset\R^n$ be a closed checkerboard set. Then $e_x\neq 1$, for each $x\in\partial\Tt$.
\end{lem}

\begin{figure}[!ht]
\begin{center}
\begin{tikzpicture}[scale=0.5]
\draw[fill=gray!,opacity=0.6,draw=none] (0,0) -- (-4,-2) -- (-4,1) -- (0,3) -- (0,0);
\draw[fill=gray!,opacity=0.6,draw=none] (0,0) -- (4,-2) -- (4,1) -- (0,3) -- (0,0);
\draw[fill=gray!,opacity=0.6,draw=none] (0,0) -- (-4,-2) -- (0,-4) -- (4,-2) -- (0,0);

\draw[fill=gray!25,opacity=0.4,draw=none] (0,0) -- (-4,2) -- (-4,5) -- (0,7) -- (4,5) -- (4,2) -- (0,0);
\draw[fill=gray!25,opacity=0.4,draw=none] (0,0) -- (0,3) -- (4,5) -- (8,3) -- (8,0) -- (4,-2) -- (0,0);
\draw[fill=gray!25,opacity=0.4,draw=none] (0,0) -- (0,3) -- (-4,5) -- (-8,3) -- (-8,0) -- (-4,-2) -- (0,0);
\draw[fill=gray!25,opacity=0.4,draw=none] (0,0) -- (0,-3) -- (-4,-5) -- (-8,-3) -- (-8,0) -- (-4,-2) -- (0,0);
\draw[fill=gray!25,opacity=0.4,draw=none] (0,0) -- (-4,-2) -- (-4,-5) -- (0,-7) -- (4,-5) -- (4,-2) -- (0,0);
\draw[fill=gray!25,opacity=0.4,draw=none] (0,0) -- (0,-3) -- (4,-5) -- (8,-3) -- (8,0) -- (4,2) -- (0,0);

\draw[gray,dashed] (-4,-5) -- (-4,1) -- (4,5) -- (4,-1) -- (-4,-5);
\draw[gray,dashed] (4,1) -- (-4,5) -- (-4,-1) -- (4,-5) -- (4,1);
\draw[gray,dashed] (0,4) -- (-8,0) -- (0,-4) -- (8,0) -- (0,4);
\draw[gray,dashed] (0,-3) -- (0,3);
\draw[gray,dashed] (-4,-2) -- (4,2);
\draw[gray,dashed] (4,-2) -- (-4,2);
\draw[gray,dashed] (0,0) -- (0,7);
\draw[gray,dashed] (8,-3) -- (0,1) -- (-8,-3);

\draw[red,line width=1.75pt] (0,0) -- (0,3);
\draw[red,line width=1.75pt] (0,0) -- (-4,-2);
\draw[red,line width=1.75pt] (0,0) -- (4,-2);
\draw[black,thick] (0,-4) -- (0,-7);
\draw[black,thick] (0,3) -- (-4,1) -- (-8,3) -- (0,7) -- (8,3) -- (4,1) -- (0,3);
\draw[black,thick] (-4,1) -- (-4,-2) -- (0,-4) -- (4,-2) -- (4,1);
\draw[black,thick] (8,3) -- (8,-3) -- (0,-7) -- (-8,-3) -- (-8,3);

\draw (1,5) node{$\Ss$};
\filldraw[cyan] (0,0) circle (5pt);
\filldraw[red] (0,3) circle (5pt);
\filldraw[red] (-4,-2) circle (5pt);
\filldraw[red] (4,-2) circle (5pt);

\end{tikzpicture}
\end{center}
\caption{\small{Closed checkerboard set $\Ss$. Set of points with $e=3$ (cyan), set of points with $e=2$ (red) and set of points with $e=0$ (grey).}\label{e}}
\end{figure}

\begin{proof}
Let $X:=\ol{\Tt}^{\zar}$. For each $x\in\partial\Tt$ there exist an open semialgebraic set $U\subset X$ equipped with a Nash diffeomorphism $u:=(u_1,\ldots,u_d):U\to\R^d$ and an integer $1\leq r_x\leq d$ such that $u(x)=0$ and $\ol{\partial\Tt}^{\an}_x=\{u_{1}\cdots u_{r_x}=0\}_x$. 

Suppose that $e_x=1$ for some $x\in\partial\Tt$. As $e_x\neq 0$, then $r_x\geq 2$, because otherwise $\Tt_x$ is the germ of a Nash manifold with boundary and $e_x=0$. Up to rename the variables if necessary, we may assume $\Tt_x\setminus\{u_1=0\}_x$ is disconnected. Suppose that for each $2\leq i\leq r_x$ the germ $\Tt_x\setminus\{u_i=0\}_x$ is connected. After changing the signs of some of the variables if necessary, we may assume $\Tt_x\subset\{u_2\geq 0,\ldots,u_{r_x}\geq 0\}_x$. Proceeding as in the proof of Lemma \ref{angoli}, as $\Tt_x\setminus\ol{\partial\Tt}^\an_x=\Reg(\Tt)_x\setminus\ol{\partial\Tt}^\an_x$ is an open and closed germ, $\Tt_x\setminus\ol{\partial\Tt}^\an_x$ is a union of connected components of $\{u_1\cdots u_{r_x}\neq0\}_x$ contained in $\{u_2>0,\ldots,u_{r_x}>0\}_x$. As $\Tt_x$ is closed and pure dimensional and $\dim(\ol{\partial\Tt}^\an_x)<\dim(\Tt_x)$, we have only two possibilities:
\begin{itemize}
\item $\Tt_x=\{u_1\geq 0,\ldots,u_{r_x}\geq0\}_x$ (up to changing the sign of the germ $u_1$ if necessary),
\item $\Tt_x=\{u_2\geq 0,\ldots,u_{r_x}\geq 0\}_x$.
\end{itemize}
In the first case $e_x=0$, which contradicts the fact that $\Tt_x\setminus\{u_1=0\}_x$ is disconnected, whereas in the second case $\{u_1=0\}_x\not\subset\ol{\partial\Tt}^{\an}$, which contradicts the fact that $\ol{\partial\Tt}^{\an}=\{u_{1}\cdots u_{r_x}=0\}_x$. Thus, there exists $2\leq i\leq r_x$ such that $\Tt_x\setminus\{u_i=0\}_x$ is disconnected, so $e_x\geq 2$ as required.
\end{proof}

We show next that the function $e(\Tt)$ is a semialgebraic function (Figure \ref{e}).

\begin{lem}[Semialgebricity of $e(\Tt)$]\label{exs}
Let $\Tt\subset\R^n$ be a $d$-dimensional closed checkerboard set and let $0\leq e\leq d$. The set $\Tt_e:=\{x\in\Tt:\, e_x=e\}$ is a semialgebraic set and $\Tt_0$ is an open subset of $\Tt$. In addition, if $Z$ is the Zariski closure of $\partial\Tt$ and $C$ is a connected component of $\Sing_\ell(Z)\setminus\Sing_{\ell+1}(Z)$ for some $0\leq\ell\leq d-1$, then either $C\cap\Tt=\varnothing$ or $C\subset\Tt$ and $e(\Tt)$ is constant on $C$.
\end{lem}
\begin{proof}
The boundary $\partial\Tt$ is a closed semialgebraic subset of the Nash manifold $X:=\ol{\Tt}^{\zar}$. For each $x\in\partial\Tt$ there exists a coordinate system $(u_1,\ldots,u_d)$ of $X$ at $x$ and an integer $1\leq r_x\leq d$ such that $\ol{\partial\Tt}^\an_x=\{u_1\cdots u_{r_x}=0\}_x$. By \cite[Prop.4.4, Prop.4.6]{fgr} there exist finitely many open semialgebraic sets $\{U_i\}_{i=1}^{s}$ equipped with Nash diffeomorphisms $u_i:=(u_{i1},\ldots,u_{id}):U_i\to\R^d$ and integers $r_i\geq1$ such that $\ol{\partial\Tt}^{\an}_x=\{u_{i1}\cdots u_{ir_i}=0\}_x$ for all $x\in\Tt\cap U_i$. 

Fix $i\in\{1,\ldots,s\}$ and $J\subset\{1,\ldots,r_i\}$. Reordering the variables if necessary, we may assume $J=\{1,\ldots,m\}$ for some $1\leq m\leq r_i$. Let $V$ be a connected component of $U_i\setminus\{u_{i,m+1}\cdots u_{i,r_i}=0\}$. After changing the signs of some of the variables if necessary, we may assume $V:=\{u_{i,m+1}>0,\ldots,u_{i,r_i}>0\}$. Consider the semialgebraic set $\Tt':=\Tt\cap U_i\cap V$ and the projection $$\pi_i:\R^d\equiv\R^{m}\times\R^{d-m}\to\R^{m}$$ onto the first $m$ coordinates. We take coordinates $(x_1,\ldots,x_m)$ on $\R^m$ and $(x_{m+1},\ldots,x_d)$ on $\R^{d-m}$. Denote $\Lambda_i:=\{\x_{i,m+1}>0,\ldots,\x_{i,r_i}>0\}\subset\R^{d-m}$. As $u_i(\Tt\cap U_i)$ is the closure of a union of connected components of $\R^d\setminus\{\x_1\ldots\x_{r_i}=0\}$, there exist $\veps_{i,1},\ldots,\veps_{i,k}\in\{-1,1\}^{r_i}$ such that
$$
u_i(\Tt\cap U_i)=\bigcup_{p=1}^k\{\veps_{1p_1}\x_1\geq 0,\ldots,\veps_{ip_{r_i}}\x_{r_i}\geq 0\}
$$
where $\veps_{ip}:=(\veps_{ip_1},\ldots,\veps_{ip_{r_i}})$. Consequently,
$$
u_i(\Tt')=u_i(\Tt\cap U_i\cap V)=\bigcup_{p=1}^k\{\veps_{1p_1}\x_1\geq 0,\ldots,\veps_{ip_{r_i}}\x_{r_i}\geq 0,\x_{m+1}>0,\ldots,\x_{r_i}>0\}
$$
Observe that
\begin{multline*}
\{\veps_{1p_1}\x_1\geq 0,\ldots,\veps_{ip_{r_i}}\x_{r_i}\geq 0,\x_{m+1}>0,\ldots,\x_{r_i}> 0\}\\
=
\begin{cases}
\{\veps_{1p_1}\x_1\geq 0,\ldots,\veps_{ip_{m}}\x_{m}\geq 0,\x_{m+1}>0,\ldots,\x_{r_i}>0\}&\text{if }\veps_{ip_{m+1}}=\cdots=\veps_{ip_{r_i}}=1,\\
\varnothing&\text{otherwise}.
\end{cases}
\end{multline*}
Thus,
$$
u_i(\Tt')=\bigcup_{\substack{p\in\{1,\ldots,k\}\\ (\veps_{ip_{m+1}},\ldots,\veps_{ip_{r_i}})=(1,\ldots,1)}}\hspace{-1cm}\{\veps_{1p_1}\x_1\geq 0,\ldots,\veps_{ip_{m}}\x_{m}\geq 0,\x_{m+1}>0,\ldots,\x_{r_i}>0\}=\pi_i(u_i(\Tt'))\times\Lambda_i.
$$
Observe that $W:=\{u_{i1}=0,\ldots,u_{im}=0,u_{i,m+1}>0,\ldots,u_{ir_i}>0\}\subset\Tt$ and for each $x\in W$, we have $e_x(\Tt)=e_0(\pi_i(u_i(\Tt')))$, so $e_x(\Tt)$ is constant on $W$. As each $x\in\Tt\cap U_i\cap\{u_{i1}\cdots u_{ir_i}=0\}$ belongs to a set of the type $W_{J,\veps}:=\{u_{ij}=0, j\in J\}\cap\{\veps_{j}u_{ij}>0, j\not\in J\}$ where $J=\{1,\ldots,r_i\}$ and $\veps_{j}\in\{-1,1\}$, the function $e(\Tt)$ provides a semialgebraic partition of $\partial\Tt\cap U_i$ for each $i=1,\ldots,s$. In particular, each set $\Tt_e$ is semialgebraic. As the condition `to be a Nash manifold with corners' is a local open condition, we deduce $\Tt_0$ is an open semialgebraic subset of $\Tt$.

We have proved above that if $Z'_i:=\{u_{i1}\cdots u_{ir_i}=0\}$ and $C_i'$ is a connected component of $\Sing_\ell(Z'_i)\setminus\Sing_{\ell+1}(Z'_i)$ for some $0\leq\ell\leq d-1$, then either $C_i'\cap\Tt=\varnothing$ or $C_i'\subset\Tt$ and $e(\Tt)$ is constant on $C_i'$. Let $C$ be a connected component of $\Sing_\ell(Z)\setminus\Sing_{\ell+1}(Z)$. As $\Tt$ is a closed checkerboard set, either $C\cap\Tt=\varnothing$ or $C\subset\Tt$. Assume we are in the second case. As we have proved above, there exists a finite semialgebraic open covering $\Ww_C:=\{W_i\}_{i=1}^p$ of $C$ such that $e(\Tt)$ is constant on $W_i$. If $x,y\in C$, there exists $W_{i_1},\ldots,W_{i_q}\in\Ww_C$ such that $x\in W_{i_1}$, $y\in W_{i_q}$ and $W_{i_j}\cap W_{i_{j+1}}\neq\varnothing$ for $j=1,\ldots,q-1$. As $e(\Tt)|_{W_{i_j}}$ is constant, we deduce recursively that $e_x(\Tt)=e_y(\Tt)$, as required.
\end{proof}

We show next that the function $e(\Tt)$ is upper semi-continuous (Figure \ref{e}).

\begin{lem}[Upper semi-continuity of $e(\Tt)$]\label{exs1}
Let $\Tt\subset\R^d$ be a $d$-dimensional closed checkerboard set and let $x\in\partial\Tt$. Then $e_x\geq e_y$ for each $y\in\Tt$ close enough to $x$.
\end{lem}
\begin{proof}
For each integer $e\geq 0$ denote $\Tt_e:=\{x\in\Tt:\, e_x=e\}$. Let $k$ be the maximum of the values $e\geq0$ such that $x\in\cl(\Tt_e)$. It is enough to check that $e_x\geq k$. Consider the Nash manifold $X:=\ol{\Tt}^{\zar}$ and let $U\subset X$ be an open semialgebraic neighborhood of $x$ equipped with a Nash diffeomorphism $u:=(u_1,\ldots,u_d):U\to\R^d$ such that $u(x)=0$ and $\ol{\partial\Tt}^\an_x=\{u_1\cdots u_r=0\}_x$. If $e_x< k$, we may assume $\Tt_x\subset\{u_k\geq0,\ldots,u_r\geq0\}_x$ and $\partial\Tt_x\subset\{u_1\cdots u_r=0\}_x$. Shrinking $U$ if necessary, we have $\Tt\cap U\subset\{u_k\geq0,\ldots,u_r\geq0\}$ and $\partial\Tt\cap U\subset\{u_1\cdots u_r=0\}$. Thus, $e_y<k$ for each $y\in U$, which is a contradiction because $x\in\cl(\Tt_k)$.
\end{proof}

\begin{remark}\label{ridremark2}
As $e(\Tt)$ is upper semi-continuous, the set $\bigcup_{k\geq e}\Tt_k$ is a closed subset of $\partial\Tt$ for each $1\leq e\leq d$. 
\hfill$\sqbullet$
\end{remark}

\subsection{Closed checkerboard sets and drilling blow-up.}

We want to study now how the value $e_x(\Tt)$ changes after performing a drilling blow-up. We show the following:

\begin{lem}\label{exs2}
Let $\Ss\subset\R^n$ be a $d$-dimensional closed checkerboard set. Denote $X:=\ol{\Ss}^{\zar}$ and $Z:=\ol{\partial\Ss}^{\zar}$. Let $Z_1,\ldots,Z_r$ be the irreducible components of $Z$ and $Y$ an irreducible component of $Z_1\cap\cdots\cap Z_\ell$ for some $2\leq\ell\leq r$. Let $(\widehat{X},\widehat{\pi})$ be the twisted Nash double of the drilling blow-up $(\widetilde{X},\pi_+)$ of $X$ with center $Y$. Let $\Tt:=\cl(\pi_+^{-1}(\Ss\setminus Y))=\pi_+^{-1}(\Ss)\cap\cl(\pi_+^{-1}(\Ss\setminus Y))$ be the strict transform of $\Ss$ under $\pi_+$. Then $\Tt$ is a $d$-dimensional closed checkerboard set, $e_y(\Tt)\leq e_{\pi_+(y)}(\Ss)$ for each $y\in\partial\Tt$ and $e_y(\Tt)<e_{\pi_+(y)}(\Ss)$ for each $y\in\partial\Tt\cap\pi_+^{-1}(Y)$ such that $\Ss_{\pi_+(y)}\setminus Z_{i,\pi_+(y)}$ is not connected for $i=1,\ldots,\ell$.
\end{lem}
\begin{proof}
As $\ell\geq2$, we have $\dim(Y)\leq d-2$, so $\Reg(\Ss)\setminus Y$ is connected, because $\Reg(\Ss)$ is a connected $d$-dimensional Nash manifold. As $\Ss$ and $\Tt$ are both pure dimensional, we have
$$
\pi_+^{-1}(\Reg(\Ss)\setminus Y)\subset\Sth(\Tt)=\Reg(\Tt)\subset\cl(\pi_+^{-1}(\Ss\setminus Y))=\cl(\pi_+^{-1}(\Reg(\Ss)\setminus Y)).
$$
Thus, $\Reg(\Tt)$ is connected and $\Tt$ is a checkerboard set, because (see \S\ref{addbu}): it is closed, $\ol{\Tt}^{\zar}=\widehat{X}$ is a non-singular real algebraic set and the Zariski closure of $\partial\Tt$ is a union of irreducible components of $\widehat{\pi}^{-1}(\ol{\partial\Ss}^{\zar})$, which is a Nash normal-crossings divisor by Fact \ref{bigstepa6} and Remark \ref{exst}(ii).

As $\pi_+|_{\widetilde{X}\setminus\pi_+^{-1}(Y)}:\widetilde{X}\setminus\pi_+^{-1}(Y)\to X\setminus Y$ is a Nash diffeomorphism, it holds $e_y(\Tt)= e_{\pi_+(y)}(\Ss)$ for each $y\in\partial\Tt\setminus\pi_+^{-1}(Y)$. Let us see what happens at the points of $\partial\Tt\cap\pi_+^{-1}(Y)$. Fix a point $y\in\partial\Tt\cap\pi_+^{-1}(Y)$ and denote $x:=\pi_+(y)\in Y$.

Assume that the irreducible components of $Z$ that contain $x$ are $Z_1,\ldots,Z_{r'}$ for some $2\leq\ell\leq r'\leq r$. Let $U\subset X$ be an open semialgebraic neighborhood of $x$ equipped with a Nash diffeomorphism $u:=(u_1,\ldots,u_d):U\to\R^d$ such that $u(Z\cap U)=\{u_1\cdots u_{r'}=0\}$ and $u(Y\cap U)=\{u_1=0,\ldots,u_\ell=0\}$. Write $e:=\dim(Y)=d-\ell$ and assume that $e_x(\Ss)=k\leq r$. Reordering the variables and changing their signs if necessary, we may assume 
\begin{equation}\label{inc4}
\Ss\cap U\subset\{u_1\geq0,\ldots,u_m\geq0,u_{\ell+1}\geq0,\ldots,u_s\geq0\}
\end{equation}
for some $0\leq m\leq\ell$ and $\ell\leq s\leq r'$ and both $m$ and $s$ are maximal satisfying \eqref{inc4}. If $m=0$, then $\Ss\cap U\subset\{u_{\ell+1}\geq0,\ldots,u_s\geq0\}$, whereas if $s=\ell$, then $\Ss\cap U\subset\{u_1\geq0,\ldots,u_m\geq0\}$. As $e_x(\Ss)=k$, we have $k=(\ell-m)+(r'-s)$. By Fact \ref{bigstepa6} we can choose (local) coordinates in $\widehat{X}$ such that $\pi_+$ behaves (with respect to the already taken (local) coordinates in $X$) as the Nash map
$$
g_+:[0,+\infty)\times\sph^{d-e-1}\times\R^e\to\R^d,\ (\rho,w,z)\mapsto(\rho w,z),
$$
where $w:=(w_1,\ldots,w_\ell)$ and $z:=(z_{\ell+1},\ldots,z_d)\in\R^e=\R^{d-\ell}$. We have 
\begin{align*}
g_+^{-1}(u(Z\cap U))&=\{\rho^\ell w_1\cdots w_\ell z_{\ell+1}\cdots z_{r'}=0,w_1^2+\cdots+w_\ell^2=1\},\\
g_+^{-1}(u(Y\cap U))&=\{\rho w_1=0,\ldots,\rho w_\ell=0,w_1^2+\cdots+w_\ell^2=1\}=\{\rho=0,w_1^2+\cdots+w_\ell^2=1\},\\
g_+^{-1}(u(\Ss\cap U\setminus Y))&\subset\{\rho\geq0,w_1\geq0,\ldots,w_m\geq0,z_{\ell+1}\geq0,\ldots,z_{s}\geq0,w_1^2+\cdots+w_\ell^2=1\}.
\end{align*}
If $m=0$, then $\Ss\cap U\subset\{u_{\ell+1}\geq 0,\ldots,u_s\geq 0\}$ and 
$$
g_+^{-1}(u(\Ss\cap U\setminus Y))\subset\{\rho\geq0,z_{\ell+1}\geq0,\ldots,z_{s}\geq0,w_1^2+\cdots+w_\ell^2=1\}.
$$
Thus, $e_y(\Tt)\leq\ell-1+r'-s=k-1<k=e_x(\Ss)$ for each $y\in g^{-1}(x)$. The condition $m=0$ means that $\Ss_x\setminus Z_{i,x}$ is not connected for $i=1,\ldots,\ell$. 

We assume in the following $m\geq 1$. Observe that $\Ss_x\setminus Z_{i,x}$ is not connected for $i=m+1,\ldots,\ell$. Let us show $e_y(\Tt)\leq e_x(\Ss)$ for each $y\in g^{-1}(x)$. It may happen that for some $y\in g^{-1}(x)$ the previous inequality is strict even if $\Ss_x\setminus Z_{i,x}$ is connected for the indices $i=1,\ldots,m$. If some $ w_i(y)\neq0$, this variable has no relevance in the description of $\Tt$ locally around $y$ and $ w_i$ behaves as $\pm\sqrt{1-\sum_{j\neq i} w_j^2}$. Analogously, if some $ z_j(y)\neq0$, this variable has no relevance in the description of $\Tt$ locally around $y$. As $ w_1^2+\cdots+w_\ell^2=1$, there exists an index $1\leq i\leq\ell$ such that $ w_i(y)\neq0$: 

\noindent{\sc Case 1.} If $1\leq i\leq m$, we may assume $i=m$. Thus, $ w_m(y)>0$ and $ w_m=+\sqrt{1-\sum_{j\neq m} w_j^2}$. We consider (local) coordinates 
$$
(\rho,w_1,\ldots,w_{m-1},w_{m+1},\ldots,w_\ell,z_{\ell+1},\ldots,z_s,z_{s+1},\ldots,z_{r'},z_{r'+1},\ldots z_d)
$$ 
and observe that
$$
e_y(\Tt)\leq\ell-(m-1+1)+(r'-s)=(\ell-m)+(r'-s)=k=e_x(\Ss).
$$

\noindent{\sc Case 2.} If $m+1\leq i\leq\ell$, we may assume $i=\ell$. Thus, $ w_\ell(y)\neq0$ and $ w_\ell=\pm\sqrt{1-\sum_{j=1}^{\ell-1} w_i^2}$. We consider (local) coordinates 
$$
(\rho,w_1,\ldots,w_m,w_{m+1},\ldots,w_{\ell-1},z_ {\ell+1},\ldots,z_s,z_{s+1},\ldots,z_{r'},z_{r'+1},\ldots z_d)
$$ 
and observe that 
$$
e_y(\Tt)\leq 1+(\ell-1)-(m+1)+(r'-s)=(\ell-m)+(r'-s)-1=k-1<k=e_x(\Ss), 
$$
as required.
\end{proof}

\subsection{Proof of Theorem \ref{red2} for closed checkerboard sets}\label{44}

We are ready to prove Theorem \ref{red2}. After the previous preparation (in particular Theorem \ref{ridwell2}) we are reduced to the case when $\Ss$ is a $d$-dimensional closed checkerboard set. 

\begin{proof}[Proof of Theorem {\em \ref{red2}} for closed checkerboard sets]
Let us construct the Nash manifold with corners $\Qq$ first. Let $\Ss\subset\R^m$ be a $d$-dimensional closed checkerboard set and denote $X_0:=\ol{\Ss}^{\zar}$. Let $Z:=\ol{\partial\Ss}^{\zar}$ and $Z_1,\ldots,Z_r$ be its irreducible components. Define $e:=\max\{e_x(\Ss):\ x\in\partial\Ss\}$.

If $e=0$, we conclude by Lemma \ref{angoli} that $\Ss$ is already a Nash manifold with corners. Otherwise, by Lemma \ref{dimensione giusta} $e\geq 2$. By Remark \ref{ridremark2} $\Ss_e:=\{x\in\Ss:\ e_x=e\}$ is a closed semialgebraic subset of $\partial\Ss$. By Lemma \ref{exs} $\Ss_e$ is a union of connected components of the semialgebraic sets $\Sing_\ell(Z)\setminus\Sing_{\ell+1}(Z)$ for $1\leq\ell\leq d-1$ (recall that $e\geq2$). 

Pick a point $x\in\Ss_e$ and assume that $Z_1,\ldots,Z_e$ are the irreducible components of $Z$ such that the germ $\Ss_x\setminus Z_{i,x}$ is not connected for $i=1,\ldots,e$. Then there exists an open semialgebraic neighborhood $U\subset X$ of $x$ equipped with a Nash diffeomorphism $u:=(u_1,\ldots,u_d):U\to\R^d$ such that $u(x)=0$, $Z\cap U=\{u_1\cdots u_r=0\}$, $Z_i\cap U=\{u_i=0\}$ and $\Ss\cap U\subset\{u_{e+1}\geq0,\ldots,u_r\geq0\}$. Thus, 
$
\{u_1=0,\ldots,u_e=0,u_{e+1}\geq0,\ldots,u_r\geq0\}\subset\Ss_e. 
$
By Lemma \ref{exs1} the connected component $\Cc$ of $(Z_1\cap\cdots\cap Z_e)\setminus\bigcup_{i=e+1}^rZ_i$ that contains $\{u_1=0,\ldots,u_e=0,u_{e+1}\geq0,\ldots,u_r\geq0\}$ is contained in $\Ss_e$. Thus, the Zariski closure of $\Cc$ is the irreducible component of $Z_1\cap\cdots\cap Z_e$ that contains $x$. As we can repeat the previous argument for each $x\in\Ss_e$, we conclude that the Zariski closure of $\Ss_e$ is a union of irreducible components of the real algebraic set $\bigcup_{\{i_1,\ldots,i_e\}\subset\{1,\ldots,r\}}\bigcap_{j=1}^eZ_{i_j}$. As $e\geq2$, we have $\dim(\ol{\Ss_e}^{\zar})\leq d-2$.

In addition, for each $x\in\Ss_e$ there exist irreducible components $Z_{i_1},\ldots,Z_{i_e}$ of $Z$ such that the germ $\Ss_x\setminus Z_{i_j,x}$ is not connected for $j=1,\ldots,e$. We proceed by double induction on $e$ and the number $m$ of irreducible components of the Zariski closure of $\ol{\Ss_e}^{\zar}$ of $\Ss_e$. 

Let $W$ be an irreducible component of $\ol{\Ss_e}^{\zar}$, which has dimension $\leq d-2$. Let $(\widehat{X}_0,\widehat{\pi})$ be the twisted Nash double of the drilling blow-up $(\widetilde{X}_0,\pi_+)$ of $X_0$ with center $W$, which is by \S\ref{addbu} a non-singular real algebraic set. Let 
$$
\Tt:=\pi_+^{-1}(\Ss)\cap\cl(\pi_+^{-1}(\Ss\setminus W))=\cl(\pi_+^{-1}(\Ss\setminus W))
$$
be the strict transform of $\Ss$ under $\pi_+$ (recall that $\Ss$ is closed). As $\Ss$ is pure dimensional and $\Ss_e\subset\ol{\partial\Ss}^{\zar}$ has dimension strictly smaller, $\Ss\setminus W$ is dense in $\Ss$, so
$$
\pi_+(\Tt)=\pi_+(\cl(\pi^{-1}_+(\Ss\setminus W)))=\cl(\pi_+(\pi_+^{-1}(\Ss\setminus W)))=\cl(\Ss\setminus W)=\Ss,
$$
because $\pi_+:\widetilde{X}_0\to X_0$ is proper and surjective. As $e\geq2$, we can apply Lemma \ref{exs2} and we deduce: {\em $\Tt$ is a checkerboard set, $e_y(\Tt)\leq e_{\pi_+(y)}(\Ss)$ for each $y\in\partial\Tt$ and $e_y(\Tt)<e_{\pi_+(y)}(\Ss)$ for each $y\in\partial\Tt\cap\pi_+^{-1}(W)$ such that $\Ss_{\pi_+(y)}\setminus Z_{i_j,\pi_+(y)}$ is not connected for $j=1,\ldots,e$}. 

If $\max\{e(\Tt)_y:\ y\in\partial\Tt\}<e$, by induction hypothesis the statement holds for $\Tt$, so it also holds for $\Ss$. If $\max\{e(\Tt)_y:\ y\in\partial\Tt\}=e$, the Zariski closure of $\Tt_e$ is contained in $\cl(\widehat{\pi}^{-1}(\ol{\Ss_e}^{\zar}\setminus W))$ and it has $m-1$ irreducible components. As by Lemma \ref{angoli} $e=0$ if and only if $\Tt$ is a Nash manifold with corners, our inductive argument is consistent. Thus, by induction hypothesis the statement holds for $\Tt$, so it also holds for $\Ss$.

Let $\Qq\subset\R^n$ be the Nash manifold with corners obtained by our inductive process. As $\Qq$ is a checkerboard set, $\Reg(\Qq)$ is connected, so also $\Qq=\cl(\Reg(\Qq))$ is connected. We have constructed $\Qq$, starting from $\Ss$, with a finite number of drilling blow-ups. Namely, we have constructed a finite number of tuples $\{\Tt_i, (\widetilde{X}_{\Tt_i},\pi_{+,i}),(\widehat{X}_{\Tt_i},\widehat{\pi}_i)\}_{i=0}^s$ where:
\begin{itemize}
\item $\Tt_0:=\Ss$, $\widetilde{X}_{\Tt_0}:=X_0=\ol{\Ss}^{\zar}$, $\pi_{+,0}:=\id_{X_0}$, $\widehat{X}_{\Tt_0}:=X_0$ and $\widehat{\pi}_0:=\id_{X_0}$. 
\item $(\widetilde{X}_{\Tt_{i+1}},\pi_{+,i+1})$ is the drilling blow-up of the irreducible non-singular real algebraic set $\ol{\Tt}_i^{\zar}$ with center a (suitable) irreducible non-singular real algebraic subset $W_i$ of $\ol{\Tt}_i^{\zar}$ of dimension $\leq d-2$.
\item $(\widehat{X}_{\Tt_{i+1}},\widehat{\pi}_{i+1})$ is the twisted Nash double of $(\widetilde{X}_{\Tt_{i+1}},\pi_{+,i+1})$ and it is a non-singular real algebraic set of dimension $d$. By Lemma \ref{irredx} $\widehat{X}_{\Tt_{i+1}}$ is the Zariski closure of $\widetilde{X}_{\Tt_{i+1}}$.
\item $\Tt_{i+1}:=\cl(\pi_{+,i+1}^{-1}(\Tt_i\setminus W_i))$ is the strict transform of $\Tt_i$ under $\pi_{+,i+1}$ and $\pi_{+,i+1}(\Tt_{i+1})=\Tt_i$. By Lemma \ref{exs2} $\Tt_{i+1}$ is a checkerboard set.
\item $\pi_{+,i+1}^{-1}(W_i)\cap\Tt_{i+1}\subset\partial\Tt_{i+1}$, $\pi_{+,i+1}^{-1}(\partial\Tt_i\setminus W_i)\subset\partial\Tt_{i+1}$ and $\pi_{+,i+1}^{-1}(\Reg(\Tt_i)\setminus W_i)\subset\Reg(\Tt_{i+1})$ (use Fact \ref{bigstepb2}(ii) and \S\ref{addbu}). Consequently, $\Reg(\Tt_{i+1})=\pi_{+,i+1}^{-1}(\Reg(\Tt_i)\setminus W_i)$ and $\partial\Tt_{i+1}=\pi_{+,i+1}^{-1}(\partial\Tt_i\setminus W_i)\cup(\pi_{+,i+1}^{-1}(W_i)\cap\Tt_{i+1})$.
\item[$(*)$] $\pi_{+,i+1}(\partial\Tt_{i+1})=\partial\Tt_i\cup(W_i\cap\Tt_i)$.
\item[$(**)$] $\pi_{+,i+1}|_{\Reg(\Tt_{i+1})}:\Reg(\Tt_{i+1})\to\Reg(\Tt_i)\setminus(W_i\cap\Tt_i)$ is a Nash diffeomorphism (Fact \ref{bigstepb2}(ii)).
\item $\Tt_s=\Qq$.
\end{itemize}
As $\Tt_i$ is connected by analytic paths, $\ol{\Tt}_i^{\zar}$ is by \cite[Lem.7.3]{fe3} an irreducible component of the pure dimensional non-singular real algebraic set $\widehat{X}_{\Tt_i}$ of dimension $d$. Thus, $\ol{\Tt}_i^{\zar}$ is a pure dimensional irreducible non-singular real algebraic set of dimension $d$. Consequently, the Nash manifold with corners $\Qq=\Tt_s$ is (by Lemma \ref{exs2}) a checkerboard set, $X:=\ol{\Qq}^{\zar}$ is a $d$-dimensional irreducible non-singular real algebraic set and $Y:=\ol{\partial\Qq}^{\zar}$ is a normal-crossings divisor of $X$. Thus, (i) and (ii) hold. As all the involved polynomials maps $\widehat{\pi}_i$ are proper, if $X_0$ is compact, then $X$ is also compact.

Consider the polynomial map $f:=\widehat{\pi}_{1}\circ\ldots\circ\widehat{\pi}_{s}:\widehat{X}_{\Tt_s}\to X_0$. By Fact \ref{bigstepb2}(i) $f:\widehat{X}_{\Tt_s}\to X_0$ is composition of proper maps, so it is proper. Moreover, as $\Ss$ is closed and $\Qq$ is obtained from $\Ss$ after a finite number of drilling blow-ups taking strict transforms in each step, $\Qq$ is a closed subset of $X$. Thus, also the restriction $f|_{\Qq}:\Qq\to X_0$ is proper. In addition, as $\widehat{\pi}_{i+1}(\Tt_{i+1})=\Tt_i$ for $i=0,\ldots,s-1$, we conclude that $f(\Qq)=\Ss$. 

Let us show (iv). By property $(**)$ applied inductively we deduce 
$$
f(\Reg(\Qq))\subset\Reg(\Tt_0)\setminus ((W_0\cap\Tt_0)\cup\bigcup_{k=1}^{s-1}(\pi_{+,1}\circ\cdots\circ\pi_{+,k})(W_k\cap\Tt_k)), 
$$
whereas by property $(*)$ applied recursively we have 
\begin{equation}\label{residuo01}
f(\partial\Qq)=\partial\Tt_0\cup(W_0\cap\Tt_0)\cup\bigcup_{k=1}^{s-1}(\pi_{+,1}\circ\cdots\circ\pi_{+,k})(W_k\cap\Tt_k),
\end{equation}
so $f(\Reg(\Qq))\cap f(\partial\Qq)=\varnothing$. Thus, as $\Ss=f(\Qq)=f(\Reg(\Qq))\cup f(\partial\Qq)$, we deduce $f(\Reg(\Qq))=\Ss\setminus f(\partial\Qq)$ and $f^{-1}(f(\partial\Qq))=\partial\Qq$. Property $(**)$ implies $f|_{\Reg(\Qq)}=(\pi_{+,1}\circ\cdots\circ\pi_{+,s})|_{\Reg(\Qq)}:\Reg(\Qq)\to\Ss\setminus f(\partial\Qq)$ is a Nash diffeomorphism (because it is a composition of finitely many Nash diffeomorphisms). By \eqref{residuo01} we have $\partial\Ss\subset f(\partial\Qq)$. The semialgebraic set $\Rr:=f(\partial\Qq)\subset\Ss$ is closed, because $f:\widehat{X}_{\Tt_s}\to X_0$ is by Fact \ref{bigstepb2}(i) proper (as we have already commented above) and $\partial\Qq$ is a closed subset of $X$, which is the Zariski closure of $\Qq$ in $\widehat{X}_{\Tt_s}$ (as we have already remarked above). As $\partial\Qq$ has dimension not greater than $d-1$, we have $\dim(\Rr)\leq d-1<d$. The semialgebraic sets $\Ss\setminus\Rr$ and $\Qq\setminus f^{-1}(\Rr)$ are Nash manifolds, because $\Ss\setminus\Rr$ is an open semialgebraic subset of the Nash manifold $\Reg(\Ss)=\Ss\setminus\partial\Ss$ and $\Qq\setminus f^{-1}(\Rr)=\Qq\setminus\partial\Qq=\Reg(\Qq)$. Consequently, $f|_{\Reg(\Qq)}:\Reg(\Qq)=\Qq\setminus f^{-1}(\Rr)\to\Ss\setminus\Rr$ is a Nash diffeomorphism, as required.
\end{proof}

\begin{figure}[!ht]
\begin{center}
\begin{tikzpicture}[scale=0.75]
\draw[fill=gray!50,opacity=0.4,draw=none] (0,0) -- (0,5) -- (5,5) -- (5,0) -- (0,0);
\draw[fill=white,draw=none] (1,2.5) -- (1,4) -- (2.5,4) -- (2.5,2.5) -- (1,2.5);
\draw[fill=white,draw=none] (2.5,1) -- (4,1) -- (4,2.5) -- (2.5,2.5) -- (2.5,1);

\draw[line width=1pt] (0,0) -- (1.5,0);
\draw[line width=1pt] (0,0) -- (0,1.5);
\draw[line width=1pt] (0,1.5) -- (0,5);
\draw[line width=1pt] (3.5,5) -- (0,5);
\draw[line width=1pt] (3.5,5) -- (5,5);
\draw[line width=1pt] (5,0) -- (5,5);
\draw[line width=1pt] (1.5,0) -- (5,0);
\draw[line width=1pt] (1,2.5) -- (1,4) -- (2.5,4) -- (2.5,2.5) -- (1,2.5);
\draw[line width=1pt] (4,1) -- (2.5,1) -- (2.5,2.5) -- (4,2.5);
\draw[line width=1pt] (4,1) -- (4,2.5);

\filldraw[color=white, fill=white] (2.5,2.5) circle (0.2);
\draw[line width=1pt] (2.5,2.3) arc (270:180:0.2cm);
\draw[line width=1pt] (2.5,2.7) arc (90:0:0.2cm);

\draw[fill=gray!50,opacity=0.4,draw=none] (8,0) -- (8,5) -- (13,5) -- (13,0) -- (8,0);
\draw[fill=white,draw=none] (9,2.5) -- (9,4) -- (10.5,4) -- (10.5,2.5) -- (9,2.5);
\draw[fill=white,draw=none] (10.5,1) -- (12,1) -- (12,2.5) -- (10.5,2.5) -- (10.5,1);

\draw[line width=1pt] (8,0) -- (8,5) -- (13,5) -- (13,0) -- (8,0);
\draw[line width=1pt] (10.5,2.5) -- (9,2.5) -- (9,4) -- (10.5,4) -- (10.5,2.5);
\draw[line width=1pt] (10.5,2.5) -- (10.5,1) -- (12,1) -- (12,2.5) -- (10.5,2.5);

\filldraw[color=white, fill=white] (1,2.5) circle (0.2);
\draw[line width=1pt] (1,2.7) arc (90:360:0.2cm);

\filldraw[color=white, fill=white] (1,4) circle (0.2);
\draw[line width=1pt] (1.2,4) arc (0:270:0.2cm);

\filldraw[color=white, fill=white] (2.5,4) circle (0.2);
\draw[line width=1pt] (2.5,3.8) arc (270:540:0.2cm);

\filldraw[color=white, fill=white] (2.5,1) circle (0.2);
\draw[line width=1pt] (2.5,1.2) arc (90:360:0.2cm);

\filldraw[color=white, fill=white] (4,1) circle (0.2);
\draw[line width=1pt] (3.8,1) arc (180:450:0.2cm);

\filldraw[color=white, fill=white] (4,2.5) circle (0.2);
\draw[line width=1pt] (4,2.3) arc (270:540:0.2cm);

\draw (3.5,3.5) node{$\Qq$};
\draw (11.5,3.5) node{$\Ss$};
\draw[->] (5.5 ,2.5)--(7.5, 2.5);
\draw (6.5,2.9) node{\small $f|_{\Qq}$};
\end{tikzpicture}
\end{center}
\caption{\small{Nash uniformization of the closed checkerboard set $\Ss$ (right) by the Nash manifold with corners $\Qq$ (left).}}
\end{figure}

\subsection{Application 1: Nash compactification of Nash manifolds with corners}
We prove next Theorem \ref{compactclosure}.

\begin{proof}[Proof of Theorem \em \ref{compactclosure}]
If $\Qq$ is compact, there is nothing to prove, so we assume that $\Qq$ is not compact. Let $M\subset\R^n$ be a Nash envelope of $\Qq$ such that $\Qq$ is closed in $M$ and the Nash closure $X$ of $\partial\Qq$ in $M$ is a Nash normal-crossings divisor satisfying $\Qq\cap X=\partial\Qq$ (Theorem \ref{divisorialPre}). Up to a suitable Nash embedding of $M$ in some aﬃne space $\R^m$ we may assume by \cite[Lem.C.1]{fe3} that:
\begin{itemize}
\item[(i)] $M$ is a (finite) union of connected components of its Zariski closure $V$ in $\R^m$, which is in addition a non-singular real algebraic subset of $\R^m$ of pure dimension $d$.
\item[(ii)] The Zariski closure $Y$ of $X$ in $\R^m$ is a normal-crossings divisor of $V$ and $M\cap Y=X$.
\end{itemize}
As $\Qq$ is non-compact and closed in $M$, we deduce that $M$ is non-compact. As $M$ is a (finite) union of connected components of $V$, also $V$ is non-compact. Let $\phi:\R^m\to\sph^m$ be the inverse of the stereographic projection of $\sph^m$ with respect to its north pole $p_N:=(0,\ldots,0,1)$. The Zariski closure of $\phi(V)$ is $W:=\phi(V)\cup\{p_N\}$ and the Zariski closure of $\phi(Y)$ is contained in $Z:=\phi(Y)\cup\{p_N\}$. If $W$ is non-singular, we do nothing in this step. Otherwise, $\Sing(W)=\{p_N\}$ and by Hironaka's resolution of singularities \cite{hi} there exist a non-singular algebraic set $W'\subset\R^q$ and a proper regular map $f:W'\to W$ such that
\begin{equation*}
f|_{W'\setminus f^{-1}(\{p_N\})}:W'\setminus f^{-1}(\{p_N\})\rightarrow W\setminus\{p_N\}
\end{equation*}
is a diffeomorphism whose inverse map is also regular. Denote $Z':=f^{-1}(Z)$ and observe that $Z'\setminus f^{-1}(p_N)$ is a normal-crossings divisor of $W'\setminus f^{-1}(p_N)$. By \cite[Thm.1.5]{bm4} there exists a non-singular algebraic set $W''\subset\R^p$ and a proper regular map $g:W''\to W'$ such that $Z'':=g^{-1}(Z')$ is a normal-crossings divisor of $W''$ and the restriction
\begin{equation*}
g|_{W''\setminus g^{-1}(f^{-1}(p_N))}:W'\setminus g^{-1}(f^{-1}(p_N))\rightarrow W'\setminus f^{-1}(p_N)
\end{equation*}
is a biregular diffeomorphism whose inverse map is also regular. As $p_N\not\in\Qq$, the inverse image $\Qq'':=g^{-1}(f^{-1}(\Qq))$ is Nash diffeomorphic to $\Qq$. Define $\Ss:=\cl(\Qq)=\Qq\cup\{p_N\}$ and $\Ss'':=\cl(\Qq'')$. The Zariski closure of $\Ss''\setminus\Int(\Qq'')\subset g^{-1}(f^{-1}(\partial\Qq))$ is a union of irreducible components of the normal-crossings divisor $Z''=g^{-1}(f^{-1}(Z))$. In addition, $\Ss''\setminus\Qq''\subset g^{-1}(f^{-1}(\Ss\setminus\Qq))=g^{-1}(f^{-1}(p_N))$. 

Observe that $e(\Ss'')_x=0$ for each $x\in\Ss''$. If we apply the procedure to prove Theorem \ref{red2} in \S\ref{44} we obtain a compact Nash manifold with corners $\Qq^\bullet\subset\R^q$ and a polynomial map $h:\R^q\to\R^p$ such that the restriction $h|_{\Qq^\bullet}:\Qq^\bullet\to\Ss''$ is a proper Nash map and there exists an algebraic set $T$ of dimension $<d$ contained in $g^{-1}(f^{-1}(p_N))$ such that $h|_{\Qq^\bullet\setminus h^{-1}(T)}:\Qq^\bullet\setminus h^{-1}(T)\to\Ss''\setminus T$ is a Nash diffeomorphism. Thus, $\Qq^\bullet\setminus h^{-1}(g^{-1}(f^{-1}(p_N)))$ is Nash diffeomorphic via $f\circ g\circ h$ to $\Ss\setminus\{p_N\}=\Qq$. Let ${\tt j}$ be the inverse of $(f\circ g\circ h)|_{\Qq^\bullet\setminus h^{-1}(g^{-1}(f^{-1}(p_N)))}$ composed with the inclusion of $\Qq^\bullet\setminus h^{-1}(g^{-1}(f^{-1}(p_N)))$ into $\Qq^\bullet$. Thus, $(\Qq^\bullet,{\tt j})$ is a compactification of $\Qq$ such that $\Qq^\bullet$ is a compact Nash manifold with corners, as required.
\end{proof}

\section{Nash uniformization of general chessboard sets}\label{s5}

To prove Theorem \ref{red4} we introduce {\em Nash quasi-manifolds with corners}, that is, semialgebraic sets $\Tt\subset\R^n$ whose closure is a Nash manifold with corners $\Qq\subset\R^n$ and $\Qq\setminus\Tt$ is a union of some of the strata of certain (Nash) stratification of $\partial\Qq$, that we introduce below. 

\subsection{Nash quasi-manifolds with corners}
Let us recall the definition of (Nash) stratification of a semialgebraic set.

\begin{defn}
Let $\Ss\subset\R^m$ be a semialgebraic set. A (Nash) \em stratification \em of $\Ss$ is a finite semialgebraic partition $\{\Ss_\alpha\}_{\alpha\in A}$ of $\Ss$, where each $\Ss_\alpha$ is a connected Nash submanifold of $\R^m$ and the following property is satisfied: if $\Ss_\alpha\cap\cl(\Ss_\beta)\neq\varnothing$ and $\alpha\neq\beta$, then $\Ss_\alpha\subset\cl(\Ss_\beta)$ and $\dim(\Ss_\alpha)<\dim(\Ss_\beta)$. The $\Ss_\alpha$ are called the \em strata \em of the (Nash) stratification and if $d:=\dim(\Ss_\alpha)$, then $\Ss_\alpha$ is a \em$d$-stratum\em.\hfill$\sqbullet$
\end{defn}

The condition $\dim(\Ss_\alpha)<\dim(\Ss_\beta)$ follows from \cite[Prop.2.8.13]{bcr}, because $\Ss_\alpha\subset\cl(\Ss_\beta)\setminus\Ss_\beta$.

\begin{defn}\label{sap}
Given a $d$-dimensional semialgebraic set $\Ss\subset\R^m$, we consider the following semialgebraic partition of $\Ss$. Recall that $\Sth(\Ss)$ is the set of points $x\in\Ss$ at which the germ $\Ss_x$ is the germ of a Nash manifold (\S\ref{regsmooth}). Define $\Gamma_1:=\Sth(\Ss)$ and $\Gamma_k:=\Sth(\Ss\setminus\bigcup_{j=1}^{k-1}\Gamma_j)$ for $k\geq 2$. Let $s\geq1$ be the largest index $k$ such that $\Gamma_k\neq\varnothing$. For each $k\geq1$ let $\Gamma_{k\ell}$ (for $\ell=1,\ldots,r_k$) be the connected components of $\Gamma_k$. The collection ${\mathfrak G}(\Ss):=\{\Gamma_{k\ell}:\ 1\leq k\leq s,1\leq\ell\leq r_k\}$ is a partition of $\Ss$. We say that ${\mathfrak G}(\Ss)$ is {\em compatible with} a semialgebraic set $\Tt\subset\Ss$ if $\Tt$ is the union of some of the $\Gamma_{k\ell}$.\hfill$\sqbullet$
\end{defn}

\begin{examples}\label{exgs}
(i) The semialgebraic partition ${\mathfrak G}(\Ss)$ of a semialgebraic set $\Ss\subset\R^n$ is not in general a stratification of $\Ss$. Consider for instance the semialgebraic set $\Ss:=\{\y^2-\x^3=0\}\cap(\{\z>0\}\cup\{\z\leq0,\y\geq0\})\subset\R^3$. Then 
$$
\Gamma_{11}:=\{\y^2-\x^3=0,\y>0\},\ \Gamma_{12}:=\{\y^2-\x^3=0,\y<0,\z>0\},\ \Gamma_{21}=\{\x=0,\y=0\}
$$
and ${\mathfrak G}(\Ss)=\{\Gamma_{11},\Gamma_{12},\Gamma_{21}\}$. Observe that $\Gamma_{21}\cap\cl(\Gamma_{12})\neq\varnothing$, but $\Gamma_{21}\not\subset\cl(\Gamma_{12})$. Thus, ${\mathfrak G}(\Ss)$ is not a stratification of $\Ss$.

(ii) If $X\subset\R^n$ is a $d$-dimensional non-singular real algebraic set and $Y\subset X$ is a normal-crossings divisor, then ${\mathfrak G}(Y)$ is a stratification of $Y$.

It is enough to consider local models, that is, $Y:=\{\x_1\cdots\x_r=0\}\subset\R^d$. In fact, we may assume $r=d$, because $Y=Z\times\R^{d-r}$ where $Z:=\{\x_1\cdots\x_r=0\}\subset\R^r$. The semialgebraic partition ${\mathfrak G}(Y)$ of $Y$ is the collection of semialgebraic sets $\Lambda_\ell:=\{\x_1*_10,\ldots,\x_d*_d0\}$ where $*_i\in\{<,=,>\}$ and at least one of the symbols $*_i$ is equal to $=$. The closure of each $\Lambda_\ell$ is a union of finitely many $\Lambda_k$ and consequently ${\mathfrak G}(Y)$ is a stratification of $Y$.

In this case $\Sth(Y)=\Reg(Y)$ and $\Sth(\Sing_\ell(Y))=\Reg(\Sing_\ell(Y))$ for each $\ell\geq1$.

(iii) Let $X\subset\R^n$ be a non-singular algebraic set, $Y\subset X$ is a normal-crossings divisor and $\Ss$ the closure of a union of connected components of $X\setminus Y$. Then ${\mathfrak G}(\Ss)$ is a stratification of $\Ss$ and ${\mathfrak G}(\Ss)$ is compatible with $\partial\Ss=\Ss\setminus\Reg(\Ss)$.

It is enough to consider local models, so we may assume $Y:=\{\x_1\cdots\x_r=0\}\subset\R^d$. We suppose $r=d$, because $Y=Z\times\R^{d-r}$ where $Z:=\{\x_1\cdots\x_r=0\}\subset\R^r$. Thus, $\Ss:=\bigcup_{\veps\in{\mathfrak F}}\Qq_\veps$, where $\Qq_\veps:=\{\veps_1\x_1\geq0,\ldots,\veps_d\x_d\geq0\}$, $\veps:=(\veps_1,\ldots,\veps_d)$ and ${\mathfrak F}\subset\{-1,1\}^d$. The semialgebraic partition ${\mathfrak G}(\Ss)$ of $\Ss$ is a collection of the type $\Gamma_\ell:=\{\x_{i_1}*_{i_1}0,\ldots,\x_{i_\ell}*_{i_\ell}0\}$ where $0\leq\ell\leq d$, $1\leq i_1<\cdots<i_\ell\leq d$ and $*_{i_k}\in\{<,=,>\}$ for $k=1,\ldots,\ell$. The closure of each $\Gamma_\ell$ is a union of finitely many $\Gamma_k$, so ${\mathfrak G}(\Ss)$ is a stratification of $\Ss$. Observe that $\partial\Ss=\Ss\setminus\Reg(\Ss)$ and $\Ss\cap Y$ are unions of finitely many of the sets $\{\x_{i_1}*_{i_1}0,\ldots,\x_{i_\ell}*_{i_\ell}0\}$ with the condition that some of the $*_{i_k}$ are equal to $=$, that is, all of them belong to ${\mathfrak G}(\Ss)$ and ${\mathfrak G}(\Ss)$ is compatible with $\partial\Ss=\Ss\setminus\Reg(\Ss)$.

In this case $\Sth(\Ss)=\Reg(\Ss)$ and $\Sth(\partial^\ell\Ss)=\Reg(\partial^\ell\Ss)$, where $\partial^\ell\Ss:=\partial(\partial^{\ell-1}\Ss)$ for each $\ell\geq2$. This is because $\partial\Ss\subset Y$ and $\partial^\ell\Ss\subset\Sing_{\ell-1}(Y)$ for each $\ell\geq2$.

(iv) If $\Qq\subset\R^n$ is a $d$-dimensional Nash manifold with corners, ${\mathfrak G}(\Qq)={\mathfrak G}(\partial\Qq)\sqcup{\mathfrak G}(\Int(\Qq))$ is a stratification of $\Qq$.

It is enough to apply Theorem \ref{divisorialPre} and (iii). 
\hfill$\sqbullet$
\end{examples}

\begin{defn}
A subset $\Tt\subset\R^n$ is a {\em Nash quasi-manifold with corners} if $\Qq:=\cl(\Tt)$ is a Nash manifold with corners and $\Qq\setminus\Tt$ is a union of elements of the stratification ${\mathfrak G}(\partial\Qq)$.\hfill$\sqbullet$ 
\end{defn}

\subsection{Proof of Theorem \ref{red4}}
We are ready to prove Theorem \ref{red4}. 

\begin{proof}[Proof of Theorem \em \ref{red4}]
The proof is conducted in several steps and subsequent reductions:

\noindent{\sc Step 1. Initial preparation.} 
We embed $\R^m$ in $\R\PP^m$ and the latter in $\R^p$ for $p$ large enough, so we can suppose that $\Ss$ is a bounded chessboard set. We may assume that the previous embedding is a regular map \cite[Prop.2.4.1]{ak2}, which is in particular a Nash map. Thus, $\cl(\Ss)$ is compact and the Zariski closure of $\Ss$ is also compact. By Theorem \ref{ridwell2} we may assume that $\Ss\subset\R^m$ is a checkerboard set, whose Zariski closure $X$ is a $d$-dimensional non-singular compact algebraic subset of $\R^m$. In particular, $\Reg(\Ss)$ is connected and the Zariski closure $Z_1$ of $\partial\Ss:=\cl(\Ss)\setminus\Reg(\Ss)$ is a normal-crossings divisor of $X$. We construct next a semialgebraic partition of $\cl(\Ss)$ as a finite union of Nash manifolds of different dimensions compatible with $\Tt_1:=\cl(\Ss)\setminus\Ss$ and $\Tt_2:=\Ss\setminus\Reg(\Ss)$. Observe that $\Tt_1$ and $\Tt_2$ are disjoint semialgebraic sets, they have dimensions $\leq d-1$ and $\partial\Ss=\Tt_1\sqcup\Tt_2$. 

\begin{figure}[!ht]
\begin{center}
\begin{tikzpicture}[scale=0.75]
\draw[fill=gray!50,opacity=0.4,draw=none] (0,0) -- (0,5) -- (5,5) -- (5,0) -- (0,0);
\draw[fill=white,draw=none] (1,2.5) -- (1,4) -- (2.5,4) -- (2.5,2.5) -- (1,2.5);
\draw[fill=white,draw=none] (2.5,1) -- (4,1) -- (4,2.5) -- (2.5,2.5) -- (2.5,1);

\draw[line width=1pt, dotted] (0,0) -- (0,5);
\draw[line width=1pt] (0,0) -- (2.5,0);
\draw[line width=1pt,dotted] (2.5,0) -- (5,0);

\draw[line width=1pt] (0,5) -- (1,5);
\draw[line width=1pt,dotted] (1,5) -- (3,5);
\draw[line width=1pt] (3,5) -- (5,5);

\draw[line width=1pt] (5,0) -- (5,5);

\draw[line width=1pt] (1,2.5) -- (1,4);
\draw[line width=1pt] (1,4)-- (2.5,4);
\draw[line width=1pt,dotted] (2.5,4) -- (2.5,2.5);
\draw[line width=1pt] (2.5,2.5) -- (1,2.5);

\draw[line width=1pt,dotted] (4,1) -- (2.5,1);
\draw[line width=1pt] (2.5,1) -- (2.5,2.5);
\draw[line width=1pt] (2.5,2.5) -- (4,2.5);
\draw[line width=1pt] (4,1) -- (4,2.5);

\draw[fill=black,draw] (0,0) circle (0.75mm);
\draw[fill=white,draw] (5,0) circle (0.75mm);
\draw[fill=white,draw] (0,5) circle (0.75mm);
\draw[fill=black,draw] (5,5) circle (0.75mm);

\draw[fill=white,draw] (1,5) circle (0.75mm);
\draw[fill=black,draw] (3,5) circle (0.75mm);
\draw[fill=white,draw] (2.5,0) circle (0.75mm);

\draw[fill=white,draw] (1,4) circle (0.75mm);
\draw[fill=black,draw] (1,2.5) circle (0.75mm);
\draw[fill=black,draw] (2.5,2.5) circle (0.75mm);
\draw[fill=white,draw] (2.5,1) circle (0.75mm);
\draw[fill=white,draw] (2.5,4) circle (0.75mm);
\draw[fill=white,draw] (4,1) circle (0.75mm);
\draw[fill=black,draw] (4,2.5) circle (0.75mm);

\draw (1.5,1.5) node {\small$\Ss$};

\draw[fill=gray!50,opacity=0.4,draw=none] (8,0) -- (8,5) -- (13,5) -- (13,0) -- (8,0);
\draw[fill=white,draw=none] (9,2.5) -- (9,4) -- (10.5,4) -- (10.5,2.5) -- (9,2.5);
\draw[fill=white,draw=none] (10.5,1) -- (12,1) -- (12,2.5) -- (10.5,2.5) -- (10.5,1);

\draw[Blue, line width=1pt] (8,0) -- (8,5);
\draw[Red, line width=1pt] (8,0) -- (10.5,0);
\draw[Blue, line width=1pt] (10.5,0) -- (13,0);

\draw[Red, line width=1pt] (8,5) -- (9,5);
\draw[Blue, line width=1pt] (9,5) -- (11,5);
\draw[Red, line width=1pt] (11,5) -- (13,5);
\draw[Red, line width=1pt] (13,0) -- (13,5);

\draw[Red, line width=1pt] (9,2.5) -- (9,4);
\draw[Red, line width=1pt] (9,4)-- (10.5,4);
\draw[Blue, line width=1pt] (10.5,4) -- (10.5,2.5);
\draw[Red, line width=1pt] (10.5,2.5) -- (9,2.5);

\draw[Blue, line width=1pt] (12,1) -- (10.5,1);
\draw[Red, line width=1pt] (10.5,1) -- (10.5,2.5);
\draw[Red, line width=1pt] (10.5,2.5) -- (12,2.5);
\draw[Red, line width=1pt] (12,1) -- (12,2.5);

\draw[fill=Red,draw=none] (8,0) circle (0.75mm);
\draw[fill=Blue,draw=none] (13,0) circle (0.75mm);
\draw[fill=Blue,draw=none] (8,5) circle (0.75mm);
\draw[fill=Red,draw=none] (13,5) circle (0.75mm);

\draw[fill=Blue,draw=none] (9,5) circle (0.75mm);
\draw[fill=Red,draw=none] (11,5) circle (0.75mm);
\draw[fill=Blue,draw=none] (10.5,0) circle (0.75mm);

\draw[fill=Blue,draw=none] (9,4) circle (0.75mm);
\draw[fill=Red,draw=none] (9,2.5) circle (0.75mm);
\draw[fill=Red,draw=none] (10.5,2.5) circle (0.75mm);
\draw[fill=Blue,draw=none] (10.5,1) circle (0.75mm);
\draw[fill=Blue,draw=none] (10.5,4) circle (0.75mm);
\draw[fill=Blue,draw=none] (12,1) circle (0.75mm);
\draw[fill=Red,draw=none] (12,2.5) circle (0.75mm);

\draw (9.5,1.5) node {\small$\Ss$};
\draw (11.5,0.5) node {\small$\Tt_1$};
\draw (12.2,4) node {\small$\Tt_2$};

\end{tikzpicture}
\end{center}
\caption{\small{A bounded checkerboard set $\Ss$ (left) and $\Ss$ with the sets $\Tt_1$ (blue) and $\Tt_2$ (red) coloured (right).}}
\end{figure}

Let $N_1$ be the union of the connected components of $\Reg(\Tt_1)\sqcup\Reg(\Tt_2)$ of dimension $d-1$. Note that the connected components of dimension $d-1$ of $\Reg(\Tt_1)\sqcup\Reg(\Tt_2)$ are in general different from the connected components of dimension $d-1$ of $\Reg(\partial\Ss)$. As $\dim(\Tt_i\setminus\Reg(\Tt_i))\leq d-2$ and $\partial\Ss=\Tt_1\sqcup\Tt_2$, the semialgebraic set $\partial\Ss\setminus N_1$ has dimension $\leq d-2$. Let $Z_2$ be the Zariski closure of $\partial\Ss\setminus N_1$. Each connected component of the Nash manifold $N_1':=\partial\Ss\setminus Z_2=N_1\setminus Z_2$ has dimension $d-1$ and it is contained either in $\Tt_1$ or $\Tt_2$. In addition, $\partial\Ss\setminus N_1'\subset Z_2$ has dimension $\leq d-2$, $Z_2$ is the Zariski closure of $\partial\Ss\setminus N_1'$ (because $\partial\Ss\setminus N_1\subset\partial\Ss\setminus N_1'\subset Z_2$ and $Z_2$ is the Zariski closure of $\partial\Ss\setminus N_1$) and $\cl(N_1')\subset\partial\Ss$.

\noindent{\sc Substep 1.1.} Let us construct recursively: 
{\em\begin{itemize}
\item Pairwise disjoint semialgebraic sets $N_k'\subset\partial\Ss$ that are either Nash manifolds of dimension $d-k$, whose connected components are contained in either $\Tt_1$ or $\Tt_2$, or the empty set, 
\item Real algebraic sets $Z_k$ of dimension $\leq d-k$,
\end{itemize}
such that :
\begin{itemize}
\item[(1)] $N_k'\subset Z_k$ is an open semialgebraic subset of $Z_k$, 
\item[(2)] $Z_k$ is the Zariski closure of $\partial\Ss\setminus (N_1'\cup\cdots\cup N_{k-1}')$, 
\item[(3)] $N_k'\cap Z_{k+1}=\varnothing$ and $Z_{k+1}\subset Z_k$, 
\item[(4)]$\partial\Ss\setminus (N_1'\cup\cdots\cup N_k')\subset Z_{k+1}$ and $\cl(N_k')\subset\partial\Ss\setminus (N_1'\cup\cdots\cup N_{k-1}')$.
\end{itemize}
In addition, $\partial\Ss=N_1'\sqcup\cdots\sqcup N_d'$ and $Z_k$ is the Zariski closure of $N'_k\sqcup\cdots\sqcup N_d'$ for $k=1,\ldots,d$. 
}

Suppose that we have constructed the Nash manifolds $N_1',\ldots,N_{k-1}'$ and the real algebraic sets $Z_1,\ldots,Z_k$ satisfying the required conditions and let us construct $N_k'$ and $Z_{k+1}$. Let $N_k$ be the union of the connected components of dimension $d-k$ of $\Reg(\Tt_1\cap Z_k)\sqcup\Reg(\Tt_2\cap Z_k)$ or the empty set if $\dim((Z_k\cap\cl(\Ss))\setminus\Ss)<d-k$. Observe that $N_k$ is an open semialgebraic subset of $Z_k$. As $\dim((\Tt_i\cap Z_k)\setminus\Reg(\Tt_i\cap Z_k))\leq d-(k+1)$, the semialgebraic set $\partial\Ss\setminus(N_1'\cup\cdots\cup N_{k-1}'\cup N_k)\subset(\partial\Ss\cap Z_k)\setminus N_k$ has dimension $\leq d-(k+1)$. Let $Z_{k+1}$ be the Zariski closure of $\partial\Ss\setminus(N_1'\cup\cdots\cup N_{k-1}'\cup N_k)$, which has dimension $\leq d-(k+1)$. 

In case $N_k=\varnothing$, then $\dim(Z_k)<d-k$ and $Z_{k+1}=Z_k$. Suppose $N_k\neq\varnothing$. Each connected component of the Nash manifold $N_k':=\partial\Ss\setminus(N_1'\cup\cdots\cup N_{k-1}'\cup Z_{k+1})=N_k\setminus Z_{k+1}$ has dimension $d-k$ and it is contained in either $\Tt_1$ or $\Tt_2$. In addition, $N_k'$ is an open semialgebraic subset of $Z_k$, $\partial\Ss\setminus(N_1'\cup\cdots\cup N_k')\subset Z_{k+1}$ (so it has dimension $\leq d-(k+1)$) and $Z_{k+1}$ is the Zariksi closure of $\partial\Ss\setminus(N_1'\cup\cdots\cup N_k')$ (because $\partial\Ss\setminus(N_1'\cup\cdots\cup N_{k-1}'\cup N_k)\subset\partial\Ss\setminus(N_1'\cup\cdots\cup N_{k-1}'\cup N_k')\subset Z_{k+1}$ and $Z_{k+1}$ is the Zariski closure of $\partial\Ss\setminus(N_1'\cup\cdots\cup N_{k-1}'\cup N_k)$).

As $Z_k\subset Z_\ell$ if $\ell<k$ and $N_\ell'\cap Z_{\ell+1}=\varnothing$, we deduce $N_\ell'\cap Z_k=\varnothing$ if $\ell<k$ and $\cl(N_k')\subset Z_k$ does not meet $N_1'\cup\cdots\cup N_{k-1}'$, so $\cl(N_k')\subset\partial\Ss\setminus(N_1'\cup\cdots\cup N_{k-1}')$. Thus, we have constructed the Nash manifolds $N'_k$ of dimension $d-k$ (or the empty set) and the real algebraic sets $Z_k$ of dimension $\leq d-k$ satisfying the required conditions for $k=1,\ldots,d-1$. In particular,
$$
\cl(N_k')\setminus N_k'\subset\partial\Ss\setminus(N_1'\cup\cdots\cup N_k')\subset Z_{k+1}
$$
for $k=1,\ldots,d$ (and $Z_{d+1}:=\varnothing$). 

\noindent{\sc Substep 1.2.} For technical reasons we reorder the indices of the real algebraic sets $Z_k$ and the Nash manifolds $N_k'$.

Define the real algebraic set $T_k:=Z_{d-k}$, which has dimension $\leq k$, and the Nash manifold $M_k:=N'_{d-k}$, which is either empty or has dimension $k$, for $k=0,\ldots,d-1$. The real algebraic set $T_k$ is the Zariski closure of $M_1\sqcup\cdots\sqcup M_k$ and each connected component of $M_k$ is either contained in $\Tt_1$ or $\Tt_2$.

If $M_0\neq\varnothing$, then $M_0$ is a finite set and $M_0=T_0$. Otherwise, if $m$ is the least $k$ such that $M_k\neq\varnothing$, then $\cl(M_m)\setminus M_m=\varnothing$ (so $M_m$ is closed in $\R^n$) and $T_m$ is the Zariski closure of $M_m$. As $M_m$ is open and closed in $T_m$, we deduce that $M_m$ is a (finite) union of connected components of $T_m$ of dimension $m$ and $M_k=\varnothing$ for $0\leq k<m$. 

Observe that $\dim(T_k)\leq k$, $T_k\subset T_{k+1}$, $M_k\subset T_k$, $M_k\cap T_{k-1}=\varnothing$, $\cl(M_k)\setminus M_k\subset M_0\sqcup\cdots\sqcup M_{k-1}\subset T_{k-1}$. In addition,
\begin{equation*}
\begin{split}
\partial\Ss\cap T_k&=N_{d-k}'\sqcup\cdots\sqcup N_d'=M_0\sqcup\cdots\sqcup M_k,\\
\Tt_1\cap T_k&=(\cl(\Ss)\setminus\Ss)\cap T_k=(M_0\cap\Tt_1)\sqcup\cdots\sqcup(M_k\cap\Tt_1),\\
\Tt_2\cap T_k&=(\Ss\setminus\Reg(\Ss))\cap T_k=(M_0\cap\Tt_2)\sqcup\cdots\sqcup(M_k\cap\Tt_2),
\end{split}
\end{equation*}
and each intersection $M_\ell\cap\Tt_i$ is the union of the connected components of the Nash manifold $M_\ell$ contained in $\Tt_i$ for $i=1,2$. Observe that $M_\ell=(M_\ell\cap\Tt_1)\sqcup(M_\ell\cap\Tt_2)$ for $\ell=0,\ldots,d-1$. 

\noindent{\sc Step 2. Initial algebraic resolution procedure.}
Denote $E_{-1}:=\varnothing$, $X_{-1}:=X$, $T_{-1,i}:=T_i$ and $Y_0:=T_0$. By Theorem \ref{hi2} there exists a non-singular (compact) real algebraic set $X_0$ and a proper surjective polynomial map $g_0:X_0\to X_{-1}$ such that $E_0:=g_0^{-1}(E_{-1}\cup Y_0)$ is a normal-crossings divisor of $X_0$ and the restriction $g_0|_{X_0\setminus E_0}:X_0\setminus E_0\to X\setminus(E_{-1}\cup Y_0)$ is a biregular diffeomorphism. In fact, $g_0$ is a composition of finitely many blow-ups whose non-singular centers are contained in $T_0\subset T_i$ and have dimension $\leq\min\{\dim(Y_0),d-2\}$. Denote $T_{0i}:=\cl(g_0^{-1}(T_{-1,i}\setminus(E_{-1}\cup Y_0)))$ the strict transform of $T_{-1,i}$ under $g_0$ for $i\geq1$, which is by Remark \ref{exst}(i) a real algebraic set of the same dimension as $T_i$ and denote $Y_1:=T_{01}$. Observe that no $T_{0i}$ is contained in the real algebraic set $E_0$ for $i\geq1$. In particular, $\dim(T_{0i}\cap E_0)<\dim(T_{0i})$ for each $i\geq1$ and no irreducible component of $Y_1$ is contained in $E_0$. We keep on similarly with $E_0\cup Y_1$. 

We proceed recursively and in the step $k\leq d-1$ we find by Theorem \ref{hi2} for the couple $E_{k-1}:=g_{k-1}^{-1}(E_{k-2}\cup Y_{k-1})$ and $Y_k:=T_{k-1,k}$ a non-singular (compact) real algebraic set $X_k$ and a proper surjective polynomial map $g_k:X_k\to X_{k-1}$ such that $E_k:=g_k^{-1}(E_{k-1}\cup Y_k)$ is a normal-crossings divisor of $X_k$ and the restriction $g_k|_{X_k\setminus E_k}:X_k\setminus E_k\to X_{k-1}\setminus(E_{k-1}\cup Y_k)$ is a biregular diffeomorphism. In fact, $g_k$ is a composition of finitely many blow-ups whose non-singular centers are contained in $Y_k=T_{k-1,k}$ and have dimension $\leq\min\{\dim(Y_k),d-2\}$. Denote 
$$
T_{ki}:=\cl(g_k^{-1}(T_{k-1,i}\setminus(E_{k-1}\cup Y_k)))
$$ 
the strict transform of $T_{k-1,i}$ under $g_k$ for $i\geq k+1$, which is by Remark \ref{exst}(i) a real algebraic set of the same dimension as $T_i$ and let $Y_{k+1}:=T_{k,k+1}$. Observe that no $T_{ki}$ is contained in the real algebraic set $E_k$ for $i\geq k+1$. In particular, $\dim(T_{ki}\cap E_k)<\dim(T_{ki})$ for each $i\geq k+1$ and no irreducible component of $Y_{k+1}$ is contained in $E_k$. Observe that $Y_d=\varnothing$ (because $T_d=\varnothing$, so $T_{d-1,d}=\varnothing$) and $E_{d-1}=(g_0\circ\cdots\circ g_{d-1})^{-1}(T_{d-1})$ is a normal-crossings divisor.

\noindent{\sc Step 3. Properties of the strict transform.} 
We begin studying some properties of the composition $g:=g_0\circ\cdots\circ g_{d-1}:X_{d-1}\to X$. We claim: {\em There exists a real algebraic subset $R\subset X$ of dimension $\leq d-2$ such that the restriction $g|_{X_{d-1}\setminus g^{-1}(R)}:X_{d-1}\setminus g^{-1}(R)\to X\setminus R$ is a Nash diffeomorphism.}

For each $k=0,\ldots,d-1$ the polynomial map $g_k:X_k\to X_{k-1}$ (where $X_{-1}:=X$) is the composition of finitely many blow-ups whose centers have dimensions $\leq d-2$. Recall that the blow-up $b:\widehat{V}\to V$ of a $d$-dimensional non-singular real algebraic set $V$ with center a non-singular real algebraic subset $W$ of dimension $\leq d-2$ provides a Nash diffeomorphism $b|_{\widehat{V}\setminus b^{-1}(W)}:\widehat{V}\setminus b^{-1}(W)\to V\setminus W$. As the image of a semialgebraic set of dimension $\leq d-2$ is a semialgebraic set of dimension $\leq d-2$, we conclude recursively that there exists a semialgebraic set $R_{k-1}\subset X_{k-1}$ of dimension $\leq d-2$ such that $g_k|_{X_k\setminus g_k^{-1}(R_{k-1})}:X_k\setminus g_k^{-1}(R_{k-1})\to X_{k-1}\setminus R_{k-1}$ is a Nash diffeomorphism and $(g_0\circ\cdots\circ g_{k-1})(R_{k-1})$ has dimension $\leq d-2$. 

Let $R$ be the Zariski closure of $R_{-1}\cup\bigcup_{k=0}^{d-2}(g_0\circ\cdots\circ g_{k})(R_k)$, which is a real algebraic set of dimension $\leq d-2$. The restriction $g|_{X_{d-1}\setminus g^{-1}(R)}:X_{d-1}\setminus g^{-1}(R)\to X\setminus R$ is a Nash diffeomorphism, as claimed.

\noindent{\sc Substep 3.1.} Define $\Ss^*:=g^{-1}(\Ss)\cap\cl(g^{-1}(\Ss)\setminus E_{d-1})$ the strict transform of $\Ss$ under $g$ and $\Ss_k:=g_k^{-1}(\Ss_{k-1})\cap\cl(g_k^{-1}(\Ss_{k-1})\setminus E_k)$ the strict transform of $\Ss_{k-1}$ under $g_k$ for $k=0,\ldots,d-1$, where $\Ss_{-1}:=\Ss$. We claim: {\em $\Ss^*=\Ss_{d-1}$}. Let us prove by induction on $\ell$ that: {\em $\Ss_\ell^*:=(g_0\circ\cdots\circ g_\ell)^{-1}(\Ss)\cap\cl((g_0\circ\cdots\circ g_\ell)^{-1}(\Ss)\setminus E_\ell)$ equals $\Ss_\ell$ for each $\ell=0,\ldots,d-1$.} 

If $\ell=0$, then $\Ss^*_0=\Ss_0$. Suppose the result true for $\ell-1$, that is, $\Ss^*_{\ell-1}=\Ss_{\ell-1}$ and let us check $\Ss^*_\ell=\Ss_\ell$. Denote $h_{\ell-1}:=g_0\circ\cdots\circ g_{\ell-1}$. We have 
$$
\Ss_\ell^*=(g_0\circ\cdots\circ g_\ell)^{-1}(\Ss)\cap\cl((g_0\circ\cdots\circ g_\ell)^{-1}(\Ss)\setminus E_\ell)
=g_\ell^{-1}(h_{\ell-1}^{-1}(\Ss))\cap\cl(g_\ell^{-1}(h_{\ell-1}^{-1}(\Ss))\setminus E_\ell)
$$
is the strict transform of $h_{\ell-1}^{-1}(\Ss)$ under $g_\ell$. It holds $\Ss^*_{\ell-1}=h_{\ell-1}^{-1}(\Ss)\cap\cl(h_{\ell-1}^{-1}(\Ss)\setminus E_{\ell-1})$ and recall that $E_\ell=g_\ell^{-1}(E_{\ell-1}\cup Y_\ell)$.

The strict transform of $\Ss^*_{\ell-1}$ under $g_\ell$ is
\begin{multline*}
g_\ell^{-1}(\Ss^*_{\ell-1})\cap\cl(g_\ell^{-1}(\Ss^*_{\ell-1})\setminus E_\ell)
=g_\ell^{-1}(h_{\ell-1}^{-1}(\Ss))\cap g_\ell^{-1}(\cl(h_{\ell-1}^{-1}(\Ss)\setminus E_{\ell-1}))\\
\cap\cl((g_\ell^{-1}(h_{\ell-1}^{-1}(\Ss))\setminus E_\ell)\cap (g_\ell^{-1}(\cl(h_{\ell-1}^{-1}(\Ss)\setminus E_{\ell-1})\setminus E_\ell)).
\end{multline*}

As $h_{\ell-1}^{-1}(\Ss)\setminus E_{\ell-1}\subset\cl(h_{\ell-1}^{-1}(\Ss)\setminus E_{\ell-1})$ and $g_\ell^{-1}(E_{\ell-1})\subset E_\ell$, we deduce
$$
g_\ell^{-1}(h_{\ell-1}^{-1}(\Ss))\setminus E_\ell\subset g_\ell^{-1}(\cl(h_{\ell-1}^{-1}(\Ss)\setminus E_{\ell-1})).
$$
Consequently, 
\begin{align*}
\cl(g_\ell^{-1}(h_{\ell-1}^{-1}(\Ss))\setminus E_\ell)&\subset g_\ell^{-1}(\cl(h_{\ell-1}^{-1}(\Ss)\setminus E_{\ell-1})),\\
g_\ell^{-1}(h_{\ell-1}^{-1}(\Ss))\setminus E_\ell&\subset g_\ell^{-1}(\cl(h_{\ell-1}^{-1}(\Ss)\setminus E_{\ell-1}))\setminus E_\ell.
\end{align*}
Thus,
\begin{align*}
&\cl((g_\ell^{-1}(h_{\ell-1}^{-1}(\Ss))\setminus E_\ell)\cap (g_\ell^{-1}(\cl(h_{\ell-1}^{-1}(\Ss)\setminus E_{\ell-1})\setminus E_\ell))=\cl(g_\ell^{-1}(h_{\ell-1}^{-1}(\Ss))\setminus E_\ell),\\
&\cl(g_\ell^{-1}(h_{\ell-1}^{-1}(\Ss))\setminus E_\ell)\cap g_\ell^{-1}(\cl(h_{\ell-1}^{-1}(\Ss)\setminus E_{\ell-1}))=\cl(g_\ell^{-1}(h_{\ell-1}^{-1}(\Ss))\setminus E_\ell).
\end{align*}
We conclude
$$
g_\ell^{-1}(\Ss^*_{\ell-1})\cap\cl(g_\ell^{-1}(\Ss^*_{\ell-1})\setminus E_\ell)=g_\ell^{-1}(h_{\ell-1}^{-1}(\Ss))\cap\cl(g_\ell^{-1}(h_{\ell-1}^{-1}(\Ss))\setminus E_\ell)=\Ss_\ell^*.
$$
As by induction hypothesis $\Ss^*_{\ell-1}=\Ss_{\ell-1}$, we have $\Ss_\ell^*=g_\ell^{-1}(\Ss_{\ell-1})\cap\cl(g_\ell^{-1}(\Ss_{\ell-1})\setminus E_\ell)=\Ss_\ell$, as claimed.

\noindent{\sc Substep 3.2.} We prove next: {\em $\Reg(\Ss^*)$ is a connected $d$-dimensional Nash manifold and it contains the inverse image $g^{-1}(\Reg(\Ss)\setminus R)$ as a connected dense open semialgebraic subset.}

As $\Reg(\Ss)$ is a connected Nash manifold and $\dim(R)\leq d-2$, also $\Reg(\Ss)\setminus R$ is a connected Nash manifold. As the restriction $g|_{X_{d-1}\setminus g^{-1}(R)}:X_{d-1}\setminus g^{-1}(R)\to X\setminus R$ is a Nash diffeomorphism, $g^{-1}(\Reg(\Ss)\setminus R)$ is a connected Nash manifold. As $\dim(R)\leq d-2$, $\Ss$ is pure dimensional of dimension $d$ and $\Reg(\Ss)$ is dense in $\Ss$, we deduce $\Reg(\Ss)\setminus R$ is a dense open semialgebraic subset of $\Ss$, so $g^{-1}(\Reg(\Ss)\setminus R)$ is a dense open semialgebraic subset of $g^{-1}(\Ss\setminus R)$. 

As $\dim(E_{d-1})=d-1$, $g^{-1}(R)\subset E_{d-1}$ and $g^{-1}(\Ss\setminus R)$ is pure dimensional of dimension $d$, we deduce that $g^{-1}(\Reg(\Ss))\setminus E_{d-1}=g^{-1}(\Reg(\Ss)\setminus R)\setminus E_{d-1}$ is a dense open semialgebraic subset of $g^{-1}(\Ss)\setminus E_{d-1}$, so 
\begin{equation}\label{bel1}
\cl(g^{-1}(\Reg(\Ss)\setminus R)\setminus E_{d-1})=\cl(g^{-1}(\Reg(\Ss))\setminus E_{d-1})=\cl(g^{-1}(\Ss)\setminus E_{d-1}).
\end{equation}
As $g^{-1}(\Reg(\Ss)\setminus R)$ is a $d$-dimensional Nash manifold and $\dim(E_{d-1})=d-1$, we deduce that $g^{-1}(\Reg(\Ss)\setminus R)\setminus E_{d-1}$ is dense in $g^{-1}(\Reg(\Ss)\setminus R)$, so 
$$
\cl(g^{-1}(\Reg(\Ss)\setminus R))=\cl(g^{-1}(\Reg(\Ss)\setminus R)\setminus E_{d-1}).
$$
Consequently, 
\begin{multline*}
g^{-1}(\Reg(\Ss)\setminus R)=g^{-1}(\Reg(\Ss)\setminus R)\cap\cl(g^{-1}(\Reg(\Ss)\setminus R)\setminus E_{d-1})\\
\subset g^{-1}(\Reg(\Ss))\cap\cl(g^{-1}(\Reg(\Ss))\setminus E_{d-1})=\Ss^*.
\end{multline*}
As $\Ss$ is connected by analytic paths, we deduce by \cite[Lem.7.16]{fe3} that also $\Ss^*$ is connected by analytic paths, so in particular $\Ss^*$ is pure dimensional (of dimension $d$). Thus, 
\begin{equation}\label{sed}
\Ss^*\setminus E_{d-1}=g^{-1}(\Ss)\setminus E_{d-1}
\end{equation}
is a dense open semialgebraic subset of $\Ss^*$. As $g^{-1}(\Reg(\Ss)\setminus R)$ is a dense open semialgebraic subset of $g^{-1}(\Ss\setminus R)$ and $g^{-1}(R)\subset E_{d-1}$, we deduce that $g^{-1}(\Reg(\Ss)\setminus R)$ is a dense open semialgebraic subset of $\Ss^*$. As $g^{-1}(\Reg(\Ss)\setminus R)$ is a $d$-dimensional connected Nash manifold, $g^{-1}(\Reg(\Ss)\setminus R)\subset\Reg(\Ss^*)$, so $\Reg(\Ss^*)$ is connected because it contains a dense connected subset. 

\noindent{\sc Substep 3.3.} Let us prove: {\em $\partial\Ss^*\subset E_{d-1}$.} 

It holds $\partial\Ss=\cl(\Ss)\setminus\Reg(\Ss)\subset\bigcup_{k=0}^{d-1}T_k$ and 
$$
g^{-1}(\partial\Ss)=g^{-1}(\cl(\Ss)\setminus\Reg(\Ss))\subset\bigcup_{k=0}^{d-1}g^{-1}(T_k).
$$
Recall that $T_{ki}$ is the strict transform of $T_{k-1,i}$ under $g_k$ for $i\geq k$, $Y_k:=T_{k-1,k}$, $E_k=g_k^{-1}(E_{k-1}\cup Y_k)$ for $k\geq 1$ and $E_0=g_0^{-1}(T_0)$. Thus,
$$
(g_0\circ\cdots\circ g_{k-1})^{-1}(T_k)\subset T_{k-1,k}\cup E_{k-1}\subset Y_k\cup E_{k-1},
$$
so $(g_0\circ\cdots\circ g_k)^{-1}(T_k)\subset g_k^{-1}(Y_k\cup E_{k-1})=E_k$ and $g^{-1}(T_k)=(g_0\circ\cdots\circ g_{d-1})^{-1}(T_k)\subset E_{d-1}$ for each $k=0,\ldots,d-1$. Thus, 
\begin{equation}\label{bel2}
g^{-1}(\partial\Ss)=g^{-1}(\cl(\Ss)\setminus\Reg(\Ss))\subset\bigcup_{k=0}^{d-1}g^{-1}(T_k)\subset E_{d-1}.
\end{equation}
We deduce using that $g^{-1}(R)\subset E_{d-1}$ and $g^{-1}(\Reg(\Ss)\setminus R)\subset\Reg(\Ss^*)$
\begin{equation*}
\begin{split}
\cl(\Ss^*)\setminus E_{d-1}&\subset\cl(g^{-1}(\Ss))\cap\cl(g^{-1}(\Ss)\setminus E_{d-1})\setminus E_{d-1}\\
&=\cl(g^{-1}(\Ss))\setminus E_{d-1}\subset g^{-1}(\cl(\Ss))\setminus E_{d-1}\\
&=(g^{-1}(\Reg(\Ss))\setminus E_{d-1})\cup(g^{-1}(\partial\Ss)\setminus E_{d-1}))\\
&=g^{-1}(\Reg(\Ss))\setminus E_{d-1}\subset g^{-1}(\Reg(\Ss)\setminus R)\subset\Reg(\Ss^*),
\end{split}
\end{equation*}
so $\partial\Ss^*=\cl(\Ss^*)\setminus\Reg(\Ss^*)\subset E_{d-1}$. 

\noindent{\sc Substep 3.4.} Let $\Tt_1^*:=\cl(\Ss^*)\setminus\Ss^*$ and $\Tt_2^*:=\Ss^*\setminus\Reg(\Ss^*)$. Let us check: $g(\Tt_i^*)\subset\Tt_i$ {\em for} $i=1,2$.

Recall that by \eqref{bel1} and \eqref{bel2} we have 
\begin{equation}\label{ccc}
\cl(g^{-1}(\Ss)\setminus E_{d-1})=\cl(g^{-1}(\cl(\Ss))\setminus E_{d-1})=\cl(g^{-1}(\Reg(\Ss))\setminus E_{d-1}).
\end{equation}
As $\Ss$ is connected by analytic paths, $\cl(\Ss)$ is also connected by analytic paths \cite[Lem.7.4]{fe3}. Thus, the strict transform $\cl(\Ss)^*$ of $\cl(\Ss)$ under $g$ is connected by analytic paths \cite[Lem.7.16]{fe3}, so it is pure dimensional of dimension $d$. Thus, 
$$
\cl(\Ss)^*\setminus E_{d-1}=g^{-1}(\cl(\Ss))\setminus E_{d-1}=g^{-1}(\Ss)\setminus E_{d-1}\subset\Ss^*\subset\cl(\Ss^*)
$$
is a dense subset of $\cl(\Ss)^*$. As $\cl(\Ss)^*$ is a closed set that contains $\Ss^*$, we conclude that $\cl(\Ss^*)=\cl(\Ss)^*$ is the strict transform under $g$ of $\cl(\Ss)$. Consequently, $\Tt_1^*=\cl(\Ss^*)\setminus\Ss^*=g^{-1}(\cl(\Ss)\setminus\Ss)\cap\cl(g^{-1}(\Ss)\setminus E_{d-1})\subset g^{-1}(\Tt_1)$, so $g(\Tt_1^*)\subset\Tt_1$.

In addition, the strict transform $\Reg(\Ss)^*$ of $\Reg(\Ss)$ under $g$ is 
$$
g^{-1}(\Reg(\Ss))\cap\cl(g^{-1}(\Reg(\Ss))\setminus E_{d-1}).
$$
As $g^{-1}(\Reg(\Ss))$ is pure dimensional of dimension $d$, because it is an open semialgebraic subset of $X_{d-1}$, and $E_{d-1}$ has dimension $d-1$, we deduce that $g^{-1}(\Reg(\Ss))\setminus E_{d-1}$ is dense in $g^{-1}(\Reg(\Ss))$, so $\Reg(\Ss)^*=g^{-1}(\Reg(\Ss))$ is an open semialgebraic subset of $X_{d-1}$. Thus, $\Reg(\Ss)^*\subset\Reg(\Ss^*)$ and by \eqref{ccc}
\begin{equation*}
\Tt_2^*=\Ss^*\setminus\Reg(\Ss^*)\subset\Ss^*\setminus\Reg(\Ss)^*=g^{-1}(\Ss\setminus\Reg(\Ss))\cap\cl(g^{-1}(\Ss)\setminus E_{d-1})\subset g^{-1}(\Tt_2),
\end{equation*}
so $g(\Tt_2^*)\subset\Tt_2$.

\noindent{\sc Substep 3.5.} Let $Z$ be the Zariski closure of $\partial\Ss^*:=\cl(\Ss^*)\setminus\Reg(\Ss^*)$ in the non-singular (compact) real algebraic set $\ol{\Ss^*}^{\zar}$. We claim: {\em $Z$ is a normal-crossings divisor of $X_{d-1}=\ol{\Ss^*}^{\zar}$, $\Ss^*$ is a checkerboard set, $\Tt_1^*:=\cl(\Ss^*)\setminus\Ss^*$, $\Tt_2^*:=\Ss^*\setminus\Reg(\Ss^*)$, $\Tt_3^*:=\cl(\Ss^*)\setminus\Reg(\cl(\Ss^*))$ and $\Tt_4^*:=\Reg(\cl(\Ss^*))\setminus\Ss^*$ are unions of elements of the stratification ${\mathfrak G}(Z)$}.

We prove first: {\em $\Tt_1^*$ and $\Tt_2^*$ are unions of elements of ${\mathfrak G}(E_{d-1})$}.

Define $T_{-1}:=\varnothing$. Let $E_{d-1}^k$ be the Zariski closure of $g^{-1}(T_k\setminus T_{k-1})$ for $k=0,\ldots,d-1$, which is the union of the irreducible components of $g^{-1}(T_k)$ that are not contained in $g^{-1}(T_{k-1})$. We have $E_{d-1}=\bigcup_{k=0}^{d-1}E_{d-1}^k$ and each irreducible component of $E_{d-1}$ is an irreducible component of $E_{d-1}^k$ for exactly one $k=0,\ldots,d-1$. Conversely, each irreducible component of $E_{d-1}^k$ is an irreducible component of $E_{d-1}$, because if $H$ is an irreducible component of $E_{d-1}^k$, then $g(H)\subset T_k$ and $g(H)\not\subset T_\ell$ if $\ell<k$. 

Let $W\in{\mathfrak G}(\partial\Ss^*)$ and $L$ be the Zariski closure of $W$. As $W$ is a connected Nash manifold, $L$ is an irreducible real algebraic set. Observe that $L$ is an irreducible component of $\Sing_j(E_{d-1})$ for some $j\geq0$ and $W$ is a connected component of $\Reg(\Sing_j(E_{d-1}))$, so $W\in{\mathfrak G}(E_{d-1})$. Let $k\geq0$ be such that $g(L)\subset T_k$, but $g(L)\not\subset T_{k-1}$. Then $L\subset E_{d-1}^\ell$ for $\ell\geq k$, but $L\not\subset\bigcup_{\ell=0}^{k-1}E_{d-1}^\ell=g^{-1}(T_{k-1})$. Observe that $W$ is a connected component of $L\setminus g^{-1}(T_{k-1})$, because 
$$
L\cap\Sing(\Sing_j(E_{d-1}))=L\cap g^{-1}(T_{k-1})
$$ 
(recall that $L\subset E_{d-1}^\ell$ for $\ell\geq k$). Then $g(W)\subset\partial\Ss\cap T_k\setminus T_{k-1}=M_k$ is connected. As each connected component of $M_k$ is contained in either $\Tt_1$ or $\Tt_2$, we deduce either $g(W)\subset\Tt_1$ or $g(W)\subset\Tt_2$. If $W\cap\Tt_i^*\neq\varnothing$, then $W\subset\Tt_i^*$, because otherwise also $W\cap\Tt_j^*\neq\varnothing$ for $j\in\{1,2\}\setminus\{i\}$ and $g(W)$ meets $\Tt_1$ and $\Tt_2$, which is a contradiction because $g(W)$ is contained either in $\Tt_1$ or $\Tt_2$. Thus, $W$ is contained either in $\Tt_1^*$ or $\Tt_2^*$. As $\partial\Ss^*=\Tt_1^*\sqcup\Tt_2^*$ is covered by the elements of ${\mathfrak G}(\partial\Ss^*)\subset{\mathfrak G}(E_{d-1})$, we conclude that both $\Tt_1^*$ and $\Tt_2^*$ are unions of elements of ${\mathfrak G}(E_{d-1})$.

We show next: {\em $\Tt_3^*$ and $\Tt_4^*$ are unions of elements of ${\mathfrak G}(E_{d-1})$}.

By \eqref{sed} and \eqref{bel2} 
$$
\Ss^*\setminus E_{d-1}=g^{-1}(\Ss)\setminus E_{d-1}=g^{-1}(\Reg(\Ss))\setminus E_{d-1}=g^{-1}(\cl(\Ss))\setminus E_{d-1}=\cl(\Ss^*)\setminus E_{d-1}
$$ 
is an open and closed subset of the Nash manifold $X_{d-1}\setminus E_{d-1}$, so it is a union of connected components of the Nash manifold $X_{d-1}\setminus E_{d-1}$. Thus, $\cl(\Ss^*)\setminus E_{d-1}\subset\Reg(\cl(\Ss^*))$, so $\Tt_3^*:=\cl(\Ss^*)\setminus\Reg(\cl(\Ss^*))$ is a $(d-1)$-dimensional semialgebraic subset contained in $E_{d-1}$ (because $\cl(\Ss^*)$ is the closure of a finite union of connected components of $X_{d-1}\setminus E_{d-1}$). In fact, using (local) coordinates, one realizes that both $\Tt_3^*$ and $\cl(\Ss^*)\cap E_{d-1}$ are unions of elements of ${\mathfrak G}(E_{d-1})$ (see Examples \ref{exgs}(iii)). Thus, 
\begin{align*}
\Tt_1^*\cap\Tt_3^*&=(\cl(\Ss^*)\setminus\Ss^*)\cap(\cl(\Ss^*)\setminus\Reg(\cl(\Ss^*)))=\cl(\Ss^*)\setminus(\Ss^*\cup\Reg(\cl(\Ss^*)))\\
&=(\cl(\Ss^*)\cap E_{d-1})\setminus((\Ss^*\cup\Reg(\cl(\Ss^*)))\cap E_{d-1})
\end{align*}
is a union of elements of ${\mathfrak G}(E_{d-1})$, so $(\Ss^*\cup\Reg(\cl(\Ss^*)))\cap E_{d-1}$ is also a union of elements of ${\mathfrak G}(E_{d-1})$. As 
$$
\Ss^*\cap E_{d-1}=(\cl(\Ss^*)\cap E_{d-1})\setminus(\cl(\Ss^*)\setminus\Ss^*)=(\cl(\Ss^*)\cap E_{d-1})\setminus\Tt_1^*
$$ 
is a union of elements of ${\mathfrak G}(E_{d-1})$, we conclude
\begin{multline*}
\Tt_4^*:=(\Reg(\cl(\Ss^*))\setminus\Ss^*)=(\Reg(\cl(\Ss^*))\setminus\Ss^*)\cap E_{d-1}\\
=((\Ss^*\cup\Reg(\cl(\Ss^*)))\cap E_{d-1})\setminus(\Ss^*\cap E_{d-1})
\end{multline*}
is a union of elements of ${\mathfrak G}(E_{d-1})$. 

\noindent{\sc Current situation (1).} 
The Zariski closure of $\Ss^*$ is $X_{d-1}$, which is a non-singular (compact) real algebraic set, and the Zariski closure of $\partial\Ss^*$ is a union of irreducible components of $E_{d-1}$, which is a normal-crossings divisor of $X_{d-1}$. As $\Ss^*$ is the strict transform of $\Ss$ under $g$ and $\Ss$ is pure dimensional of dimension $d$, the restriction $g|_{\Ss^*}:\Ss^*\to\Ss$ is a proper surjective map. Take $\Rr:=g(E_{d-1})\cap\Ss$, which has dimension $\leq d-1$ and observe that $g|_{\Ss^*\setminus g^{-1}(\Rr)}:\Ss^*\setminus g^{-1}(\Rr)\to\Ss\setminus\Rr$ is a Nash diffeomorphism, because $g|_{X_{d-1}\setminus E_{d-1}}:X_{d-1}\setminus E_{d-1}\to X\setminus T_{d-1}$ is a Nash diffeomorphism. As $\Ss^*\setminus g^{-1}(\Rr)=\Ss^*\setminus E_{d-1}\subset\Ss^*\setminus\partial\Ss^*$ is a Nash manifold, its image $\Ss\setminus\Rr$ under $g|_{X_{d-1}\setminus E_{d-1}}$ is also a Nash manifold. It still remains to improve the construction to achieve that $\cl(\Ss^*)$ is a Nash manifold with corners. 

At this point we assume that the initial situation is the one quoted concerning $\Ss^*$. For the sake of simplicity we reset all the previous notations above to continue the proof.

\noindent{\sc Step 4. First drilling desingularization procedure.} 
We assume: {\em $\Ss$ is a checkerboard set (and $\Reg(\Ss)$ is in particular a connected Nash manifold), $\cl(\Ss)$ is compact, the Zariski closure $X$ of $\Ss$ is a non-singular (compact) real algebraic set, the Zariski closure $Z$ of $\cl(\Ss)\setminus\Reg(\Ss)$ is a normal-crossings divisor of $X$, the semialgebraic sets $\Tt_1:=\cl(\Ss)\setminus\Ss$, $\Tt_2:=\Ss\setminus\Reg(\Ss)$, $\Tt_3:=\cl(\Ss)\setminus\Reg(\cl(\Ss))$ and $\Tt_4:=\Reg(\cl(\Ss))\setminus\Ss$ are unions of elements of the stratification ${\mathfrak G}(Z)$}. 

By Theorem \ref{red2} applied to $\cl(\Ss)$ there exist:
\begin{itemize}
\item[{\rm(i)}] A $d$-dimensional (compact) irreducible non-singular real algebraic set $X'$ and a normal-crossings divisor $Z'\subset X'$.
\item[{\rm(ii)}] A connected Nash manifold with corners $\Qq\subset X'$ (which is a closed subset of $X'$) whose boundary $\partial\Qq$ has $Z'$ as its Zariski closure.
\item[\rm{(iii)}] A polynomial map $g:\R^n\to\R^m$ such that the restriction $g|_{\Qq}:\Qq\to\cl(\Ss)$ is proper and $g(\Qq)=\cl(\Ss)$. 
\item[\rm{(iv)}] A closed semialgebraic set $\Rr\subset\cl(\Ss)$ of dimension strictly smaller than $d$ such that $\cl(\Ss)\setminus\Rr$ and $\Qq\setminus g^{-1}(\Rr)$ are Nash manifolds and the polynomial map 
$$
g|_{\Qq\setminus g^{-1}(\Rr)}:\Qq\setminus g^{-1}(\Rr)\to\cl(\Ss)\setminus\Rr
$$ 
is a Nash diffeomorphism. 
\end{itemize}
In addition, as $X$ is compact, also $X'$ is compact.

Let $\Ss^*:=g^{-1}(\Ss)\cap\cl(g^{-1}(\Ss)\setminus\Rr)$ be the strict transform of $\Ss$ under $g$. We claim: {\em The semialgebraic sets $\Tt_1^*:=\cl(\Ss^*)\setminus\Ss^*$, $\Tt_2^*:=\Ss^*\setminus\Reg(\Ss^*)$, $\Tt_3^*:=\cl(\Ss^*)\setminus\Reg(\cl(\Ss^*))$ and $\Tt_4^*:=\Reg(\cl(\Ss^*))\setminus\Ss^*$ are unions of elements of the stratification ${\mathfrak G}(Z')$.} 

Using the properties of the drilling blow-up and especially Fact \ref{bigstepa6}, one deduces straightforwardly that the semialgebraic sets $\Tt_1^*$, $\Tt_2^*$, $\Tt_3^*$ and $\Tt_4^*$ are unions of elements of the stratification ${\mathfrak G}(Z')$. To that end, one can almost reproduce the procedure developed in {\sc Step 3} taking into account the particularities of the proof of Theorem \ref{red2}, when one applies it to $\cl(\Ss)$. 

\noindent{\sc Current situation (2).}
If $\Reg(\cl(\Ss^*))=\Reg(\Ss^*)$, then $\Tt_3^*=\cl(\Ss^*)\setminus\Reg(\Ss^*)=\partial\Ss^*$ and $\Tt_4^*=\varnothing$. Thus, $\Ss^*$ is a Nash quasi-manifold with corners and the proof is finished for this case. Consequently, to continue we suppose $\Reg(\cl(\Ss^*))\neq\Reg(\Ss^*)$. This means that $\Tt_4^*\neq\varnothing$, because otherwise $\Reg(\cl(\Ss^*))\subset\Ss^*$ and consequently $\Reg(\cl(\Ss^*))=\Reg(\Ss^*)$ (because $\Reg(\cl(\Ss^*))$ is an open semialgebraic subset of $X'$ contained in $\Ss^*$ that contains $\Reg(\Ss^*)$).

At this point we assume that the initial situation is the one quoted concerning $\Ss^*$ and $\cl(\Ss^*)$. For the sake of simplicity we reset all the previous notations above to continue the proof.

\noindent{\sc Step 5. Second drilling desingularization procedure.}
We assume in the following: {\em $\Ss$ is a checkerboard set (and in particular $\Reg(\Ss)$ is a connected Nash manifold), $\Qq:=\cl(\Ss)$ is a compact Nash manifold with corners, the Zariski closure $X$ of $\Ss$ is a non-singular (compact) real algebraic set, the Zariski closure $Z$ of $\Qq\setminus\Reg(\Ss)$ is a normal-crossings divisor of $X$, the semialgebraic sets $\Tt_1:=\Qq\setminus\Ss$, $\Tt_2:=\Ss\setminus\Reg(\Ss)$, $\Tt_3:=\Qq\setminus\Reg(\Qq)$ and $\Tt_4:=\Reg(\Qq)\setminus\Ss\neq\varnothing$ are unions of elements of the stratification ${\mathfrak G}(Z)$}. Let us prove: {\em We may assume in addition $\Reg(\Tt_4)$ is a pure dimensional semialgebraic set of dimension $d-1$, $\partial\Tt_4=\cl(\Tt_4)\setminus\Reg(\Tt_4)\subset\partial\Qq$ and $\Tt_4=\Reg(\Tt_4)$.}

As $\Tt_4$ is a union of elements of the stratification ${\mathfrak G}(Z)$, the semialgebraic set $\cl(\Tt_4)$ is a union of elements of the stratification ${\mathfrak G}(Z)$. If $\Tt_4$ has dimension $\leq d-2$, it is contained in $\Sing(Z)$ and $\varnothing\neq\Tt_4\subset\cl(\Tt_4)\cap\Reg(\Qq)$. If $\Tt_4$ has dimension $d-1$, define $\Mm_4$ as the union of the connected components of $\Reg(\Tt_4)$ of dimension $d-1$. Observe that both $\Mm_4$ and $\cl(\Tt_4)\setminus\Mm_4$ are unions of elements of the stratification ${\mathfrak G}(Z)$. Define 
$$
\Aa(\Ss):=\begin{cases}
\cl(\Tt_4)\cap\Reg(\Qq)&\text{if $\dim(\Tt_4)\leq d-2$,}\\
(\cl(\Tt_4)\setminus\Mm_4)\cap\Reg(\Qq)&\text{if $\dim(\Tt_4)=d-1$.}
\end{cases}
$$
Observe that $\dim(\Aa(\Ss))\leq d-2$ and $\Aa(\Ss)=\varnothing$ if and only if $\dim(\Tt_4)=d-1$ and $\cl(\Tt_4)\setminus\Mm_4\subset\partial\Qq$. As $\Tt_4=\Reg(\Qq)\setminus\Ss$, this means that $\Reg(\Tt_4)=\Mm_4$ is a pure dimensional semialgebraic set of dimension $d-1$. We next develop an inductive procedure to reduce to the latter case.
 
Let $Y$ be the Zariski closure of $\Aa(\Ss)$, which is a union of irreducible components of $\Sing_\ell(Z)$ for $\ell=1,\ldots,d-2$, maybe of different dimensions and denote $\Sing_0(Y):=Y$. Let $e$ be the dimension of $Y$ and $Y_k$ the union of $\Sing_{e-k}(Y)$, which is either the empty set or a real algebraic set of dimension $k$, and the irreducible components of $Y$ of dimension $k$ for $k=0,\ldots,e$. 

If $\Aa(\Ss)=\varnothing$, define $\ell(\Ss)=d-1$ and $m(\Ss)=0$. If $\Aa(\Ss)\neq\varnothing$, let $\ell:=\ell(\Ss)\leq e$ be the minimum value $k$ such that $Y_k\neq\varnothing$ and $m:=m(\Ss)$ the number of irreducible components of $Y_\ell$. Observe that $Y_\ell$ is a pure dimensional non-singular (compact) real algebraic set of dimension $\leq d-2$. We proceed by double induction on $\ell$ and $m$.

Let $W$ be an irreducible component of $Y_\ell$. Let $(\widehat{X},\widehat{\pi})$ be the twisted Nash double of the drilling blow-up $(\widetilde{X},\pi_+)$ of $X$ with center $W$, which is by \S\ref{addbu} a real algebraic set. Let 
$$
\Qq^*:=\pi_+^{-1}(\Qq)\cap\cl(\pi_+^{-1}(\Qq\setminus W))
$$
be the strict transform of $\Qq$ under $\pi_+$. As $\Qq$ is pure dimensional and $Y_\ell\subset Z$ has dimension strictly smaller, $\Qq\setminus W$ is dense in $\Qq$, so $\pi_+(\Qq^*)=\Qq$, because $\pi_+:\widetilde{X}\to X$ is proper and surjective. By Lemma \ref{exs2} $\Qq^*$ is a checkerboard set and a Nash manifold with corners such that $\pi_+^{-1}(W)\cap\Qq^*\subset\partial\Qq^*$. Let $\Ss^*:=\pi_+^{-1}(\Ss)\cap\cl(\pi_+^{-1}(\Ss\setminus W))$ be the strict transform of $\Ss^*$ under $\pi_+$, which keeps the same properties required to $\Ss$ (to check this fact one proceeds similarly as we have done in {\sc Steps 3} and 4). As $\pi_+^{-1}(W)\cap\Qq^*\subset\partial\Qq^*$, we deduce $\Aa(\Ss^*)=\pi_+^{-1}(\Aa(\Ss)\setminus W)$, so $m(\Ss^*)=m(\Ss)-1$ and 
$$
\ell(\Ss^*)\begin{cases}
>\ell(\Ss)&\text{if $m(\Ss)=1$,}\\
=\ell(\Ss)&\text{if $m(\Ss)>1$.}
\end{cases}
$$
The restriction $\pi_+|_{\Ss^*}:\Ss^*\to\Ss$ is a surjective proper polynomial map and if $\Rr:=\partial\Ss\cup W$, the restriction $\pi_+|_{\Ss^*\setminus\pi_+^{-1}(\Rr)}:\Ss^*\setminus\pi_+^{-1}(\Rr)\to\Ss\setminus\Rr$ is a Nash diffeomorphism. 

We proceed inductively and after finitely many steps we may assume $\Aa(\Ss)=\varnothing$. Under such assumption, we have: $\Reg(\Tt_4)=\Tt_4$. 

As $\Tt_4=\Reg(\Qq)\setminus\Ss\subset\Reg(\Qq)$ and $\partial\Tt_4=\cl(\Tt_4)\setminus\Reg(\Tt_4)\subset\partial\Qq$, we deduce $\Tt_4\cap\partial\Tt_4\subset\Reg(\Qq)\cap\partial\Qq=\varnothing$, so $\Tt_4=\Reg(\Tt_4)$.

\noindent{\sc Step 6. Final drilling desingularization procedure.} After resetting notations, we assume in the following: {\em $\Ss$ is a checkerboard set (and in particular $\Reg(\Ss)$ is a connected Nash manifold), $\Qq:=\cl(\Ss)$ is a compact Nash manifold with corners, the Zariski closure $X$ of $\Ss$ is a non-singular (compact) real algebraic set, the Zariski closure $Z$ of $\Qq\setminus\Reg(\Ss)$ is a normal-crossings divisor of $X$, the semialgebraic sets $\Tt_1:=\Qq\setminus\Ss$, $\Tt_2:=\Ss\setminus\Reg(\Ss)$, $\Tt_3:=\Qq\setminus\Reg(\Qq)$, and $\Tt_4:=\Reg(\Qq)\setminus\Ss\neq\varnothing$ are unions of elements of the stratification ${\mathfrak G}(Z)$. In addition, $\Tt_4=\Reg(\Tt_4)$ is a pure dimensional semialgebraic set of dimension $d-1$ and $\partial\Tt_4=\cl(\Tt_4)\setminus\Reg(\Tt_4)\subset\partial\Qq$.}

In order to finish the proof, we will take advantage of Fact \ref{locald-1}. Except for the initial embedding of $\R\PP^m$ in $\R^p$, which is a regular map, until this step all the involved maps are polynomial. As we will perform the drilling blow-up of a Nash submanifold of dimension $d-1$, we have to proceed carefully in order to not disconnect the regular locus of $\Ss$ (Example \ref{truffa}). Thus, in the following the involved maps are a priori only Nash maps.

Let $C_1,\ldots,C_m$ be the connected components of $\Tt_4$. Each intersection $\cl(C_i)\cap C_j=\varnothing$ for $i\neq j$, because $\cl(C_i)\cap C_j=\cl(C_i)\cap\Tt_4\cap C_j=C_i\cap C_j=\varnothing$ (we have used that the connected components of $\Tt_4$ are pairwise disjoint closed subsets of $\Tt_4$). As $\Reg(C_i)=C_i$, we have
\begin{equation}\label{dci}
\partial\Tt_4=\cl(\Tt_4)\setminus\Reg(\Tt_4)=\cl(\Tt_4)\setminus\Tt_4=\cl\Big(\bigcup_{i=1}^mC_i\Big)\setminus\bigcup_{i=1}^mC_i=\bigcup_{i=1}^m\cl(C_i)\setminus C_i=\bigcup_{i=1}^m\partial C_i.
\end{equation}
Consequently, $\partial C_i\subset\partial\Tt_4\subset\partial\Qq$ for $i=1,\ldots,m$. As $\Reg(\Tt_4)=\Tt_4$ is a union of elements of the stratification ${\mathfrak G}(Z)$, each connected component $C_i$ of $\Reg(\Tt_4)$ is a union of elements of the stratification ${\mathfrak G}(Z)$.

\noindent{\sc Substep 6.1.} We claim: {\em $\cl(C_i)\cap\cl(C_j)=\varnothing$ if $i\neq j$}.

Assume $\cl(C_1)\cap\cl(C_2)\neq\varnothing$. As $\cl(C_i)\cap C_j=\varnothing$ if $i\neq j$, we deduce 
\begin{equation}\label{clcicj}
\cl(C_1)\cap\cl(C_2)=(\cl(C_1)\setminus C_1)\cap(\cl(C_2)\setminus C_2)=\partial C_1\cap\partial C_2\subset\partial\Qq.
\end{equation}
Pick $x\in\cl(C_1)\cap\cl(C_2)$ and let $U\subset X$ be an open semialgebraic neighborhood of $x$ such that $Z\cap U=\{\x_1\cdots\x_r=0\}$ in (local) coordinates. As ${\mathfrak G}(Z)$ is compatible with $C_1,C_2$, which are connected components of the Nash manifold $\Tt_4=\Reg(\Tt_4)$, there exist indices $1\leq i,j\leq r$ such that $C_1\cap U\subset\{\x_i=0\}$ and $C_2\cap U\subset\{\x_j=0\}$. If $i=j$, we suppose $\Qq\cap U=\{\x_1\geq0,\ldots,\x_s\geq0\}$ for some $1\leq s\leq r-1$, $C_1\cap U\subset\{\x_r=0\}$ and $C_2\cap U\subset\{\x_r=0\}$. As ${\mathfrak G}(Z)$ is compatible with $C_1$ and $C_2$, we may assume 
\begin{align*}
\{\x_1\geq0,\ldots,\x_s\geq0,\x_{s+1}\geq0,\ldots,\x_{r-2}\geq0,\x_{r-1}\geq0,\x_r=0\}&\subset\cl(C_1)\cap U,\\
\{\x_1\geq0,\ldots,\x_s\geq0,\x_{s+1}*_{s+1}0,\ldots,\x_{r-2}*_{r-2}0,\x_{r-1}*_{r-1}0,\x_r=0\}&\subset\cl(C_2)\cap U,
\end{align*}
where $*_j\in\{\geq,\leq\}$ for $j=s+1,\ldots,r-1$. By \eqref{clcicj}
\begin{multline*}
\{\x_1\geq0,\ldots,\x_s\geq0,\x_{s+1}=0,\ldots,\x_r=0\}\subset\cl(C_1)\cap\cl(C_2)\cap U\\
\subset\partial\Qq\cap U=\bigcup_{i=1}^s\{\x_1\geq0,\ldots,\x_s\geq0,\x_i=0\},
\end{multline*}
which is a contradiction. If $i\neq j$, we suppose $\Qq\cap U=\{\x_1\geq0,\ldots,\x_s\geq0\}$ for some $1\leq s\leq r-2$, $C_1\cap U\subset\{\x_{r-1}=0\}$ and $C_2\cap U\subset\{\x_r=0\}$. As ${\mathfrak G}(Z)$ is compatible with $C_1$ and $C_2$, we may assume 
\begin{align*}
\{\x_1\geq0,\ldots,\x_s\geq0,\x_{s+1}\geq0,\ldots,\x_{r-2}\geq0,\x_{r-1}=0,\x_r\geq0\}&\subset\cl(C_1)\cap U,\\
\{\x_1\geq0,\ldots,\x_s\geq0,\x_{s+1}*_{s+1}0,\ldots,\x_{r-2}*_{r-2}0,\x_{r-1}\geq0,\x_r=0\}&\subset\cl(C_2)\cap U,
\end{align*}
where $*_j\in\{\geq,\leq\}$ for $j=s+1,\ldots,r-2$. By \eqref{clcicj}
\begin{multline*}
\{\x_1\geq0,\ldots,\x_s\geq0,\x_{s+1}=0,\ldots,\x_r=0\}\subset\cl(C_1)\cap\cl(C_2)\cap U\\
\subset\partial\Qq\cap U=\bigcup_{i=1}^s\{\x_1\geq0,\ldots,\x_s\geq0,\x_i=0\},
\end{multline*}
which is a contradiction. Consequently, $\cl(C_1)\cap\cl(C_2)=\varnothing$, as claimed.

\noindent{\sc Substep 6.2.} Let $\Gamma$ be a stratum of ${\mathfrak G}(Z)$ contained in $\Qq$ such that $\Gamma$ is not contained in $\cl(C_i)$. We prove next: {\em If the Zariski closure of $\Gamma$ is contained in the Zariski closure of $C_i$, then $\cl(\Gamma)\cap\cl(C_i)=\varnothing$}.

As $\Gamma$ is not contained in $\cl(C_i)$ and the stratification ${\mathfrak G}(Z)$ is compatible with $C_i$, we have $\cl(\Gamma)\cap C_i=\varnothing$, so 
$$
\cl(\Gamma)\cap\cl(C_i)=\cl(\Gamma)\cap(\cl(C_i)\setminus C_i)=\cl(\Gamma)\cap\partial C_i.
$$
Suppose $\cl(\Gamma)\cap\cl(C_i)\neq\varnothing$, pick $x\in\cl(\Gamma)\cap\cl(C_i)$ and let $U\subset X$ be an open semialgebraic neighborhood of $x$ such that $Z\cap U=\{\x_1\cdots\x_r=0\}$ in (local) coordinates. We may assume $\Qq\cap U=\{\x_1\geq0,\ldots,\x_s\geq0\}$ for some $1\leq s\leq r-1$ and $C_i\cap U\subset\{\x_r=0\}$. As the Zariski closure of $\Gamma$ is contained in the Zariski closure of $C_i$ and $x\in\cl(\Gamma)\cap\cl(C_i)\cap U$, we deduce $\cl(\Gamma)\cap U\subset\Qq\cap U\cap\{\x_r=0\}$. In addition, $\partial C_i\cap U\subset\partial\Qq\cap U=\bigcup_{i=1}^s\{\x_1\geq0,\ldots,\x_s\geq0,\x_i=0\}$. As ${\mathfrak G}(Z)$ is compatible with $C_i$ and $\partial C_i\subset\partial\Qq$ (see \eqref{dci}), we conclude
\begin{align*}
C_i\cap U&=\{\x_1>0,\ldots,\x_s>0,\x_r=0\}=\Reg(\Qq)\cap U\cap\{\x_r=0\},\\
\cl(C_i)\cap U&=\{\x_1\geq0,\ldots,\x_s\geq0,\x_r=0\}=\Qq\cap U\cap\{\x_r=0\},
\end{align*}
so $\cl(\Gamma)\cap U\subset\Qq\cap U\cap\{\x_r=0\}=\cl(C_i)\cap U$, which is a contradiction. Thus, $\cl(\Gamma)\cap\cl(C_i)=\varnothing$.

\noindent{\sc Substep 6.3.} Recall that $\dim(C_i)=d-1$ for each $i=1,\ldots,m$. The Zariski closure of $C_i$ is the irreducible component $Z_i$ of $Z$ that contains $C_i$. The semialgebraic set $Z_i\cap\Qq\setminus\cl(C_i)$ is a union of elements of ${\mathfrak G}(Z)$ and it is closed, because otherwise there exists a stratum $\Gamma$ of ${\mathfrak G}(Z)$ contained in $\Qq$ such that $\Gamma\not\subset\cl(C_i)$ but $\cl(\Gamma)\cap\cl(C_i)\neq\varnothing$, which is a contradiction by {\sc Substep 6.2}. Consider the closed semialgebraic set $K_i:=(Z_i\cap\Qq\setminus\cl(C_i))\cup\bigcup_{j\neq i}\cl(C_j)$ and observe that $K_i\cap\cl(C_i)=\varnothing$. As both semialgebraic sets are compact (recall that the Zariski closure of $\Ss$ is compact) and disjoint,
$$
\veps:=\frac{1}{2}\min\{\dist(K_i,\cl(C_i)):\ i=1,\ldots,m\}>0.
$$ 
Define $U_i:=\{x\in Z_i:\ \dist(x,\cl(C_i))<\veps\}$, which is an open semialgebraic neighborhood of $\cl(C_i)$ in $Z_i$. We claim: {\em 
\begin{itemize}
\item[(1)] The union $N:=\bigcup_{i=1}^mU_i$ is a closed Nash submanifold of the Nash manifold $M_0:=X\setminus\bigcup_{i=1}(\cl(U_i)\setminus U_i)$, 
\item[(2)] $\Qq\subset M_0$ and 
\item[(3)] $\Qq\cap N=\bigcup_{i=1}^m\cl(C_i)=\cl(\Tt_4)$.
\end{itemize}
}

We prove first (1). It is clear that $N$ is a closed subset of $M_0$. As each $U_i$ is an open semialgebraic subset of the Nash manifold $Z_i$, to prove that $N\subset M_0$ is a Nash manifold, it is enough to show that $\cl(U_i)\cap\cl(U_j)=\varnothing$ if $i\neq j$. 
If there exists $x\in\cl(U_i)\cap\cl(U_j)$, then 
\begin{multline*}
\dist(\cl(C_i),\cl(C_j))\leq\dist(x,\cl(C_i))+\dist(x,\cl(C_j))\\
<2\veps\leq\dist(K_i,\cl(C_i))\leq\dist(\cl(C_i),\cl(C_j)),
\end{multline*}
which is a contradiction. Consequently, the semialgebraic sets $\cl(U_i)$ for $i=1,\ldots,m$ are pairwise disjoint and $N$ is a closed Nash submanifold of $M_0$.

We check next (2): $\Qq\subset M_0$.

Suppose there exists $x\in\Qq\cap(\cl(U_i)\setminus U_i)$, so 
$$
x\in Z_i\cap\Qq\cap(\cl(U_i)\setminus U_i)\subset((Z_i\cap\Qq)\setminus\cl(C_i))\cap\cl(U_i),
$$ 
because $\cl(C_i)\subset U_i$. As $x\in\cl(U_i)$ and $x\in(Z_i\cap\Qq)\setminus\cl(C_i)$, we have 
\begin{equation*}
\dist(x,\cl(C_i))\leq\veps\leq\frac{1}{2}\dist(K_i,\cl(C_i))\leq\frac{1}{2}\dist(Z_i\cap\Qq\setminus\cl(C_i),\cl(C_i))\leq\frac{1}{2}\dist(x,\cl(C_i)),
\end{equation*}
which is a contradiction. Consequently, $\Qq\cap(\cl(U_i)\setminus U_i)=\varnothing$ for each $i=1,\ldots,m$. Thus, $\Qq\subset X\setminus\bigcup_{i=1}(\cl(U_i)\setminus U_i)=M_0$.

As $\cl(\Tt_4)=\bigcup_{i=1}^m\cl(C_i)\subset\Qq\cap\bigcup_{i=1}^mU_i=\Qq\cap N$, to prove (3): $\Qq\cap N=\cl(\Tt_4)$, it is enough to check: $\Qq\cap U_i\subset\cl(C_i)$ or, equivalently, $\Qq\cap(U_i\setminus\cl(C_i))=\varnothing$ for $i=1,\ldots,m$. 

If $x\in\Qq\cap(U_i\setminus\cl(C_i))\subset Z_i$, then $x\in U_i$ and $x\in Z_i\cap\Qq\setminus\cl(C_i)$, which is a contradiction as we have seen when proving (2). Thus, $\Qq\cap U_i\subset\cl(C_i)$ for $i=1,\ldots,m$.

\noindent{\sc Substep 6.4.} As $\partial\Tt_4\subset\partial\Qq$, we have $\cl(\Tt_4)\cap\Reg(\Qq)=\Reg(\Tt_4)\cap\Reg(\Qq)$, so $\cl(\Tt_4)\cap\Reg(\Qq)=\Tt_4\cap\Reg(\Qq)$. As $\Tt_4=\Reg(\Qq)\setminus\Ss$, the difference
$$
\Reg(\Qq)\setminus\cl(\Tt_4)=\Reg(\Qq)\setminus\Tt_4=\Ss\cap\Reg(\Qq)\subset\Ss
$$
is an open semialgebraic subset of $X$ contained in $\Ss$, so $\Ss\cap\Reg(\Qq)\subset\Reg(\Ss)$. As $\Reg(\Ss)\subset\Ss\cap\Reg(\Qq)$, we conclude $\Reg(\Ss)=\Reg(\Qq)\setminus\cl(\Tt_4)=\Reg(\Qq)\setminus(\Qq\cap N)=\Reg(\Qq)\setminus N$. In general, $N$ is not a real algebraic set and its Zariski closure $Y$ is not an option because $\Reg(\Qq)\setminus Y$ might be disconnected (Example \ref{truffa}). Thus, the following drilling blow-up is in general only a Nash map. Let $(\widehat{M},\widehat{\pi})$ be the twisted Nash double of the drilling blow-up $(\widetilde{M},\pi_+)$ of $M:=X$ with center the closed Nash submanifold $N$ of $M$, which is by \S\ref{gdndbu} a Nash manifold. As $\widehat{\pi}:\widehat{M}\to M$ is proper and surjective and $M$ is compact, also $\widehat{M}$ is compact. We have denoted $X$ by $M$ in order to stress that $\widehat{M}$ is a compact Nash manifold, which is not in general a non-singular real algebraic set (but only one of its compact connected components). Let $\Qq^\bullet:=\pi_+^{-1}(\Qq)\cap\cl(\pi_+^{-1}(\Qq\setminus N))$ be the strict transform of $\Qq$ under $\pi_+$. As $\Qq$ is pure dimensional and $N\subset Z$ has dimension strictly smaller, $\Qq\setminus N$ is dense in $\Qq$, so $\pi_+(\Qq^\bullet)=\Qq$, because $\pi_+:\widetilde{M}\to M$ is proper and surjective. By Facts \ref{bigstepa6} and \ref{locald-1} $\Qq^\bullet$ is a Nash manifold with corners such that $\pi_+^{-1}(N)\cap\Qq^\bullet\subset\partial\Qq^\bullet$. Observe that $\Qq^\bullet\setminus\pi_+^{-1}(N)=\pi_+^{-1}(\Qq\setminus N)$ is Nash diffeomorphic to $\Qq\setminus N$. Thus, 
\begin{equation*}
\Sth(\Qq^\bullet)=\Sth(\pi_+^{-1}(\Qq\setminus N))=\pi_+^{-1}(\Reg(\Qq\setminus N))=\pi_+^{-1}(\Reg(\Qq)\setminus N)=\pi_+^{-1}(\Reg(\Ss)),
\end{equation*} 
so $\Sth(\Qq^\bullet)$ is connected, because $\pi_+|_{\widetilde{M}\setminus\pi_+^{-1}(N)}:\widetilde{M}\setminus\pi_+^{-1}(N)\to M\setminus N$ is a Nash diffeomorphism.
Let $\Ss^\bullet:=\pi_+^{-1}(\Ss)\cap\cl(\pi_+^{-1}(\Ss\setminus N))$ be the strict transform of $\Ss$ under $\pi_+$, which keeps the same properties required to $\Ss$ if one changes the operator $\Reg(\cdot)$ by the operator $\Sth(\cdot)$ in each case. To check this fact one proceeds similarly as we have done in {\sc Steps 3}, 4 and 5. In addition, 
$$
\Sth(\Ss^\bullet)\subset\Sth(\Qq^\bullet)=\pi_+^{-1}(\Reg(\Ss))\subset\Sth(\Ss^\bullet)
$$ 
(because $\Reg(\Ss)\subset M\setminus N$), so $\Sth(\Ss^\bullet)=\Sth(\Qq^\bullet)$. Observe that $\pi_+|_{\Ss^\bullet}:\Ss^\bullet\to\Ss$ is a surjective proper Nash map and if $\Rr:=\Ss\setminus\Reg(\Ss)$, the restriction $\pi_+|_{\Ss^\bullet\setminus\pi_+^{-1}(\Rr)}:\Ss^\bullet\setminus\pi_+^{-1}(\Rr)\to\Ss\setminus\Rr$ is a Nash diffeomorphism.

\noindent{\sc Substep 6.5.} To finish we shall `algebrize' our construction as much as possible. Recall that by \cite[Thm.1.1]{ak} the pair constituted by a compact Nash manifold and a Nash normal-crossings divisor is diffeomorphic to a pair constituted by a non-singular (compact) real algebraic set and a normal-crossings divisor and the previous diffeomorphism preserves Nash irreducible components of the corresponding Nash normal-crossings divisor. By the proof of the approximation results \cite[Thm.1.7 \& Prop.8.2]{bfr} modified to fit our situation, we may assume that the previous diffeomorphism is in addition a Nash diffeomorphism. To that end, we have to substitute Efroymson's approximation result \cite[Thm.II.4.1]{sh} for differentiable semialgebraic functions on a Nash manifold by Nash functions by Stone-Weierstrass approximation for differentiable functions on differentiable manifolds by polynomial functions (see also \cite[Thm.4.5 \& Appendix B]{gs}). 

Using the previous fact and \cite[Lem.8.3 \& Lem.C.1]{fe3}, we may assume in addition (using a suitable Nash embedding of $\widetilde{M}$ in some affine space) that the Nash quasi-manifold with corners $\Ss^\bullet$ is a checkerboard set, the Nash manifold with corners $\Qq^\bullet=\cl(\Ss^\bullet)$ is a checkerboard set, the Zariski closure $X^\bullet$ of $\Ss^\bullet$ is a connected compact irreducible non-singular real algebraic set, the Zariski closure $Z^\bullet$ of $\partial\Ss^\bullet=\Qq^\bullet\setminus\Reg(\Ss^\bullet)=\Qq^\bullet\setminus\Reg(\Qq^\bullet)=\partial\Qq^\bullet$ is a normal-crossings divisor of $X^\bullet$ and the stratification ${\mathfrak G}(Z^\bullet)$ is compatible with $\Ss^\bullet\setminus\Reg(\Ss^\bullet)$. 
\end{proof}

\begin{figure}[!ht]
\begin{center}
\begin{tikzpicture}[scale=0.75]
\draw[fill=gray!50,opacity=0.4,draw=none] (0,0) -- (0,5) -- (5,5) -- (5,0) -- (0,0);
\draw[fill=white,draw=none] (1,2.5) -- (1,4) -- (2.5,4) -- (2.5,2.5) -- (1,2.5);
\draw[fill=white,draw=none] (2.5,1) -- (4,1) -- (4,2.5) -- (2.5,2.5) -- (2.5,1);

\draw[line width=1pt, dotted] (0,0) -- (0,5);
\draw[line width=1pt] (0,0) -- (2.5,0);
\draw[line width=1pt, dotted] (2.5,0) -- (5,0);

\draw[line width=1pt] (0,5) -- (1,5);
\draw[line width=1pt, dotted] (1,5) -- (3,5);
\draw[line width=1pt] (3,5) -- (5,5);

\draw[line width=1pt] (5,0) -- (5,5);

\draw[line width=1pt] (1,2.5) -- (1,4);
\draw[line width=1pt] (1,4)-- (2.5,4);
\draw[line width=1pt, dotted] (2.5,4) -- (2.5,2.5);
\draw[line width=1pt] (2.5,2.5) -- (1,2.5);

\draw[line width=1pt, dotted] (4,1) -- (2.5,1);
\draw[line width=1pt] (2.5,1) -- (2.5,2.5);
\draw[line width=1pt] (2.5,2.5) -- (4,2.5);
\draw[line width=1pt] (4,1) -- (4,2.5);

\draw[fill=black,draw] (0,0) circle (0.5mm);
\draw[fill=white,draw] (5,0) circle (0.5mm);
\draw[fill=white,draw] (0,5) circle (0.5mm);
\draw[fill=black,draw] (5,5) circle (0.5mm);

\draw[fill=white,draw] (1,5) circle (0.5mm);
\draw[fill=black,draw] (3,5) circle (0.5mm);
\draw[fill=white,draw] (2.5,0) circle (0.5mm);

\draw[fill=white,draw] (1,4) circle (0.5mm);
\draw[fill=black,draw] (1,2.5) circle (0.5mm);
\draw[fill=black,draw] (2.5,2.5) circle (0.5mm);
\draw[fill=white,draw] (2.5,1) circle (0.5mm);
\draw[fill=white,draw] (2.5,4) circle (0.5mm);
\draw[fill=white,draw] (4,1) circle (0.5mm);
\draw[fill=black,draw] (4,2.5) circle (0.5mm);

\filldraw[color=white, fill=white] (3,5) circle (0.2);
\draw[line width=1pt] (2.8,5) arc (180:360:0.2cm);
\draw[fill=white,draw] (2.8,5) circle (0.5mm);
\draw[fill=black,draw] (3.2,5) circle (0.5mm);

\filldraw[color=white, fill=white] (1,2.5) circle (0.2);
\draw[line width=1pt] (1,2.7) arc (90:360:0.2cm);
\draw[fill=black,draw] (1,2.7) circle (0.5mm);
\draw[fill=black,draw] (1.2,2.5) circle (0.5mm);

\filldraw[color=white, fill=white] (4,2.5) circle (0.2);
\draw[line width=1pt] (4,2.3) arc (-90:180:0.2cm);
\draw[fill=black,draw] (4,2.3) circle (0.5mm);
\draw[fill=black,draw] (3.8,2.5) circle (0.5mm);

\filldraw[color=white, fill=white] (2.5,2.5) circle (0.2);
\draw[line width=1pt] (2.7,2.5) arc (0:90:0.2cm);
\draw[line width=1pt] (2.3,2.5) arc (180:270:0.2cm);
\draw[fill=black,draw] (2.7,2.5) circle (0.5mm);
\draw[fill=white,draw] (2.5,2.7) circle (0.5mm);
\draw[fill=black,draw] (2.5,2.3) circle (0.5mm);
\draw[fill=black,draw] (2.3,2.5) circle (0.5mm);

\filldraw[color=white, fill=white] (1,4) circle (0.2);
\draw[ dotted, line width=1pt] (1.2,4) arc (0:270:0.2cm);
\draw[fill=black,draw] (1.2,4) circle (0.5mm);
\draw[fill=black,draw] (1,3.8) circle (0.5mm);

\filldraw[color=white, fill=white] (2.5,4) circle (0.2);
\draw[ dotted, line width=1pt] (2.5,3.8) arc (-90:180:0.2cm);
\draw[fill=white,draw] (2.5,3.8) circle (0.5mm);
\draw[fill=black,draw] (2.3,4) circle (0.5mm);

\filldraw[color=white, fill=white] (2.5,1) circle (0.2);
\draw[ dotted, line width=1pt] (2.5,1.2) arc (90:360:0.2cm);
\draw[fill=black,draw] (2.5,1.2) circle (0.5mm);
\draw[fill=white,draw] (2.7,1) circle (0.5mm);

\filldraw[color=white, fill=white] (4,1) circle (0.2);
\draw[ dotted, line width=1pt] (3.8,1) arc (-180:90:0.2cm);
\draw[fill=black,draw] (4,1.2) circle (0.5mm);
\draw[fill=white,draw] (3.8,1) circle (0.5mm);

\filldraw[color=white, fill=white] (2.5,0) circle (0.2);
\draw[ dotted, line width=1pt] (2.7,0) arc (0:180:0.2cm);
\draw[fill=black,draw] (2.3,0) circle (0.5mm);
\draw[fill=white,draw] (2.7,0) circle (0.5mm);

\filldraw[color=white, fill=white] (1,5) circle (0.2);
\draw[ dotted, line width=1pt] (0.8,5) arc (180:360:0.2cm);
\draw[fill=black,draw] (0.8,5) circle (0.5mm);
\draw[fill=white,draw] (1.2,5) circle (0.5mm);

\draw[fill=gray!50,opacity=0.4,draw=none] (8,0) -- (8,5) -- (13,5) -- (13,0) -- (8,0);
\draw[fill=white,draw=none] (9,2.5) -- (9,4) -- (10.5,4) -- (10.5,2.5) -- (9,2.5);
\draw[fill=white,draw=none] (10.5,1) -- (12,1) -- (12,2.5) -- (10.5,2.5) -- (10.5,1);

\draw[line width=1pt, dotted] (8,0) -- (8,5);
\draw[line width=1pt] (8,0) -- (10.5,0);
\draw[line width=1pt, dotted] (10.5,0) -- (13,0);

\draw[line width=1pt] (8,5) -- (9,5);
\draw[line width=1pt, dotted] (9,5) -- (11,5);
\draw[line width=1pt] (11,5) -- (13,5);
\draw[line width=1pt] (13,0) -- (13,5);

\draw[line width=1pt] (9,2.5) -- (9,4);
\draw[line width=1pt] (9,4)-- (10.5,4);
\draw[line width=1pt, dotted] (10.5,4) -- (10.5,2.5);
\draw[line width=1pt] (10.5,2.5) -- (9,2.5);

\draw[line width=1pt, dotted] (12,1) -- (10.5,1);
\draw[line width=1pt] (10.5,1) -- (10.5,2.5);
\draw[line width=1pt] (10.5,2.5) -- (12,2.5);
\draw[line width=1pt] (12,1) -- (12,2.5);

\draw[fill=black,draw] (8,0) circle (0.5mm);
\draw[fill=white,draw] (13,0) circle (0.5mm);
\draw[fill=white,draw] (8,5) circle (0.5mm);
\draw[fill=black,draw] (13,5) circle (0.5mm);

\draw[fill=white,draw] (9,5) circle (0.5mm);
\draw[fill=black,draw] (11,5) circle (0.5mm);
\draw[fill=white,draw] (10.5,0) circle (0.5mm);

\draw[fill=white,draw] (9,4) circle (0.5mm);
\draw[fill=black,draw] (9,2.5) circle (0.5mm);
\draw[fill=black,draw] (10.5,2.5) circle (0.5mm);
\draw[fill=white,draw] (10.5,1) circle (0.5mm);
\draw[fill=white,draw] (10.5,4) circle (0.5mm);
\draw[fill=white,draw] (12,1) circle (0.5mm);
\draw[fill=black,draw] (12,2.5) circle (0.5mm);

\draw (3.5,3.5) node{$\Ss^\bullet$};
\draw (11.5,3.5) node{$\Ss$};
\draw[->] (5.5 ,2.5)--(7.5, 2.5);
\draw (6.5,2.9) node{\small $f|_{\Ss^\bullet}$};
\end{tikzpicture}
\end{center}

\caption{\small{Nash uniformization of the checkerboard set $\Ss$ (right) by the Nash quasi-manifold with corners $\Ss^\bullet$ (left).}}
\end{figure}

\vspace{3mm}
\begin{example}\label{truffa}
Let $X:=\{\x_1^2+\cdots+\x_n^2=1\}\subset\R^n$ and $\Ss:=X\cap\{\x_n^2\leq\tfrac{1}{4}\}\setminus\{\x_{n-2}\leq0,\x_{n-1}=0\}$, which is a checkerboard set whose Zariski closure is $X$. The real algebraic set $X$ is the $(n-1)$-dimensional unit sphere, so it is compact and non-singular. The closure $\cl(\Ss)=X\cap\{\x_n^2\leq\frac{1}{4}\}$ is a compact Nash manifold with corners. Observe that $\Reg(\Ss)=\Ss\cap\{\x_n^2<\frac{1}{4}\}$, so 
$$
\cl(\Ss)\setminus\Reg(\Ss)=(X\cap\{\x_n^2=\tfrac{1}{4}\})\cup(X\cap\{\x_{n-2}\leq0,\x_{n-1}=0\}\cap\{\x_n^2\leq\tfrac{1}{4}\}).
$$
The Zariski closure of $\cl(\Ss)\setminus\Reg(\Ss)$ is
$$
Z:=(X\cap\{\x_n=\tfrac{1}{2}\})\cup(X\cap\{\x_n=-\tfrac{1}{2}\})\cup(X\cap\{\x_{n-1}=0\}),
$$
which is a normal-crossings divisor of $X$. Denote $\Qq:=\cl(\Ss)$. The semialgebraic sets
\begin{align*}
\Tt_1:&=\Qq\setminus\Ss=X\cap\{\x_{n-2}\leq0,\x_{n-1}=0\}\cap\{\x_n^2\leq\tfrac{1}{4}\},\\
\Tt_2:&=\Ss\setminus\Reg(\Ss)=X\cap\{\x_n^2=\tfrac{1}{4}\}\setminus\{\x_{n-2}\leq0,\x_{n-1}=0\},\\
\Tt_3:&=\Qq\setminus\Reg(\Qq)=X\cap\{\x_n^2=\tfrac{1}{4}\},\\
\Tt_4:&=\Reg(\Qq)\setminus\Ss=X\cap\{\x_{n-2}\leq0,\x_{n-1}=0\}\setminus\{\x_n^2=\tfrac{1}{4}\}\neq\varnothing
\end{align*}
are unions of elements of the stratification ${\mathfrak G}(Z)$. In addition, $\Reg(\Tt_4)$ is a pure dimensional semialgebraic set of dimension $d-1$ and
$$
\partial\Tt_4=\cl(\Tt_4)\setminus\Reg(\Tt_4)=X\cap\{\x_n^2=\tfrac{1}{4}\}\cap\{\x_{n-2}\leq0,\x_{n-1}=0\}\subset\partial\Qq.
$$
Thus, we are under the hypothesis of {\sc Step 6} of the Proof of Theorem \ref{red4}. We consider as $N:=X\cap\{\x_{n-2}<0,\x_{n-1}=0\}$. If we consider the Zariski closure $Y:=X\cap\{\x_{n-1}=0\}$ of $N$, we have that $\Reg(\Ss)\setminus Y$ has two connected components, which are $X\cap\{\x_{n-1}>0,\x_n^2<\frac{1}{4}\}$ and $X\cap\{\x_{n-1}<0,\x_n^2<\frac{1}{4}\}$. This means that we cannot take $Y$ instead of $N$ to perform the drilling blow-up of $X$ with center $Y$, because $\Reg(\Ss)\setminus Y$ is not connected, whereas $\Reg(\Qq)\setminus N=\Reg(\Ss)\setminus N=X\cap\{\x_n^2<\frac{1}{4}\}\setminus\{\x_{n-2}\leq0,\x_{n-1}=0\}=\Reg(\Ss)$ is connected. Consequently, when one applies the procedure of {\sc Step 6} of the Proof of Theorem \ref{red4}, the reasonable choice for the center of the drilling blow-up is $N$.
\hfill$\sqbullet$
\end{example}

\subsection{Application 2: Nash uniformization of general semialgebraic sets}
We prove next Corollary \ref{amal}, which is the combination of Bierstone-Parusi\'nski's desingularization of semialgebraic sets \cite[Thm.1.1., Rmks.2.3 \& 2.6]{bp} with Theorem \ref{red4}. We take advantage once more of Theorem \ref{ridwell2} (see \cite[Thm.8.4]{fe3}), whose proof uses implicitly the same strategy as \cite[Thm.1.1., Rmks.2.3 \& 2.6]{bp}. We also recall the concept of {\em bricks of a semialgebraic set} proposed in \cite[\S3]{fe1}. For each non-empty semialgebraic set $\Ss\subset\R^n$ there exists a unique finite family of non-empty pure dimensional semialgebraic sets $\Bb_1,\ldots,\Bb_s$ of dimensions $d_1>\cdots>d_s$ such that $\Bb_i$ is the closure in $\Ss$ of the set of points of dimension $d_i$ of $\Ss$. We have $\Ss=\bigcup_{i=1}^s\Bb_i$ and $\Bb_i\setminus\bigcup_{j\neq i}\Bb_j$ is a dense semialgebraic subset of $\Bb_i$ for $i=1,\ldots,s$. 

\begin{proof}[Proof of Corollary {\em\ref{amal}}]
Let $\Bb_1,\ldots,\Bb_s$ be the bricks of the semialgebraic set $\Ss$. Define $\Bb:=\bigsqcup_{k=1}^s(\Bb_k\times\{k\})$, whose bricks $\Bb_k\times\{k\}$ are pairwise disjoint, and consider the projection $\pi:\R^n\times\R\to\R^n$ onto the first factor. For each $t\in\R$ the restriction $\pi|_{\R^n\times\{t\}}:\R^n\times\{t\}\to\R^n$ is an affine isomorphism, and in particular $\pi|_{\Bb}:\Bb\to\Ss$ is a surjective proper Nash map and $\pi|_{\Bb_k\times\{k\}}:\Bb_k\times\{k\}\to\Bb_k\subset\Ss$ is a surjective proper Nash map and a Nash diffeomorphism for each $k=1,\ldots,s$. Thus, it is enough to prove Corollary \ref{amal} for each $\Bb_k\times\{k\}$, so we may assume from the beginning that $\Ss$ is pure dimensional. By Theorem \ref{ridwell2} we may assume that $\Ss$ is a checkerboard set. Now we apply either Theorem \ref{red2} (if $\Ss$ is in addition closed) or Theorem \ref{red4} (otherwise) to prove the statement.
\end{proof}






\bibliographystyle{amsalpha}

\end{document}